\documentclass[]{amsart}
\usepackage[all]{xy}
\usepackage{amscd,amsthm,amssymb,amsfonts,amsmath,euscript,mathtools}

\theoremstyle{plain}
\newtheorem{thm}{Theorem}[section]
\newtheorem{lemma}[thm]{Lemma}
\newtheorem{prop}[thm]{Proposition}
\newtheorem{cor}[thm]{Corollary}

\theoremstyle{definition}
\newtheorem{question}[thm]{Question}
\theoremstyle{remark}
\newtheorem{remark}[thm]{Remark}

\newcommand{\nc}{\newcommand}

\def\makeop#1{\expandafter\def\csname#1\endcsname
  {\mathop{\rm #1}\nolimits}\ignorespaces}
\makeop{Hom}   \makeop{End}   \makeop{Aut}   \makeop{Isom}  \makeop{Pic} 
\makeop{Gal}   \makeop{ord}   \makeop{Char}  \makeop{Div}   \makeop{Lie} 
\makeop{PGL}   \makeop{Corr}  \makeop{PSL}   \makeop{sgn}   \makeop{Spf}
\makeop{Spec}  \makeop{Tr}    \makeop{Nr}    \makeop{Fr}    \makeop{disc}
\makeop{Proj}  \makeop{supp}  \makeop{ker}   \makeop{im}    \makeop{dom}
\makeop{coker} \makeop{Stab}  \makeop{SO}    \makeop{SL}    \makeop{SL}
\makeop{Cl}    \makeop{cond}  \makeop{Br}    \makeop{inv}   \makeop{rank}
\makeop{id}    \makeop{Fil}   \makeop{Frac}  \makeop{GL}    \makeop{SU}
\makeop{Nrd}   \makeop{Sp}    \makeop{Tr}    \makeop{Trd}   \makeop{diag}
\makeop{Res}   \makeop{ind}   \makeop{depth} \makeop{Tr}    \makeop{st}
\makeop{Ad}    \makeop{Int}   \makeop{tr}    \makeop{Sym}   \makeop{can}
\makeop{length}\makeop{SO}    \makeop{torsion} \makeop{GSp} \makeop{Ker}
\makeop{Adm}   \makeop{Mat}
\def\makebb#1{\expandafter\def
  \csname bb#1\endcsname{{\mathbb{#1}}}\ignorespaces}
\def\makebf#1{\expandafter\def\csname bf#1\endcsname{{\bf
      #1}}\ignorespaces} 
\def\makegr#1{\expandafter\def
  \csname gr#1\endcsname{{\mathfrak{#1}}}\ignorespaces}
\def\makescr#1{\expandafter\def
  \csname scr#1\endcsname{{\EuScript{#1}}}\ignorespaces}
\def\makecal#1{\expandafter\def\csname cal#1\endcsname{{\mathcal
      #1}}\ignorespaces} 

\def\doLetters#1{#1A #1B #1C #1D #1E #1F #1G #1H #1I #1J #1K #1L #1M
                 #1N #1O #1P #1Q #1R #1S #1T #1U #1V #1W #1X #1Y #1Z}
\def\doletters#1{#1a #1b #1c #1d #1e #1f #1g #1h #1i #1j #1k #1l #1m
                 #1n #1o #1p #1q #1r #1s #1t #1u #1v #1w #1x #1y #1z}
\doLetters\makebb   \doLetters\makecal  \doLetters\makebf
\doLetters\makescr 
\doletters\makebf   \doLetters\makegr   \doletters\makegr
     \def\qed{\qedmark\medbreak}%
\def\qedmark{{\enspace\vrule height 6pt width 5pt depth 1.5pt}}%
    
\def\Gm{{{\bbG}_{\rm m}}}   

\normalsize

\makeop{Bl}

\def\sfF{\mathsf{F}}
\def\sfV{\mathsf{V}}

\def\Spec{{\rm Spec}\,}

\def\Fpbar{\overline{\bbF}_p}
\def\Fp{{\bbF}_p}

\def\Qpbar{\overline{{\bbQ}_p}}
\def\Qp{{\bbQ}_p}

\def\Zp{{\bbZ}_p}
\def\Qbar{\overline{\bbQ}}
\def\Sh{{\rm Sh}}

\newcommand{\Z}{\mathbb Z}
\newcommand{\Q}{\mathbb Q}
\newcommand{\R}{\mathbb R}
\newcommand{\C}{\mathbb C}
\renewcommand{\H}{\mathbb H}  
\newcommand{\A}{\mathbb A}    
\newcommand{\F}{\mathbb F}

\newcommand{\<}{\langle}   
\renewcommand{\>}{\rangle} 

\newcommand{\isoto}{\stackrel{\sim}{\longrightarrow}}
\nc{\embed}{\hookrightarrow}

\newcommand{\ch}{characteristic }
\newcommand{\ac}{algebraically closed }
\newcommand{\dieu}{Dieudonn\'{e} }

\nc{\ol}{\overline}
\nc{\wt}{\widetilde}
\nc{\opp}{\mathrm{opp}}
\def\ul{\underline}

\makeop{Ram}
\makeop{Rep}

\begin{document}
\renewcommand{\thefootnote}{\fnsymbol{footnote}}
\setcounter{footnote}{-1}
\numberwithin{equation}{section}


\title[Definite quaternion multiplication]
{On reduction of moduli schemes of abelian varieties with definite
  quaternion multiplications}
\author{Chia-Fu Yu}
\address{
(YU) Institute of Mathematics, Academia Sinica 
and NCTS \\ 
6th Floor, Astronomy Mathematics Building \\
No. 1, Roosevelt Rd. Sec. 4 \\ 
Taipei, Taiwan, 10617} 
\email{chiafu@math.sinica.edu.tw}



\date{\today}

\begin{abstract}
  In this paper we make an initial study on type D moduli spaces in
  positive \ch $p\neq 2$, where we allow the prime $p$ to ramify in the
  defining datum. We classify explicitly the isogeny
  classes of $p$-divisible groups with additional structures in
  question. We also study the reduction of the type D moduli spaces
  of minimal rank.  
\end{abstract} 

\maketitle


\def\char{{\rm char\,}}

\def\GAut{{\rm GAut}}
\def\calMpK{\calM^{(p)}_K}
\section{Introduction}
\label{sec:01}

\subsection{}
\label{sec:1.1}
PEL moduli spaces parametrize abelian varieties with additional
structures of polarizations, endomorphisms and level structures. 
When the adjoint group $G^{\rm ad}$ of the defining algebraic group $G$ is $\Q$-simple, these
moduli spaces are divided into types A, C and D according to the Dynkin diagram of $G^{\rm ad}$.
Previous studies of these moduli spaces and
their integral models have mainly focused on the spaces of types
A and C in the case of good reduction. There is comparatively less known
about type D moduli spaces in the literature. 
Certain important results on
all smooth PEL-type moduli spaces, 
which of course include the case of type D, 
have been obtained by Wedhorn~\cite{wedhorn:ordinary, wedhorn:eo} and
Moonen~\cite{moonen:bt1, moonen:eo, moonen:st},
where they consider the density of the $\mu$-ordinary locus and the
Ekedahl-Oort (EO) strata. In this paper we study the type D moduli spaces
in mainly positive \ch and certain basic classification problems for
abelian varieties and 
associated $p$-divisible groups with additional structures in
question. A main point here is that
we allow the prime $p$ to be ramified in the definite quaternion
algebra concerned. In the minimal rank case, 
we also exhibit a method for studying the case with arbitrary 
polarization degree.
 


Throughout this section let $p$ denote an odd prime number. 
Let $F$ be a totally real algebraic number field 
and $O_F$ the ring of integers. 
Let $B$ be a totally definite quaternion algebra
over $F$ and let $*$ be the canonical involution on $B$, which is the
unique positive involution on $B$. Let $O_B$ be an $O_F$-order 
in $B$ which is stable under the involution $*$ 
and maximal at $p$, that is, the completion 
$O_B\otimes_\Z \Zp$ at $p$ is a maximal order in the algebra
$B_p:=B\otimes_\Q \Qp$. A {\it polarized $O_B$-abelian  variety}\footnote{The author also used the terminology "abelian $O_B$-variety" for the same object in his earlier papers \cite{yu:reduction,yu:cm, yu:mass_hb,yu:gmf, yu:smf}. He apologizes for the inconsistency.} 
is a tuple
$\ul A=(A, \lambda,\iota)$, where $(A,\lambda)$ is a polarized abelian
variety and $\iota: O_B \to \End(A)$ is a ring monomorphism such that 
$\lambda\circ \iota(b^*)= \iota(b)^t \circ \lambda$, i.e.~the map
$\iota$ is compatible with the involution
$*$ and the Rosati involution induced by 
the polarization $\lambda$. 
Clearly, this notion can be defined over any base scheme and one
can study families of such objects. 

Let $m\ge 1$ be a positive integer, and 
let $\calM=\calM_{m,O_B}$ be the coarse moduli scheme over
$\Spec \Z_{(p)}$ of $2m[F\!:\!\Q]$-dimensional 
polarized $O_B$-abelian varieties $\ul A=(A,\lambda,\iota)$. 
Let $\calM^{(p)}\subset \calM$ 
be the open and closed subscheme consisting of
objects $\ul A$ with prime-to-$p$ polarization degree. 
Both moduli spaces $\calM$ and $\calM^{(p)}$ are schemes locally
of finite type.
Let $\calM_K\subset \calM$ (resp.~$\calMpK\subset \calM^{(p)}$)
denote the closed subscheme parametrizing the
objects $\ul A$ in $\calM$ (resp.~$\calM^{(p)}$) that satisfy the
determinant condition; see Section~\ref{sec:23}.
We shall call $\calM$ or one of its variants a moduli scheme (or
moduli space) of type $D_m$. 
The fibers of these moduli schemes are non-empty; 
see Lemmas~\ref{22} and \ref{215}. 
The goal of the paper is to investigate the geometry of these
moduli spaces. We make a detailed study on the moduli spaces of type $D_1$
because this is the most basic and most accessible case, and some speculation indicates that this family behaves quite differently from the higher rank cases. 

 



We explain the relation of the usual Shimura varieties of PEL-type D
and the moduli spaces of type $D_m$. 
Let $(V,\psi)$ be a $\Q$-valued non-degenerate 
skew-Hermitian $B$-module of $B$-rank $m$. 
Let $G:=GU_B(V,\psi)$ denote 
the $\Q$-group of $B$-linear similitudes on $(V,\psi)$, and let $X$ be
the $G(\R)$-conjugacy class of an $\R$-homomorphism $h:\C^\times \to
G_\R$ such that $\psi(x,h(i)y)$ is definite on $V_\R$. 
Note that up to a central torus, the $\Qbar$-group $G_{\Qbar}:=G\otimes \Qbar$ is isogenous to the product of copies of the even orthogonal 
group $O_{2m}$; see Section~\ref{sec:22}.
Let $U\subset
G(\A_f)$ be the open compact subgroup that stabilizes a lattice $\Lambda_0$.
Then the associated Shimura variety $\Sh_U(G,X)_{\Qbar}$ is 
is an open and closed subscheme of $\calM_{\Qbar}$.
Conversely, any irreducible component 
of $\calM_{\Qbar}$ is isomorphic to a component of 
$\Sh_U(G,X)_{\Qbar}$ for
some $(V,\psi)$ and $U$ as above; see Section~\ref{sec:adelic}.

The main contents of this paper handle the following two basic
problems:  

\begin{itemize}
\item[{\bf (a)}] Classify explicitly the isogeny classes of
  quasi-polarized 
  $p$-divisible groups with additional structures (of arbitrary rank $m$) 
  in question.
\item [{\bf (b)}] Study the reduction of the moduli spaces of type $D_1$. 
\end{itemize}

As the reader can see from known results on classical moduli spaces like
Siegel or Hilbert moduli spaces, the results obtained so far 
for moduli spaces of type D are 
comparably much less.   
Wishfully we could obtain more results based on the present work.
Below we illustrate our main results.    


\subsection{Part (a)}
\label{sec:1.2}

For any polarized $O_B$-abelian variety
$\ul A=(A,\lambda_A, \iota_A)$ over an \ac field $k$ of \ch $p$, 
the associated $p$-divisible group 
$(H,\lambda_H,\iota_H):=(A,\lambda_A, \iota_A)[p^\infty]$ with
additional structures is a quasi-polarized $p$-divisible $O_B\otimes
\Zp$-module (see Section~\ref{sec:51}). We would like to determine the slope
sequences and isogeny 
classes of these $p$-divisible groups with additional structures. 
As a first standard step, we decompose these $p$-divisible
groups and study the same problem for each component independently. 
Write 
\begin{equation}
  \label{eq:1.1}
F\otimes_\Q \Qp=\prod_{v|p} F_v, \quad B\otimes_{\Q} \Qp=\prod_{v|p}
B_v, \quad O_B\otimes_\Z \Zp=\prod_{v|p} \calO_{B_v}.  
\end{equation}
Then we get a decomposition $(H,\lambda_H,\iota_H)=\prod_{v|p}
(H_v,\lambda_{H_v},\iota_{H_v})$. The slope sequence $\ul \nu(\ul A)$
of $\ul A$ is then
defined to be the collection $(\nu(H_v))_{v|p}$ 
with slope sequences $\nu(H_v)$ indexed by
the set of places $v|p$ of $F$. So we may study 
quasi-polarized $p$-divisible $\calO_{B_v}$-modules for each place $v|p$ separately.  We shall write $\bfB$, $\bfF$
and $\calO_{\bfB}$ for $B_v$, $F_v$ and $\calO_{B_v}$, respectively
for brevity. In this part we do the following:
\begin{enumerate}
\item Study the structure of skew-Hermitian $\calO_\bfB\otimes
  W$-modules and quasi-polarized \dieu $\calO_\bfB$-modules (see
  Section~\ref{sec:51}), where $W$
  denotes the ring of Witt vectors over $k$. See 
  Sections~\ref{sec:04} and \ref{sec:05}.

\item Determine all possible slope sequences of quasi-polarized \dieu
  $\calO_\bfB$-modules of rank $4dm$, where $d=[\bfF:\Qp]$. Moreover,
  we show that these slope sequences can be also realized by 
  {\it separably}
  quasi-polarized \dieu $\calO_\bfB$-modules. See Theorems~\ref{64}
  and \ref{73} for the
  precise statements; also see Corollaries~\ref{76} and~\ref{77} for
  the list of all possible slope sequences in the cases $m=1$ and $m=2$. 

\item Classify the isogeny classes of  quasi-polarized \dieu
  $\calO_\bfB$-modules of rank $4dm$; see Section~\ref{sec:09}.  
 
\end{enumerate}

The method of finding possible slope sequences in (2) uses a
criterion for 
embeddings of a simple algebra into another one over a local
field (see \cite{yu:embed} and Section~\ref{sec:63}). 
This gives a description for possible slope
sequences. Then we construct a separably quasi-polarized \dieu
$\calO_\bfB$-module realizing each possible slope sequence. The
construction is divided into the supersingular part and
non-supersingular part. For the supersingular part we can even 
write down a separably quasi-polarized
superspecial \dieu $\calO_\bfB$-module that also 
satisfies the determinant condition. 
For the non-supersingular part we use the
``double construction''; see Lemma~\ref{71}. 
The ``double construction'' easily
produces a separable $\calO_{\bfB}$-linear polarization. However,   
a \dieu $\calO_\bfB$-module obtained in this manner 
rarely satisfies the determinant condition (see Remark~\ref{74}).
In fact, given a possible slope sequence $\nu$ as in (2) 
(see Theorem~\ref{73} for a precise description), 
it is not always possible to construct a \dieu
$\calO_\bfB$-module $M$ with slope sequence $\nu$ 
which {\it both} admits a separable $\calO_\bfB$-linear
quasi-polarization and satisfies the determinant condition. 
We will discuss this in more detail in the minimal case later
(cf.~Theorem~\ref{1.2}).   

To classify the isogeny classes of the $p$-divisible groups with
additional structures, it suffices to classify those with
a fixed slope sequence $\nu$. Rapoport and Richartz
\cite{rapoport-richartz} gave a cohomological description of this
finite set $I(\nu)$. In fact, they treated a general case
of connected groups while the groups here are not connected. 
(For more general classification by Galois cohomology 
which also includes non-connected
groups, see Kottwitz~\cite{kottwitz:isocrystals2}.) 
Here we carry out a more elementary approach 
which involves only quadratic forms and Hermitian 
forms, and obtain more explicit results in terms of invariants 
(rather than cohomology classes). 
One can first
reduce the case to where $\nu$ is supersingular; see Lemma~\ref{91}.      
Then we establish a bijection between the set $I(\nu)$ ($\nu$ is
supersingular) and the set
of isomorphism classes of skew-Hermitian $\bfB'$-modules for a certain
twisted quaternion $\bfF$-algebra $\bfB'$; see
Theorem~\ref{92}. 
When $\bfB'$ is the matrix algebra, one reduces to classifying
quadratic forms over $\bfF$, and we apply the classical theory of 
quadratic forms over local fields (cf.~O'Meara~\cite{omeara:book}). 
When $\bfB'$ is the quaternion
division algebra, we adopt the work of Tsukamoto \cite{tsukamoto:1961}.

\subsection{Part (b)}
\label{sec:1.3}
In this part we restrict ourselves to the case where $m=1$. 
Part (b) consists of Sections~\ref{sec:11}--\ref{sec:14}.
A main result of this part states as follows.

\begin{thm}\label{11} Assume that $m=1$.
  \begin{enumerate}
  \item Suppose that $p$ is unramified in $B$, that is, the algebra 
   $B\otimes\Qp$ is a
    product of matrix algebras over unramified field extensions of
    $\Qp$. Then the moduli scheme 
  $\calM_K\to \Spec \Z_{(p)}$ is flat and every connected component is
  projective of
  relative dimension zero. 

\item  The moduli scheme $\calMpK\to \Spec \Z_{(p)}$ is flat and every
  connected component
  is projective of relative dimension zero. 
  \end{enumerate}
\end{thm}

Theorem~\ref{11} confirms a special case of the Rapoport-Zink conjecture
on integral models of Shimura varieties; see
\cite{rapoport-zink}. 
Theorem~\ref{11} (1) is proved in \cite[Theorem 4.5]{yu:lift}. 
The proof of
Theorem~\ref{11} (2), given in Section~\ref{sec:11}, uses local models
(see Rapoport-Zink \cite{rapoport-zink}).
More precisely, let $\bfM_\Lambda$ be the local model over $\Spec \Zp$ 
associated to a unimodular skew-Hermitian free $O_B\otimes \Zp$-module 
$\Lambda$ (cf. Section~\ref{sec:10.1}). 
Using the Rapoport-Zink local model diagram, 
one can prove Theorem~~\ref{11} (2) by proving the flatness of 
$\bfM_{\Lambda}$.
To prove the latter, 
we show that any point in $\bfM_{\Lambda}(k)$, 
where $k$ is an algebraically closed 
field of \ch $p$, can be lifted to \ch
zero and that all its geometric fibers are zero-dimensional.

We remark that the analog of the flatness result in Theorem~\ref{11}
(1) does not hold for higher $m$; cf. \cite[Example
8.3]{pappas-rapoport:3}, which shows that flatness fails without the
"spin condition" for $m>1$. Example 8.2 of  
loc.~cit.~shows that the parahoric analog of the flatness result in
Theorem~\ref{11} (1) for $m=1$ does not hold without the "spin
condition", either.   
 

For the possible slope sequences of objects in the case $m=1$, we
have the following result, which is a refinement of Theorem~\ref{73}.

\begin{thm}[Theorem~\ref{12.3}]\label{1.2}
  Let $M$ be a separably quasi-polarized \dieu
  $\calO_\bfB$-module of rank $4d$ satisfying the determinant
  condition, where $d=[\bfF:\Qp]$. 
  \begin{enumerate}
  \item If $\bfB$ is the matrix algebra, then 
\[  \nu(M)=\left \{\left (\frac{a}{d}\right )^{2d}, \left
 (\frac{d-a}{d}\right )^{2d}\right \},  \]
where $a$ can be any integer with $0\le a< d/2$, or 
\[\nu(M)=\left \{\left (\frac{1}{2}\right )^{4d}\right \}. \]    
  \item If $\bfB$ is the division algebra, then 
\[  \nu(M)=\left \{\left (\frac{a}{2d}\right )^{2d}, \left
 (\frac{2d-a}{2d}\right )^{2d}\right \},  \]
where $a$ can be any odd integer with $2\lfloor e/2 \rfloor f \le a< d$, or 
\[ \nu(M)=\left \{\left (\frac{1}{2}\right )^{4d}\right \}. \]
Here $e$ and $f$ are the ramification index and the inertia degree of
$\bfF$ over $\Qp$, respectively.    
\end{enumerate}
\end{thm}

In the remainder of this part (Sections~\ref{sec:13} and \ref{sec:14})
we limit ourselves to the case $F=\Q$.
We determine the dimensions of the special fibers
of various moduli schemes as above. 


\begin{thm}[Theorem~\ref{14.1} and Proposition~\ref{14.5}]\label{13}
  Assume that $m=1$ and $F=\Q$. Let 
$\calM_{\Fpbar},  \calM^{(p)}_{\Fpbar},$ and
$\calM_{K,\Fpbar}$ be the geometric special fibers of $\calM$,
$\calM^{(p)}$ and $\calM_{K}$, respectively.
  \begin{enumerate}
  \item If $p$ is unramified in $B$, then $\dim \calM_{\Fpbar}=0$.
  \item If $p$ is ramified in $B$, then $\dim \calM_{\Fpbar}=1$.
  \item We have $\dim \calM^{(p)}_{\Fpbar}=0$. 
  \item If $p$ is ramified in $B$, then $\dim \calM^{}_{K, \Fpbar}=1$.  
  \end{enumerate}
\end{thm}

We explain the ideas of the proof.
For (1) any object $\ul A=(A,\lambda,\iota)\in \calM(k)$ is either
ordinary or superspecial. For the first case 
we use the canonical lifting for ordinary abelian varieties 
and the fact that the generic fiber 
has dimension zero.
For the superspecial case, we use the fact that 
the superspecial locus has dimension zero. 
For (2) and (4), we construct a $\bfP^1$-family 
of supersingular polarized $O_B$-abelian surfaces. 
This produces a 
one-dimensional family in the moduli space $\calM_{\Fpbar}$. 
A close examination shows that this family actually lands in the locus $\calM_{K,\Fpbar}\subset \calM_{\Fpbar}$; see Lemma~\ref{14.4}. 
This gives a lower bound for the dimensions of our moduli spaces:
\[ 1\le \dim \calM_{K,\Fpbar} \le \dim \calM_{\Fpbar}. \]
For the other bound, we consider the finite morphism $f: \calM_{\Fpbar} \to
\calA_{2,\Fpbar}$ to the moduli space $\calA_{2,\Fpbar}$ of polarized
abelian surfaces, through forgetting the endomorphism structure. 
As $p$ is ramified, 
every object in $\calM(k)$ is supersingular (cf.~Corollary~\ref{76}),
and hence the whole space $\calM_{\Fpbar}$ is supersingular. The
image of $\calM_{\Fpbar}$ in $\calA_{2,\Fpbar}$ then lands in the
supersingular locus $\calS_2$ of $\calA_{2,\Fpbar}$. Then we use a
result of Norman-Oort \cite{norman-oort} (also cf. Katsura and Oort
\cite{katsura-oort:surface} 
for the principally polarized case) that $\dim \calS_2=1$ 
and get the other bound $\dim \calM_{\Fpbar}\le 1$. 
This proves (2) and (4). 
Note that the result $\dim \calS_2=\dim \calA_{2,\Fpbar}^{(0)}=1$ we
use is a special case of a theorem of Norman and Oort stating that 
the $p$-rank zero locus $\calA^{(0)}_{g,\Fpbar}$ of the Siegel moduli 
space $\calA_{g,\Fpbar}$ has co-dimension $g$. 
For (3) we only need to treat the case where $p$ is
ramified. In this case we show that any separably quasi-polarized
$O_B\otimes\Zp$-\dieu module (of rank $4d$) is superspecial,
and again use the dimension of the superspecial locus.

\begin{remark}\label{14}
 We end this part with a few remarks about Theorem~\ref{13}.

\begin{enumerate}
\item Theorem~\ref{13} (3) yields another proof of the result $\dim
\calM^{(p)}_{K,\Fpbar}=0$, which follows from Theorem~\ref{11}.
\item When $p$ is ramified in $B$, 
both $\calM_{\Fpbar}$ and $\calM_{K,\Fpbar}$ contain components of
dimension zero and one. This follows from the result that 
both $\calM^{(p)}_{\Fpbar}$ and $\calM^{(p)}_{K,\Fpbar}$ are
zero-dimensional and non-empty (by Theorem~\ref{13} (3) and Lemma~\ref{215}). 
\item Suppose $p$ is ramified in $B$. By Theorem~\ref{13} (2) and (4),
  we conclude that the moduli schemes $\calM$ and $\calM_{K}$ are not
  flat over $\Spec \Z_{(p)}$. Moreover, there is a (non-separably) 
  polarized $O_B$-abelian
  surface satisfying the determinant condition which can not be
  lifted to \ch zero. 
\item When $p$ is unramified in $B$, we have $\calM^{(p)}=\calMpK$ 
  and hence the
  moduli scheme $\calM^{(p)}$ is flat over $\Spec \Z_{(p)}$. When
  $p$ is ramified in $B$, we construct a superspecial prime-to-$p$
  degree  
  polarized $O_{B}$-abelian surface which does not satisfy
  the determinant condition; see Lemma~\ref{14.3}. In
  particular, this point can not be lifted to \ch zero.
  This shows
  that the inclusion $\calM^{(p)}_{K}(k) \subset
  \calM^{(p)}(k)$ is strict. 
This phenomenon is different from the reduction modulo $p$ of
Hilbert moduli schemes or Hilbert-Siegel moduli schemes. 
In the Hilbert-Siegel case, any separably polarized
abelian varieties with RM by $O_F$ of a totally real algebraic number 
field $F$ satisfy the determinant condition automatically; see Yu
\cite{yu:reduction}, G\"ortz~\cite{goertz:topological} and
Vollaard~\cite{vollaard:thesis}. 
  
\item To generalize Theorem~\ref{13} to the case where $B$ is a
  quaternion algebra over any totally real field $F$, one can
  construct Moret-Bailly families to get a lower bound for the
  dimensions. However, we do not know how to produce a good upper
  bound.   
\end{enumerate} 

\end{remark}

Local models for orthogonal groups of general rank 
have been studied in Smithling
\cite{smithling:admD, smithling:tfD} in the case when $p$ is 
unramified in $B$. 
Vasiu~\cite{vasiu:D1,vasiu:D2} shows the existence of the canonical
integral model of certain Shimura varieties of PEL-type D over
$\Z_{(2)}$. 
He and Rapoport \cite{he-rapoport} 
show non-emptiness of
various stratifications (NP, KR and EO) in the reduction of Shimura
varieties under certain axioms. Rapoport and Viehmann
\cite{rapoport-viehmann} studies local analogues of Shimura varieties
(which are RZ spaces in the PEL-type cases) and explore non-emptiness
of these spaces and the Kottwitz set $B(G,\{\mu\})$ (see the
definition in \cite{kottwitz:isocrystals2}). 

We make a remark on the assumption of $p$. Our primary goal was
to establish as many results as we can only limited to the assumption
$p\neq 2$. However, as far as we can see, several pointwise statements
(in terms of \dieu modules or isocrystals), as long as which do not
rely on the classification up to isomorphism, do not require this
assumption. After a close examination, we do not need to assume this
in Sections 2--8, except for Section~\ref{sec:44}, Lemma~\ref{51},
Corollary~\ref{57} and Proposition~\ref{98}. The assumption $p\neq 2$
is needed for the construction of the local model diagram by Rapoport
and Zink (see \cite[p.~75 and Theorem 3.16]{rapoport-zink}). 
Thus, we make this assumption starting from Section 9.

\section{Moduli spaces}
\label{sec:02}

\subsection{Moduli spaces}
\label{sec:21}

Let $p$ be a prime number. 
Let $F$ be a totally real field of
degree $d=[F:\Q]$ and $O_F$ the ring of integers. 
Let $B$ be a totally definite quaternion algebra
over $F$ and let $*$ be the canonical involution on $B$, which is the
unique positive involution on $B$. Let $O_B\subset B$ be an
$O_F$-order in $B$ which is stable under the involution $*$ and 
maximal at $p$, that is, $O_B\otimes_\Z \Zp$ is a maximal order 
in the algebra $B_p:=B\otimes_\Q \Qp$. A {\it polarized $O_B$-abelian 
scheme} over a base 
scheme $S$ is a tuple
$(A, \lambda,\iota)$, where 
\begin{itemize}
\item $(A,\lambda)$ is a polarized abelian
scheme over $S$, and
\item $\iota: O_B \to \End_S(A)$ is a ring monomorphism that satisfies
  the compatibility condition 
$\lambda\circ \iota(b^*)= \iota(b)^t \circ \lambda$ for all $b\in O_B$.
\end{itemize}
The pair $(A,\iota)$, where $A$ and $\iota$ are as above, 
 is called an $O_B$-abelian scheme. A
polarization $\lambda$ on an $O_B$-abelian scheme $(A,\iota)$
 satisfying the above compatibility condition is said to be 
{\it $O_B$-linear}. Similar objects can be also defined
when the algebra $B$ is replaced by an arbitrary 
semi-simple $\Q$-algebra together with a positive involution. 

Let $m\ge 1$ be a positive integer, and   
let $\calM=\calM_{m,O_B}$ be the coarse moduli scheme over  
$\Spec \Z_{(p)}$ of $2dm$-dimensional polarized $O_B$-abelian varieties
$(A,\lambda,\iota)$. 
Let $\calM^{(p)}\subset \calM$ 
be the subscheme parametrizing the 
objects in $\calM$ which have a prime-to-$p$ degree polarization. 
The moduli spaces $\calM$ and $\calM^{(p)}$ are schemes both locally
of finite type. Each of them is an union of infinitely many open and
closed subschemes which are of finite type:
\begin{equation}
  \label{eq:21}
  \calM=\coprod_{D \ge 1} \calM_D, \quad 
  \calM^{(p)}=\coprod_{D\ge 1, p\nmid D} \calM^{}_D, 
\end{equation}
where $D$ runs through all positive integers and 
$\calM_D\subset \calM$ is the subscheme parametrizing the objects
$(A,\lambda,\iota)$ with polarization degree $\deg \lambda=D^2$. 



\subsection{Study of $\calM_\C$.}
\label{sec:22}

Let $(V,\psi)$ be a $\Q$-valued non-degenerate 
skew-Hermitian $B$-module of $B$-rank $m$. 
That is, $\psi:V\times V\to \Q$ is a non-degenerate alternating
pairing such that $\psi(ax,y)=\psi(x, a^* y)$ for all $a\in B$ and
$x,y\in V$. Let $G:=GU_B(V,\psi)$ denote 
the group of $B$-linear similitudes on $(V,\psi)$ over $\Q$. 
For any commutative $\Q$-algebra $R$, the group of $R$-valued points
is given by 
\begin{equation}
  \label{eq:22}
  G(R)=\{ g\in \GL_{B\otimes_\Q R}(V\otimes_\Q R)\, \mid
  \, c(g)=g'g\in R^\times\, \},
\end{equation}
where $g\mapsto g'$ is the adjoint involution with respect to
$\psi$. Let $G_1=U_B(V,\psi)$ be the kernel of the multiplier
homomorphism $c:G\to \Gm$; one has 
a short exact sequence of algebraic $\Q$-groups
\begin{equation}
  \label{eq:G}
  \begin{CD}
  1 @>>> G_1 @>>> G @>{c}>> \Gm @>>> 1.    
  \end{CD}
\end{equation}


One can easily show that 
\begin{equation}
  \label{eq:GQ}
  G_{\Qbar}:=G\otimes \Qbar\simeq \left \{(A_i)\in \GL_{2m, \Qbar}^d\, \biggm| \, A_i
  \begin{pmatrix}
     0 & I_m \\
    I_m & 0 \\
  \end{pmatrix} A_i^t=c  
  \begin{pmatrix} 
    0 & I_m \\
    I_m & 0 \\
  \end{pmatrix} \text{for all $i$} \right \}. 
\end{equation}
A simple calculation shows that for $t\in \Gm$, one has
$\left(\left[\begin{smallmatrix} I_m & 0 \\ 0 & t
    I_m \end{smallmatrix}\right],\dots, \left[\begin{smallmatrix} I_m & 0 \\ 0 &
    t I_m \end{smallmatrix}\right]\right)\in G_{\Qbar}$. This gives a
section of $c$ over $\Qbar$. 
Thus, one has
$G_{1, \Qbar}\simeq  O_{2m}^d$ and 
$G_{\Qbar}\simeq O_{2m}^d\rtimes \Gm$ over $\Qbar$. In particular,
both $G_1$ and $G$ have $2^d$ connected components. One can also 
easily show that
\begin{equation}
  \label{eq:Z}
  \begin{split}
    & Z(G)(\Qbar) =\{x\in (F\otimes \Qbar)^\times| x^2\in
    \Qbar^\times\}, \\
    & Z(G)=Z(G^0), \quad \text{if $m>1$},      
  \end{split}
\end{equation}
and that for $m=1$, the group $G^0$ is a torus of dimension $d+1$.  
 


There is a unique $B$-valued skew-Hermitian pairing $\psi_B:V\times
V\to B$, 
i.e. 
\[ \psi_B(a_1x, a_2y)=a_1 \psi_B(x,y) a_2^*, \quad  
\forall\, a_1,a_2\in B, \quad x,y\in V, \] 
such that $\psi(x,y)=\Trd \psi_B(x,y)$, where $\Trd$ is
the reduced trace from $B$ to $\Q$. Note that the property
$\psi_B(ax,y)=\psi_B(x,a^*y)$ for all $a\in B$ and $x, y\in V$ does
not hold anymore as $B$ is not commutative.
We can choose an orthogonal basis $\{e_i\}$ for $\psi_B$ and put
$b_i:=\psi_B(e_i, e_i)$. Then $b_i^*=-b_i$ for $i=1,\dots, m$ and 
\[ \psi\left (\sum_{i=1}^m x_i e_i,\sum_{i=1}^m y_i e_i \right )=
 \sum_{i=1}^m \Trd (x_ib_iy_i^*). \] 
For any anti-symmetric element $b\in B^\times$, i.e.~$b^*=-b$, 
we define a ($\Q$-valued) rank-one
skew-Hermitian $B$-module $(B,\psi_b)$, 
where $\psi_b(x,y):=\Trd(xby^*)$.
Then we have a decomposition of skew-Hermitian $B$-modules
\begin{equation}
  \label{eq:23}
  (V,\psi)=\bigoplus_{i=1}^m\, (B,\psi_{b_i}).
\end{equation}

\begin{lemma}\label{21}
There is a $B\otimes \R$-linear complex structure $J_0$ on
$V_\R=V\otimes_\Q \R$ such that $\psi(J_0x, J_0 y)=\psi(x,y)$ for 
$x,y\in V_\R$ and the symmetric bilinear form 
$(x,y):=\psi(x,J_0y)$ is negative definite.  
\end{lemma}
\begin{proof}
  By (\ref{eq:23}) we may assume that $m=1$ and
  $(V,\psi)=(B,\psi_b)$. Let $J_0$ be the 
  right multiplication of the element
  $b/\sqrt{\Nr_{B/F}(b)}$ in $B\otimes \R$, 
  where $\Nr_{B/F}$ is the reduced norm from $B$ to $F$. Then one obtains
  $\psi_b(x,J_0 y)=-\Trd(xy^*)$, which is negative definite. \qed
\end{proof}

We call a complex structure $J_0$ as in Lemma~\ref{21} 
an {\it admissible} complex structure on $(V_\R,\psi)$. 
The group $G_1(\R)$ of real points 
acts transitively on the set of all admissible complex structures
on $(V_\R,\psi)$ by conjugation
(see \cite[Lemma 4.3]{kottwitz:jams92}). 
It is well known that the Hermitian 
symmetric space 
\[ X_1:=G_1(\R)/K_\infty \]  
has dimension $d m(m-1)/2$, where $K_\infty$ is 
the stabilizer of a fixed admissible
complex structure $J_0$. Fix an $O_B$-lattice $\Lambda$ such that
$\psi(\Lambda,\Lambda)\subset \Z$ and let 
$\Gamma_{\Lambda}\subset G_1(\Q)$ be the 
arithmetic subgroup which stabilizes the lattice $\Lambda$. The natural 
map $g\mapsto (V_\R/\Lambda, \Int(g)J_0, \psi)$ induces an open and
closed immersion of analytic spaces       
\[ \Phi_{(\Lambda,\psi)}:\Gamma_\Lambda\backslash X_1\embed \calM(\C). \]
Let $\calM_{(\Lambda,\psi)}$ denote the open and closed subscheme of
$\calM_\C$ over $\C$ whose underlying space is the image of
$\Phi_{(\Lambda,\psi)}$. Then we have a decomposition of $\calM_\C$
into open and closed subschemes 
\begin{equation}
  \label{eq:calMC}
  \calM_\C=\coprod_{(\Lambda,\psi)} \calM_{(\Lambda,\psi)},
\end{equation}
where $(\Lambda,\psi)$
runs through the isomorphism classes of all $\Z$-valued non-degenerate 
skew-Hermitian $O_B$-lattices of rank $m$. We will see that each  
subscheme $\calM_{(\Lambda,\psi)}$ is irreducible (Proposition~\ref{26}).


\begin{lemma}\label{22}\ 
\begin{enumerate}
  \item There is an anti-symmetric element $b\in B^\times$ such that
  {\rm (a)} $\psi_b(O_B,O_B)\subset \Z$ and {\rm (b)} $O_B\otimes
  \Z_p$ is a 
  self-dual 
  lattice with respect to $\psi_b$, where $\psi_b(x,y):=\Trd(xby^*)$.
\item The moduli space $\calM^{(p)}_\C$ is non-empty.
\end{enumerate}
     
\end{lemma}
\begin{proof}
  (1) We have the decomposition $O_B\otimes \Zp=\oplus_{v|p} \calO_{B_v}$ with
  respect to $O_F\otimes \Zp=\prod_{v|p} \calO_{v}$. 
  We first show that for each place $v$ of $F$ over $p$, one can
  choose an anti-symmetric
  element $b_v \in B_v^\times$ such that $\psi_{b_v}$ is a 
  $\Zp$-valued perfect paring
  on $\calO_{B_v}$. When $v$ is unramified in $B$, we are reduced to 
  finding a $\Zp$-valued perfect
  symmetric pairing on $\calO_{v}^2$ which clearly exists, and let $b_v$
  be the element corresponding to this perfect pairing. 
  When $v$ is ramified in $B$, one may choose a prime element $\Pi_v$ of
  $B_v$ such that $\Pi^2_v$ is a uniformizer of $F_v$, and let
  $b_v:=\delta_v \Pi_v^{-1}$, where $\delta_v$ is a generator of
   the inverse different $\calD_{F_v/\Qp}^{-1}$.  

  Using weak approximation, there is
  an anti-symmetric element $b\in B^\times$ close to $b_v$ for each
  place $v|p$. Replacing $b$ by a prime-to-$p$ multiple $b M$ of $b$,
  one gets a pairing $\psi_b$ that satisfies 
  the desired properties. 

  (2) Choose a pairing $\psi=\psi_b$ on $B$ as in (1). Then the triple 
$(V_\R/O_B, J_0,\psi)$ defines a $2d$-dimensional polarized
complex $O_B$-abelian variety $(A_0,\lambda_0,\iota_0)$. 
Put $\ul A=(A_0^m, \lambda_0^m, \iota)$, where $\iota:O_B\to
\End(A_0^m)$ is the diagonal embedding. 
This gives an object in $\calM^{(p)}(\C)$ and hence that
$\calM^{(p)}_\C$ is non-empty. \qed 
\end{proof}


\subsection{Moduli spaces with the determinant condition}
\label{sec:23}
Let $(V,\psi)$ be any skew-Hermitian $B$-module of $B$-rank $m$. 
Let $J_0$ be an admissible complex structure on $V_\R$. 
Let $V_\C=V^{-1,0}\oplus V^{0,-1}$
be the eigenspace decomposition where $J_0$ acts respectively by
$i$ and $-i$ on $V^{-1,0}$ and $V^{0,-1}$. Let
$\Sigma:=\Hom(F,\R)=\Hom (F,\C)$ be
the set of real embeddings of $F$. We have 
\[ B_\C:=B\otimes_\Q \C=\prod_{\sigma\in \Sigma} B_\sigma, \quad
B_\sigma:=B\otimes_{F,\sigma} \C\simeq \Mat_2(\C). \]
The action of $B_\C$ on $V^{-1,0}$ gives the decomposition
\[  V^{-1,0}=\sum_{\sigma\in \Sigma} V_\sigma,  \] 
where each subspace $V_\sigma$ is a $B_\sigma$-module of
$\C$-dimension $2m$ and it is isomorphic to the direct sum 
of $m$-copies of the simple $B_\sigma$-module $\C^2$. 
If $a\in B$, then the characteristic
polynomial of $a$ on $V_\sigma$ is equal to $\sigma({\rm
  char}_F(a)^m)$, 
where ${\rm char}_F(a)\in F[T]$ is the reduced \ch polynomial of $a$.  
Therefore, 
the characteristic polynomial of $a\in
B$ on $V^{-1,0}$ is given by  
\begin{equation}\label{eq:275}
\char (a|V^{-1,0})=\prod_{\sigma\in \Sigma}\sigma({\rm char}_F
(a))^m=\, {\char} (a)^m\in \Q[T],
\end{equation}
where $\char(a)=N_{F/\Q}\, {\rm char}_F (a)\in \Q[T]$ 
is the reduced \ch polynomial of $a$ from $B$ to $\Q$.

Let $\ul A = (A,\lambda,\iota)_S \in \calM(S)$ be a polarized $O_B$-abelian
scheme over $S$, where $S$ is a connected locally Noetherian
$\Z_{(p)}$-scheme. 
Since the Lie algebra 
$\Lie(A)$ is a locally free $\calO_S$-module, 
the characteristic polynomial $\char
(\iota(a)|\Lie(A))$, for any element $a\in O_B$, is defined and 
it is a polynomial in $\calO_S[T]$ of
degree $2dm$.  
The {\it determinant condition} for $\ul A$ is the equality of the
following two polynomials 
\begin{equation}
  \label{eq:25}
  {\rm (K)}\quad  \char (\iota(a)|\Lie(A))=\char (a)^m\in \calO_S[T], \quad \forall\, a\in O_B.
\end{equation}
This is a closed condition and it depends on $m$, but not on the choice of $(V,\psi)$. 

Let $\calM_K\subset \calM$ (resp.~$\calMpK\subset \calM^{(p)}$)
denote the closed subscheme parametrizing the
objects $\ul A$ in $\calM$ (resp.~$\calM^{(p)}$) that satisfy the
determinant condition (K). 

Let $L\subset B$ be a maximal subfield such that any place $v$
of $F$ lying over $p$ is unramified in $L$ and that the order 
$L\cap O_B$ is maximal at $p$. We can construct such a maximal
subfield $L$ by 
\begin{itemize}
\item [(a)] constructing a maximal commutative semi-simple subalgebra
$\bfL\subset
B\otimes \Qp$ such that $\bfL$ is the unramified quadratic extension of
$F\otimes \Qp$ and the maximal order $\calO_{\bfL}$ of $\bfL$ 
is contained in $O_B\otimes \Zp$, and
\item [(b)] applying weak approximation.
\end{itemize}


\begin{lemma}\label{215}\ 
  {\rm (1)} The moduli space $\calM^{(p)}_{K,\Q}:=\calM^{(p)}_K\otimes \Q$ is non-empty.
  
  {\rm (2)} The special fiber $\calM^{(p)}_{K,\Fpbar}:=\calMpK\otimes \Fpbar$ is non-empty.
\end{lemma}
\begin{proof}
(1) This follows from the fact that the determinant condition for any
polarized $O_B$-abelian variety in \ch zero holds automatically,  
and Lemma~\ref{22} (2). 

(2) We may assume that $m=1$ by the same reduction step
as we show Lemma~\ref{22} (2). 
In this case,
$\calM(\C)=\calM(\Qbar)$ as
each subscheme $\calM_{D}\otimes  \Q$ is
zero-dimensional and of finite type.
By Grothendieck's semi-stable reduction theorem \cite[IX, Theorem 3.6]{sga7-1}
\footnote{We refer to the Math. Review of Artin-Winters~\cite{artin-winters}
for a historic background of this theorem.}, 
any object $(A,\lambda, \iota)\in \calM(\Qpbar)$ has good
reduction, due to the fact that the $\Z$-rank of $O_B$ is larger 
than $\dim A$. Since $\calM^{(p)}(\Qpbar)=\calMpK(\Qpbar)$ is non-empty by (1), the reduction modulo $p$ of some point $y\in \calMpK(\Qpbar)$ gives a point $x$ which is contained in  
$\calM^{(p)}_K(\Fpbar)$, because $x\in \ol{\{y\}}$ and (K) is a closed condition. Thus, 
$\calM^{(p)}_{K,\Fpbar}$ is non-empty.  \qed
\end{proof}

\subsection{Study of $\calM_{(\Lambda,\psi)}$. }
\label{sec:24}
Let $\H$ denote the real Hamilton quaternion algebra. One has
\begin{equation}
  \label{eq:27}
  \H=\C+\C\, \bfj, \quad \bfj\, a=\bar{a} \bfj\, ,\quad  a\in \C.  
\end{equation}
It is a standard fact that 
any non-degenerate skew-Hermitian $\H$-module of rank $m$ is isomorphic 
to $(\H^m, \psi_0)$, where $\psi_0(e_i,e_j)=\bfj\, \delta_{i,j}$. Put
$\bfJ_m:=\diag(\bfj,\dots, \bfj)\in \Mat_m(\H)$. We extend 
the canonical involution $*$ on $\Mat_m(\H)$ by
$(a_{ij})^*=(a'_{ij})$, where $a_{ij}\in \H$ and
$a_{ij}':=a_{ji}^*$. Let ${\rm O}_{2m}^*$ denote the
algebraic $\R$-group of isometries of $(\H^m,\psi_0)$; one has
\begin{equation}
  \label{eq:28}
  {\rm O}_{2m}^*(\R)=\{A\in \GL_m(\H) \mid A \bfJ_m A^*=\bfJ_m\, \}. 
\end{equation}
The group $G_1\otimes \R$ is isomorphic to 
$\prod_{\sigma\in \Sigma} {\rm O}_{2m}^*$.

\begin{lemma}\label{24}
  One has $G_1(\R)=G^0_1(\R)$, $G(\R)=G^0(\R)$ and
  $c(G(\R))=\R^\times$. 
\end{lemma}
\begin{proof}
  We first show that for any
  element $g\in \GL_m(\H)$ one has $\Nrd(g)>0$. Since the set
  $\GL_m(\H)^{\rm ss}\subset \GL_m(\H)$ of semi-simple elements is 
  open and dense in the classical topology, it suffices to show
  $\Nrd(g)>0$ for $g\in \GL_m(\H)^{\rm ss}$. As such $g$ is contained
  in a maximal commutative semi-simple subalgebra of $\Mat_m(\H)$,
  which is isomorphic to $\C^{m}$, one has $\Nrd(g)>0$.

  It follows that 
\begin{equation}
  \label{eq:285}
  {\rm O}_{2m}^*(\R)=\{A\in \GL_m(\H) \mid A \bfJ_m A^*=\bfJ_m, \Nrd(A)=1\, \}. 
\end{equation}
  Let $\SO_{2m}^*=\{A\in {\rm O}_{2m}^* | \Nrd(A)=1\}$. The group 
  $\SO_{2m}^*$ is a form of $\SO_{2m}$ and hence it is
  the neutral component of ${\rm O}_{2m}^*$. 
  Therefore,
  \begin{equation}
    \label{eq:29}
   G_1(\R)=\prod_{\sigma} {\rm O}_{2m}^*(\R)=\prod_{\sigma}
  \SO_{2m}^*(\R)=G^0_1(\R). 
  \end{equation}
  For the second statement, if $A {\bf J}_m A^*=c {\bf J}_m$ for some $c\in 
  \R^\times$, then   
  $\Nrd(A/c^m)>0$ and one has $G(\R)=G^0(\R)$.
  For the
  last statement we just need to find an element $g$ such
  that $c(g)<0$. Consider diagonal elements $x=\diag (y,\dots, y),
  y\in \H^\times$ and
  we are reduced to showing this in the case where $m=1$. In this case 
  one has $\bfi \bfj \bfi^*=-\bfj$ and hence $c(\bfi)=-1$. 
  This proves the lemma. \qed        
\end{proof}

\begin{lemma}\label{25} \
  \begin{enumerate}
  \item The Lie group $G_1(\R)$ is connected. 
  \item The Lie group $G(\R)$ has two
connected components with the neutral component 
\[ G(\R)^{+}=\{g\in G(\R) \mid c(g)>0 \, \}. \]
  \end{enumerate}
\end{lemma}
\begin{proof}
  (1) It suffices to show that the Lie group 
  ${\rm O}_{2m}^*(\R)=\SO_{2m}^*(\R)$ is connected. 
  We embed $\H\embed \Mat_2(\C)$ as 
\[ \H\ni a+b\bfj \mapsto \begin{pmatrix}
  a & b \\
  -\bar b & \bar a 
\end{pmatrix}\in \Mat_2(\C), \]
and have $\Mat_m(\H)\subset \Mat_{2m}(\C)$.  
Let $J$ and $J_m$ be the image of $\bfj$ and $\bfJ_m$ in $\Mat_{2}(\C)$
and $\Mat_{2m}(\C)$, respectively. Clearly, $J_m^t=-J_m$ and
$J_{m}^{-1}=-J_{m}$.  The complex conjugation on $\Mat_2(\C)$ coming from
the $\R$-structure of $\H$ is given by 
\[ A=
\begin{pmatrix}
  a & b \\
  c & d \\
\end{pmatrix}\mapsto
J\ol A J^{-1}=
\begin{pmatrix}
  \ol {d} & -\ol c \\
  -\ol b & \ol a \\
\end{pmatrix}\]
where $A\mapsto \ol A$ is the usual complex conjugation. Thus, we can
recover $\Mat_m(\H)$ from $\Mat_{2m}(\C)$ by  
\begin{equation}
  \label{eq:2.10}
  \Mat_{m}(\H)=\{A\in \Mat_{2m}(\C) \mid J_{m} \ol A J_{m}^{-1}=A\, \}.  
\end{equation}
We have $A^*=\ol {A^t}$ for $A\in \Mat_m(\H)$. By (\ref{eq:28}) and
(\ref{eq:2.10}), one gets
\begin{equation}
  \label{eq:2.11}
  {\rm O}_{2m}^*(\R)=\{A\in \Mat_{2m}(\C)\mid AJ_m \ol{A^t}=J_m,\ J_{m} \ol
  A J_{m}^{-1}=A,\ \det(A)=1 \, \}.  
\end{equation}
The first condition in (\ref{eq:2.11}) gives 
\begin{equation}
  \label{eq:2.12}
  J_m \ol{A^t} J_m^{-1}=A^{-1}. 
\end{equation}
Taking the transpose, the second condition 
in (\ref{eq:2.11}) becomes
the condition
\begin{equation}
  \label{eq:2.13}
  A^t=(J_m^{-1})^{t} \ol {A^t} J_m^t=J_m^{-1} \ol {A^t} J_m=J_m \ol {A^t}
  J_m^{-1}.  
\end{equation}
With (\ref{eq:2.12}), the condition $(\ref{eq:2.13})$ becomes $A^t
A=I_{2m}$. Therefore, 
\begin{equation}
  \label{eq:2.14}
  {\rm O}_{2m}^*(\R)=\{A\in \Mat_{2m}(\C)\mid AJ_m \ol{A^t}=J_m,\ 
  A^t A=I_{2m},\ \det(A)=1 \, \}.  
\end{equation}
The group in the RHS of (\ref{eq:2.14}) is denoted as $\SO^*(2m)$ in 
\cite[Chapter X, Section 2]{helgason:gsm34} and 
this group is connected \cite[Chapter X, Lemma
2.4]{helgason:gsm34}. Thus, ${\rm O}_{2m}^*(\R)$ is connected.

(2) This follows from (1) and Lemma~\ref{24}. \qed 
\end{proof}

\begin{remark}\label{2.55}
  In Section~\ref{sec:22}, we constructed 
  a section for $c:G\to \Gm$ over
  $\C$. 
  Now we show that there is no 
  section of $c$ over $\R$  when $m$ is odd. Suppose that 
  there is a section
  $s:t\mapsto (A_1(t),\dots,A_d(t))$ over $\R$. The composition of the 
  first component with the   
  reduced norm map $\Nrd$ gives a character $\Gm\to \Gm$, $t\mapsto \Nrd  
  (A_1(t))$,  which is $t^n$ for some $n\in \Z$. By 
  $A_1(t) {\bf J}_m A_1(t)^*=t {\bf 
  J}_m$, one has $\Nrd (A_1(t))^2=t^{2m}$ and hence $\Nrd (A_1(t))=t^{m}$. 
  Since $t^m=\Nrd(A_1(t))>0$ for all $t\in \R^\times$, $m$ must be even.

  When $m$ is even, the map 
  \[ \Gm\to G_\R,\quad t\mapsto \left(
  \begin{pmatrix} 0 & t \cdot I_{m/2} \\ I_{m/2} & 0 \end{pmatrix}, \dots, \begin{pmatrix} 0 & t \cdot I_{m/2} \\ I_{m/2} & 0 \end{pmatrix}\right) \]
  gives a section of $c$ over $\R$. 
\end{remark}

\begin{prop}\label{26}
  Let $(\Lambda,\psi)$ be a non-degenerate skew-Hermitian
  $O_B$-lattice. The subscheme $\calM_{(\Lambda,\psi)}$ of $\calM_\C$ is
  irreducible. 
\end{prop}
\begin{proof}
  This follows immediately from Lemma~\ref{25} (1). \qed
\end{proof}

\begin{cor}
  If $m=1$, then 
  the group $G_1(\R)$ is isomorphic to $(\C^\times_1)^d$ and
  the quotient space $X_1=G_1(\R)/K_\infty$ consists of one point, where 
  $\C_1^\times=\{z\in \C^\times | z \bar z=1 \}$. \qed 
\end{cor}

The following result is an immediate consequence of Proposition~\ref{26}.

\begin{prop}\label{28} Suppose that $m=1$.
  The map that sends each object $(A,\lambda,\iota)\in
  \calM(\Qbar)=\calM(\C)$ to
  its first homology group $(H_1(A(\C), \Z),\psi_\lambda)$, together 
  with the Riemann form and the $O_B$-action, induces
  a bijection between the space $\calM(\Qbar)$ and the discrete set 
  of isomorphism classes of $\Z$-valued rank one skew-Hermitian 
  $O_B$-modules
  $(V_\Z,\psi)$. 

  Moreover, the subspace $\calM^{(p)}(\Qbar)$ corresponds
  to the subset of classes with the property that $V_\Z\otimes \Z_p$
  is self-dual with respect to the pairing $\psi$. \qed  
\end{prop}

\subsection{Connection with the adelic description}
\label{sec:adelic}

Let $(V,\psi)$, $G_1$, $G$, $J_0$ be as before. 
Let $h_0:\C \to \End_{B\otimes \R}(V_\R)$ be the $\R$-algebra
homomorphism defined by $h_0(i)=J_0$ and denote again by 
$h_0:\C^\times \to G_\R$ the homomorphism of $\R$-groups. 
Let $X$ be the $G(\R)$-conjugacy class of $h_0$. Fix an 
$O_B$-lattice $\Lambda_0$ in $V$ and let $U\subset G(\A_f)$ be the open and
compact subgroup that stabilizes the lattice $\Lambda_0 \otimes \hat
\Z$. Here $\hat \Z$ is the profinite completion of $\Z$ and 
$\A_f=\hat \Z \otimes \Q$ is the finite adele ring of $\Q$. One forms
a complex Shimura variety (even though $G$ is not connected)
\[ \Sh_U(G,X)_{\C}=G(\Q)\backslash X\times G(\A_f)/U. \]  
\begin{lemma}\label{29}
  The Hermitian symmetric space $X$ is $G(\R)/\R^\times K_\infty$ and
  it has 
  two connected components.
\end{lemma}
\begin{proof}
  Since $c(G(\R))=\R^\times$ (Lemma~\ref{24}), 
  the closed immersion $G\to \GSp(V,\psi)$
  induces a surjective map $\pi_0(X)\to \pi_0(\H_g^{\pm})$, where
  $\H_g^{\pm}$ is the Siegel double space. This shows that $X$ has at
  least two connected components. On the other hand,
  since the group $G(\R)$ has two connected components (Lemma~\ref{25}), 
  $X$ can only have two connected
  components. From the short exact sequence 
  \begin{equation}\label{eq:2.15}
  \begin{CD}
     G_1(\R) @>>> G(\R)/Z(G(\R)) @>{c}>> \{\pm1\} @>>> 1, 
  \end{CD}
\end{equation}
we conclude that $X=G(\R)/Z(G(\R))
  K_\infty=G(\R)/\R^\times K_\infty$. \qed  
\end{proof}


By Lemma~\ref{24}, the group $G(\Q)$ is dense in $G(\R)$. Thus, 
\begin{equation}
  \label{eq:2.16}
  \begin{split}
  \Sh_U(G,X)_\C&=G(\Q)^+\backslash X_1\times G(\A_f)/U \\
  &=\coprod_{i=1}^h \Gamma_i \backslash X_1,\quad \Gamma_i=G(\Q)^+
     \cap c_i U c_i^{-1},    
  \end{split}
\end{equation}
where $G(\Q)^+=G(\Q)\cap G(\R)^+$ and $c_1, \dots, c_h$ are coset 
representatives for the finite set $G(\Q)^+\backslash G(\A_f)/U$.

We now describe (\ref{eq:2.16}) in terms of lattices. 
We say two $O_B$-lattices $\Lambda$ and $\Lambda'$ in $V$ 
are {\it similar} (resp.~{\it strictly similar}), 
denoted $\Lambda\sim \Lambda'$ (resp.~$\Lambda\sim_s \Lambda'$), 
if there is an element $g\in G(\Q)$ 
(resp.~$g\in G(\Q)^+$) such that $\Lambda'=g \Lambda$. Similarly, for each prime $\ell$, write $\Lambda_{\ell}\sim \Lambda'_{\ell}$  
if $\Lambda'_\ell=g \Lambda_\ell$ for some $g\in G(\Q_\ell)$, where $\Lambda_\ell:=\Lambda\otimes {\Z_{\ell}}$ denotes the completion of $\Lambda$ at $\ell$. 
We say $\Lambda$ and $\Lambda'$ are in
the same {\it idealcomplex} if $\Lambda_\ell\sim \Lambda_\ell'$ for all primes $\ell$.
Let 
\[ \grI:=\{\Lambda\subset V\mid \Lambda_\ell \sim \Lambda_{0,\ell} \
\forall\,\ell \} \] 
be the idealcomplex containing the $O_B$-lattice $\Lambda_0$.
The map $c \mapsto c \Lambda_0$, where $c\in G(\A_f)$, induces a
bijection between the double coset space $G(\Q)^+\backslash
G(\A_f)/U $ and the set of strict similitude classes
in $\grI$. 
In particular, the complex Shimura variety $\Sh_U(G,X)_\C$ has
$h(\grI)$ connected components, where $h(\grI)$ is the strict class
number of $\grI$.

Put $\Lambda_i=c_i \Lambda_0$, then $\Lambda_1,\dots, \Lambda_h$ are representatives
of the strict similitude classes of $\grI$. After rescaling 
we may assume that $\psi$ takes $\Z$-values on $\Lambda_i$ for all
$i$. It is easy to verify that $\Gamma_i=\Aut(\Lambda_i,
\psi)=\Gamma_{\Lambda_i}$. Thus, by \eqref{eq:2.16} we get an isomorphism
\begin{equation}
  \label{eq:2.17}
  \Sh_U(G,X)_\C\simeq \coprod_{i=1}^h \calM_{(\Lambda_i,\psi)}.
\end{equation} 

We now describe $\Sh_U(G,X)_\C$ in terms of the Shimura variety for
$(G^0,X)$. By Lemma~\ref{24}, the $G^0(\R)$-conjugacy class of $h_0$
is also $X$. We also have $G(\Q)=G^0(\Q)$. Let $g_1,\dots, g_r$ be coset representatives for the double
coset space $G^0(\A_f)\backslash G(\A_f)/U$. Then elements 
$(1,g_1),\dots, (1,g_r)$
are also coset representatives for $G^0(\A)\backslash G(\A)/\R^\times
K_\infty U$. The map $x\mapsto G(\Q) x (1,g_i) \R^\times
K_\infty U$, for $x\in G^0(\A)$, induces an isomorphism
\begin{equation}
  \label{eq:222}
   G^0(\Q)\backslash G^0(\A)/\R^\times K_\infty U_i^0 \simeq G(\Q)\backslash 
G^0(\A)(1,g_i)\R^\times K_\infty U/  \R^\times K_\infty U,
\end{equation}
where $U_i^0=g_i U g_i^{-1} \cap G^0(\A_f)$. The left hand side of
(\ref{eq:222}) is equal to  $\Sh_{U^0_i}(G^0,X)_\C$. The union of RHS of (\ref{eq:222}) for $i=1,\dots, r$ is equal to $G(\Q)\backslash G(\A)/\R^\times K_\infty U$.  Thus, we have the decomposition
\begin{equation}
  \label{eq:223}
  \Sh_U(G,X)_\C=\coprod_{i=1}^r \Sh_{U^0_i}(G^0,X)_\C. 
\end{equation}
It would be interesting to determine the number $r$ in \eqref{eq:223} 
explicitly. 

\begin{remark}\label{2.11}
  In \cite{kottwitz:isocrystals2}, Kottwitz studied the set $B(G)$ of
  ($F$-)isocrystals 
  with $G$-structure for (not necessarily connected) reductive
  $\Qp$-groups $G$. However, the study of the 
  Kottwitz set $B(G,\{\mu\})$ is limited to 
  connected ones (\cite[Sect.~6]{kottwitz:isocrystals2}). 
  In the case of good 
  reduction and of PEL-type D, Wedhorn \cite{wedhorn:ordinary} 
  introduced the union of
  the sets $B(G^0,\{\mu_i\})$, where $\{\mu_1\},\dots,\{\mu_s\}$ are
  the $G^0(\Qpbar)$-orbits of the $G(\Qpbar)$-conjugacy class of the
  cocharacter $\mu$ over $\Qpbar$ arising from $X$. He showed that
  the $\mu$-ordinary locus, which is the union of points in the
  special fiber whose associated isocrystals are the maximal element in
  $B(G^0,\{\mu_i\})$ (the $\mu_i$-ordinary locus), is dense in
  the special fiber. The description (\ref{eq:223}) shows 
  that the set of the isocrystals of the points in the special fiber 
  are the same as those of the subscheme defined by $(G^0,X)$, 
  which is expected to be $B(G^0, \{\mu\})$.
        
\end{remark}

\section{Arithmetic properties}
\label{sec:03}



\subsection{}

Let $\ul A=(A,\lambda,\iota)$ be a $2dm$-dimensional polarized $O_B$-abelian
variety over any field $K$. Denote by $\Gamma_K=\Gal(K_s/K)$ the
Galois group 
of $K$, where $K_s$ is a separable closure of $K$.
For any prime $\ell\neq \char K$, let $T_\ell=T_\ell(A)$ denote the 
Tate module of $A$, and put $V_\ell:=T_\ell\otimes_{\Z_\ell} \Q_\ell$.
Let
\[ \rho_\ell:\Gamma_K\to \Aut(T_\ell) \]
denote the associated $\ell$-adic representation of the Galois group
$\Gamma_K$. 


Let $\<\,,\>_A:T_\ell(A)\times T_\ell(A^t)\to \Z_\ell(1)$ be the
canonical perfect pairing, where $A^t$ denotes the dual abelian variety of
$A$. The Galois group $\Gamma_K$ acts on $\Z_\ell(1)$ by the
$\ell$-adic cyclotomic character $\chi_\ell$ and 
this pairing is $\Gamma_K$-equivariant.
We may identify $T_\ell(A^t)$ $\Gamma_K$-equivariantly 
with the linear dual $T_\ell(A)^*:=\Hom(T_\ell(A),\Z_\ell(1))$. 
The polarization $\lambda$ induces an alternating 
non-degenerate 
pairing
\[ \<\,,\>=\<\,,\>_\lambda :T_\ell(A)\times T_\ell(A)\to \Z_\ell(1),
\quad \<x,y\>_\lambda:=\<x,\lambda y\>_A. \] 
Denote by $g\mapsto g'$ ($g\in \End(V_\ell)$) 
the adjoint with respect to $\<\,,\>$. Clearly 
$g'=\lambda^{-1} g^t \lambda$, where $g^t\in \End(V_\ell^*)$ is the
pull-back map. 
It follows from 
\begin{equation}
  \label{eq:31}
  \<\rho_\ell(\sigma) x,\rho_\ell(\sigma)
y\>=\rho_\ell(\sigma)(\<x,y\>)=\chi_\ell(\sigma)\<x,y\>, \quad x,y\in
V_\ell, \  \sigma\in \Gamma_K
\end{equation}
that $\rho_\ell(\sigma)' \rho_\ell(\sigma)=\chi_\ell(\sigma)$.
%
Thus, if 
we put $B_\ell:=B\otimes \Q_\ell$ and let
$G_\ell=GU_{B_\ell}(V_\ell)$
be the group of $B_\ell$-linear similitudes on
$V_\ell$,
then the $\ell$-adic representation $\rho_\ell$ factors through this
subgroup 
\[ \rho_\ell:\Gamma_K \to G_\ell(\Q_\ell). \]

We have the following basic properties:

\begin{lemma}\label{31} \ 
Keep the notation as above.
\begin{itemize}
\item[(1)] $T_\ell$ is a free $O_F\otimes \Z_\ell$-module of rank
  $4m$.
\item[(2)]  $V_\ell$ is a free $O_B\otimes \Q_\ell$-module of rank
  $m$.
\item[(3)] If $O_B$ is maximal at $\ell$, i.e.~$O_B\otimes \Z_\ell$
is a maximal order, then $T_\ell$ is an $O_B\otimes \Z_\ell$-module of rank
  $m$.  

\item[(4)] The center of $G_\ell(\Q_\ell)$ is given by
\[ Z(G_\ell(\Q_\ell))=\{x\in (F\otimes \Q_\ell)^\times \mid x^2\in
\Q_\ell^\times \}. \]
\item[(5)] If $m=1$, then the connected component $G^0_\ell$ of
  $G_\ell$ is a torus
  and $G_\ell(\Q_\ell)/G^0(\Q_\ell)$ is a finite  elementary 2-group.
\end{itemize}
\end{lemma}
\begin{proof}
  Statement (1)
  follows from the fact that $\Tr(a;V_\ell;
  \Q_\ell)=4m \Tr_{F/\Q}(a)$ for all $a\in O_F$ and that $O_F\otimes
  \Z_\ell$ is a maximal order. Statements (2) and (3) are obvious.
  Statement (4) follows from a direct computation.
  Statement (5) is proved in \cite[Lemma 3.4]{yu:smf}.  \qed 
\end{proof}

\begin{remark}\label{32}
  
(1)  The group $G_\ell(\Q_\ell)$ 
does not need to be Zariski dense in $G_\ell$. 
For example, if $F_\ell=F\otimes \Q_\ell$ remains a 
field and $B$ splits at the prime over $\ell$,
then $[G_\ell(\Q_\ell):G^0_\ell(\Q_\ell)]=1$ or $2$. 
However, we always have $[G_\ell:G^0_\ell]=2^d$. 
\end{remark}

Recall that an abelian variety $A$ over any field $K$ is said to have
{\it sufficiently many complex multiplications} (smCM) or be 
{\it of CM-type} over $K$ if there
exists a semi-simple commutative $\Q$-subalgebra $L\subset
\End^0_K(A)=\End_K(A)\otimes_\Z \Q$
such that $[L:\Q]=2\dim A$. It is said to have {\it potentially
  smCM} or be {\it potentially of CM-type} if
there exists a finite field extension $K_1$ over $K$ such that
the base change $A_{K_1}$ is of CM type over $K_1$.  

\begin{prop} \label{33}
For any $2d$-dimensional polarized $O_B$-abelian
variety or $O_B$-abelian variety $\ul A$ over $K$, 
the underlying abelian variety $A$ is potentially of CM-type.
\end{prop}

\begin{proof}
  We prove the polarized version since any $O_B$-abelian
  variety admits an $O_B$-linear polarization. We may assume
  that $K$ is finitely generated over its prime field. 
  Let $\Q_\ell[\Gamma_\ell]$ be the subalgebra of $\End(V_\ell)$
  generated by the image $\Gamma_\ell:=\rho_\ell(\Gamma_K)$. 
  Replacing $K$ by a finite extension of $K$, we may assume 
  that $\Gamma_\ell\subset G^0_\ell(\Q_\ell)$ is abelian.
  By the semi-simplicity of Tate modules
  due to Faltings and Zarhin 
  (see \cite{faltings:end} and \cite{zarhin:end}), $\Q_\ell[\Gamma_\ell]$
  is a commutative and semi-simple subalgebra. 
  Let $L$ be a maximal semi-simple commutative
  subalgebra in $\End^0(A)$, then so is $L\otimes \Q_\ell\subset
  \End^0(A)\otimes_\Q \Q_\ell$.
  By
  the theorem of Faltings and Zarhin on Tate's conjecture 
  (see \cite{faltings:end} and \cite{zarhin:end}), we have
  $\End^0(A)\otimes \Q_\ell=\End_{\Q_\ell[\Gamma_\ell]}(V_\ell)$. As
  $\Q_\ell[\Gamma_\ell]$ is commutative and semi-simple, any maximal
  semi-simple commutative subalgebra in 
  $\End_{\Q_\ell[\Gamma_\ell]}(V_\ell)$
  has degree $2 \dim A$ over $\Q_\ell$. This shows $[L:\Q]=2 \dim A$ 
  and finishes the proof of the proposition. \qed    
\end{proof}

\begin{cor}\label{34} Let $\ul A$ be as in Proposition~\ref{33}. 
\begin{itemize}
\item[(1)] If $\char K=0$, then for a suitable finite extension
  $K_1/K$, the base change $\ul A\otimes K_1$ 
  admits a model $\ul A_0$ defined over a
  number field. 

\item[(2)] If $\char K=p>0$, then for a suitable finite extension
  $K_1/K$, the base change $\ul A\otimes K_1$ is isogenous to $\ul A'$
  which admits a model $\ul A_0'$ defined over a finite field. 
  \end{itemize}
\end{cor}
\begin{proof}
  Statement (1) follows from Proposition~\ref{33} and the fact that
  any complex 
 abelian variety of CM type is defined over a number
  field. 
  Statement (2) follows from
  Proposition~\ref{33} and the Grothendieck theorem, stating
  that any isogeny class
  of abelian varieties of CM type 
  in positive characteristic is defined over a
  finite field (see
  \cite{oort:cm} and \cite{yu:cm}). \qed
\end{proof}

\section{Skew-Hermitian $\calO_{\bfB}\otimes_{\Zp} W$-modules}
\label{sec:04}


\subsection{}
\label{sec:41} In this section we investigate basic properties of
related local modules with additional structures. We use the following
notation.  

Let $k$ be an \ac field
of \ch $p>0$. 
Let $W=W(k)$ be the ring of Witt vectors over $k$,
and $B(k):=\Frac W$ the fraction field of $W$. Let $\sigma$ be the
Frobenius map on $W$ and on $B(k)$. 
Let $\bfF$ be a finite field extension of $\Qp$ and $\calO$ the
ring of integers. Let $e$ and $f$ be the ramification index and
inertial degree of $\bfF/\Qp$, respectively, and 
$\pi$ a uniformizer of $\calO$. 
Let $\bfF^{\rm nr}$ denote the maximal unramified subfield extension
of $\Q_p$ in $\bfF$, and put $\calO^{\rm nr}:=\calO_{\bfF^{\rm nr}}$
the ring of integers. 

Let $\bfB$ be a quaternion algebra over $\bfF$ and
$\calO_\bfB$ be a maximal order. As before, we denote by $*$ the
canonical involution. If $\bfB$ is the matrix algebra, then we fix an
isomorphism $\bfB=\Mat_2(\bfF)$ with
$\calO_\bfB=\Mat_2(\calO)$.

Choose an unramified maximal subfield 
$\bfL\subset \bfB$ such that the integral 
ring $\calO_\bfL$ is contained in
$\calO_{\bfB}$. If $\bfB$ is a division algebra, then $\calO_{\bfL}$ is
contained in $\calO_\bfB$ always. In this case we choose    
a presentation
\begin{equation}
  \label{eq:41}
  \calO_\bfB=\calO_{\bfL}[\Pi]=\{a+b\Pi; a, b\in \calO_{\bfL} \}
\end{equation}
with the relations
\begin{equation}
  \label{eq:42}
  \quad \Pi^2=-\pi \quad \text{and} \quad \Pi a =\bar a
   \Pi, \quad \forall\, 
a\in \calO_{\bfL},
\end{equation}
where $a\mapsto \bar a$ is the non-trivial automorphism of $\bfL/\bfF$.

Indeed we first choose a presentation of $\calO_\bfB$ as (\ref{eq:41}) with
relations $\Pi a=\bar a \Pi$ and $\Pi^2=- \pi u$ for some 
element $u\in \calO^\times$. Then 
replacing $\Pi$ by $\alpha \Pi$ for some element 
$\alpha\in \calO_{\bfL}^\times$, one gets the relation $\Pi^2=-\pi$.
Similarly, one
could also choose a presentation but with the relation
$\Pi^2=\pi$ instead. Nevertheless we simply fix the presentation of
$\bfB$ as in (\ref{eq:41}) and (\ref{eq:42}).

We may regard $\calO_\bfB$
as an $\calO$-subalgebra of  $\Mat_2(\calO_\bfL)$ by sending
\[ \Pi \mapsto 
\begin{pmatrix}
  0 & -1 \\ \pi & 0 
\end{pmatrix}
,\quad \text{and} \quad a \mapsto 
\begin{pmatrix}
  a & 0 \\ 0 & \bar a 
\end{pmatrix}, \quad \forall\, a\in \calO_\bfL. \]
Thus, 
\[ \calO_\bfB=\left\{a+b\Pi=
\begin{pmatrix}
  a & -b \\ \pi \bar b & \bar a 
\end{pmatrix} {\Big |}\,  a, b\in \calO_\bfL \right \}\subset
\Mat_2(\calO_\bfL).   
\]
We have the following properties 
\[ \Pi^*=-\Pi, \quad \text{and} \quad (a+b\Pi)^*=\bar a-b \Pi. \]

\subsection{}
\label{sec:42}
Let $\Sigma_0:=\Hom_{\Zp}(\calO^{\rm nr}, W)$ be the set of embeddings
of $\calO^{\rm nr}$ into $W$. Write $\Sigma_0=\{\sigma_i\}_{i\in
  \Z/f\Z}$ in the way that $\sigma \sigma_i=\sigma_{i+1}$ for all
$i\in \Z/f\Z$. 
For any $W$-module $M$ together with a $W$-linear action of
$\calO^{\rm nr}$, write
\begin{equation}
  \label{eq:43}
  M^i:=\{x\in M \large\mid ax=\sigma_i(a) x,\ \forall\, a\in
  \calO^{\rm nr} \}
\end{equation}
for the $\sigma_i$-component, and we have the decomposition
\begin{equation}
  \label{eq:44}
  M=\bigoplus_{i\in \Z/f\Z} M^i.
\end{equation}
If $V$ is a finite-dimensional $k$-vector space with a $k$-linear
action of $\F_{p^f}$, we write
\begin{equation}
  \label{eq:45}
  V=k^{m_0} \oplus \dots \oplus k^{m_{f-1}}
\end{equation}
for the decomposition $V=\oplus V^i$ as in (\ref{eq:44}) with
$m_i=\dim_k V^i$ for all $i\in \Z/f\Z$. 

Let $P(T)\in \calO^{\rm nr}[T]$ be the minimal polynomial of $\pi$;
one has $\calO=\calO^{\rm nr}[\pi]=\calO^{\rm nr}[T]/P(T)$. For any
$i\in \Z/f\Z$, set $W^i:=W[T]/(\sigma_i(P(T))$ and denote again by $\pi$
the image of $T$ in $W^i$. Each $W^i$ is a complete discrete valuation
ring and one has the decomposition 
\begin{equation}
  \label{eq:46}
 \calO \otimes_{\Zp} W =\prod_{i\in \Z/f\Z} W^i. 
\end{equation}
The action of the Frobenius
map $\sigma$ on  $\calO \otimes_{\Zp} W$ through the right factor 
gives a
map $\sigma: W^i\to W^{i+1}$ which sends $a$ to $\sigma(a)$ for $a\in
W$ and $\sigma(\pi)=\pi$.
If $M$ is an $\calO\otimes_{\Zp} W$-module, then we have
the decomposition (\ref{eq:44}) with each component $M^i$ a $W^i$-module.  
Note that the structure of $M$ as an $\calO\otimes_{\Zp} W$-module is
determined by the structure of each $M^i$ as a $W^i$-module for all
$i\in \Z/f\Z$. 

Let $\bfL^{\rm nr}$ be the maximal unramified extension of $\Qp$ in
$\bfL$, and let $\calO_{\bfL^{\rm nr}}$ be the ring of integers. 
Let $\Sigma:=\Hom_{\Zp}(\calO_{\bfL^{\rm nr}}, W)$ be the set of
embeddings 
of $\calO_{\bfL^{\rm nr}}$ into $W$. Write $\Sigma=\{\tau_j\}_{j\in
  \Z/2f\Z}$ in the way that $\sigma \tau_j=\tau_{j+1}$ and
$\tau_{j}|_{\calO^{\rm nr}}=\sigma_{i}$ where $i=j\! \mod f$ 
for all $j$. 
The Galois group $\Gal(\bfL/\bfF)$ acts on the set 
$\Sigma$ by composing with the conjugate: $\bar
\tau_i(x):=\tau_i(\bar x)$. One has $\bar \tau_i=\sigma^f
\circ \tau_i=\tau_{i+f}$.  
For any $W$-module $M$ together with a $W$-linear action of
$\calO_{\bfL}$, write
\begin{equation}
  \label{eq:47}
  M^j:=\{x\in M \large\mid ax=\tau_j(a) x,\ \forall\, a\in
  \calO_{\bfL^{\rm nr}} \}
\end{equation}
for the $\tau_j$-component, and we have the decomposition
\begin{equation}
  \label{eq:48}
  M=\bigoplus_{j\in \Z/2f\Z} M^j.
\end{equation}
Similarly, each $M^j$ is a $W^i$-module where $i=j\! \mod f$ and the
structure of $M$ as an $\calO_{\bfL}\otimes_{\Zp} W$ is determined by
the structure of each $M^j$ as a $W^i$-module for all $j\in \Z/2f\Z$.

\subsection{Finite $\calO_{\bfB}\otimes_{\Zp} W$-modules}
\label{sec:43}

Suppose $M$ is a finite $W$-module together with a $W$-linear action by
$\calO_\bfB$. 

Suppose $\bfB$ is the matrix algebra. One has the
decomposition   
\begin{equation}
  \label{eq:485}
   M=e_{11}M\oplus e_{22} M=:M_1\oplus M_2,
\end{equation}
where $e_{11}$ and $e_{22}$ are standard idempotents of
$\Mat_2(\calO)$, and $M_1$ and $M_2$ are finite $W$-modules with a
$W$-linear action by $\calO$ with $\rank_W M_1=\rank_W M_2$. 
The Morita equivalence states that the
module $M$ is uniquely determined by the $\calO\otimes_{\Zp} W$-module
$M_1$. Furthermore, the structure of $M_1$ as an $\calO\otimes_{\Zp}
W$-module is given by
its decomposition 
$M_1=\oplus_{i\in \Z/f\Z} M_1^i$ as $W^i$-submodules and the
$W^i$-module structure 
of each component $M_1^i$.
This describes finite $\calO_\bfB\otimes W$-modules
when $\bfB$ is the matrix algebra. In particular, if $M$ is free
as a $W$-module, then $M$ is uniquely determined by the numbers 
$\rank_{W^i} M^i$, which equals $2\rank_{W^i} M^i_1$,  up
to isomorphism (and these ranks can be arbitrary even non-negative 
integers). The module $M$ is a
free $\calO_\bfB\otimes W$-module if and only if 
the ranks $\rank_{W^i} M^i$ are constant. 

Suppose $\bfB$ is the division algebra. Write
$\calO_\bfB=\calO_{\bfL}[\Pi]$ as in Section~\ref{sec:41}. The action by
$\calO_{\bfL^{\rm nr}}$ gives the decomposition 
\begin{equation}
  \label{eq:488}
  M=\bigoplus_{j\in \Z/2f \Z} M^j
\end{equation}
with each component $M^j$ a finite $W^i$-module where $i= j \mod f$. Moreover, one has
\[ \Pi:M^j \to M^{j+f},\quad \Pi^2=-\pi\quad \text{on $M^j$} \]
for all $j\in \Z/2f\Z$. 
To see this, if $x\in M^j$, then for $a\in \calO_{\bfL^{\rm nr}}$,
\[ a\cdot \Pi x=\Pi\, \bar a \cdot x= \Pi\, (\tau_j(\bar a))x
=\Pi\, \tau_{j+f}(a)x=\tau_{j+f}(a)\Pi x. \] 
The map $\Pi$ induces an isomorphism $M^j\otimes \Frac W^i \simeq M^{j+f} \otimes \Frac W^i$. Thus, 
\begin{equation}
  \label{eq:49}
  \rank_{W^i}
M^j=\rank_{W^i} M^{j+f}, \quad \forall\, j\in \Z/2f\Z.
\end{equation}
This is the only constraint for $M$ to be an $\calO_\bfB\otimes
W$-module. Put
\begin{equation}
  \label{eq:410}
  a_j:=\dim_k M^j/\Pi M^{j-f}.
\end{equation}
If $M$ is free as a $W$-module, then $M$ is uniquely determined by  
the numbers $\{a_j\}_j$ up to isomorphism. The number $a_j$ can
be arbitrary between $0$ and $\rank_{W^i} M^j$ subject to the
condition $a_j+a_{j+f}=\rank_{W^i} M^j$, 
Furthermore, $M$ is a free $\calO_\bfB\otimes W$-module
if and only if the numbers $a_j$ are constant.

\subsection{Skew-Hermitian $\calO_{\bfB}\otimes_{\Zp} W$-modules}
\label{sec:44}

In this subsection we assume that $p\neq 2$.
Let $M$ be a finite non-degenerate skew-Hermitian $\calO_\bfB$-module
over $W$, that is, it is a finite and free $W$-module  
with a $W$-linear action of 
$\calO_\bfB$ and together with an alternating non-degenerate bilinear
pairing
\[ \psi:M\times M\to W \]
satisfying the condition 
\begin{equation}
  \label{eq:412}
   \psi(bx, y)=\psi(x,b^* y), \quad \forall\, x,y\in M, b\in
   \calO_\bfB. 
\end{equation}
(non-degeneracy here means that the induced map $M\to M^t:=\Hom_W(M,W)$
is injective). If the pairing $\psi$ is perfect, we call $M$ {\it
  self-dual}. 

Suppose $\bfB$ is the matrix algebra. Let $C:=\left[\begin{smallmatrix}
0 & 1 \\ -1 & 0 \end{smallmatrix}\right]$ be the Weil element. Put
$\varphi(x,y):=\psi(x,Cy)$. Then the pairing 
$\varphi:M\times M\to W$ is symmetric; this follows from the property $C^*=-C$.
One also has $C^{-1} a^* C=a^t$ for all $a\in
\bfB=\Mat_2(\bfF)$ and the idempotents $e_{11}$ and $e_{22}$ are fixed by the adjoint with respect to $\varphi$.
Thus, the decomposition $M=M_1\oplus M_2$ in (\ref{eq:485}) respects the
pairing $\varphi$. 
By the Morita equivalence, we may
describe the symmetric $\calO\otimes W$-module $M_1$. 
Note that the condition (\ref{eq:412}) becomes equivalent to the condition 
$\varphi(ax,y)=\varphi(x,ay)$ for $x,y\in M_1$ and $a\in \calO$. The latter implies that $\varphi(M^i_1,M_1^{i'})=0$
for $i\neq i'$ in $\Z/f\Z$, where the $M^i_1$'s are the components in 
the decomposition (\ref{eq:44}). 
Consider the restriction of $\varphi$ to each component $M_1^i$.
Then there is a unique $W^i$-bilinear pairing 
\[ \varphi_i:M_1^i\times M^i_1 \to \calD^{-1}_{W^i/W} \]
such that $\varphi=\Tr_{W_i/W} \varphi_i$ on each $M_1^i$, 
where $\calD^{-1}_{W^i/W}$ is
the inverse different of $W^i$ over $W$. 
Thus, one can describe
the symmetric $W^i$-modules $M_1^i$ instead.
The classification of symmetric local lattices  
is well known; see O'Meara \cite[\S 92]{omeara:book}. Since $W^i$ is non-dyadic, 
each $M_1^i$ has an orthogonal basis for $\varphi_i$ (see loc.~cit.). 
As our ground field $k$ is algebraically
closed of \ch $p\neq 2$, one has $W^{i\times}=(W^{i\times})^2$. Thus, if $M_1^i$ is self-dual (with respect to the values
in $\calD^{-1}_{W^i/W}$), then the  
isomorphism class of $M_1^i$ is uniquely determined by 
its rank $\rank_{W^i} M_1^i$. Note that $M$ is self-dual with
respect to the pairing $\psi$ if and only if each submodule 
$M_1^i$ is self-dual with respect to the pairing $\varphi_i$ (with
values in $\calD^{-1}_{W^i/W}$).
Following from this, we conclude the following result. 


\begin{lemma}\label{41}
  Assume $\bfB$ is the matrix algebra. Any two self-dual
  skew-Hermitian $\calO_{\bfB}\otimes W$-modules $M$ and $N$ are
  isomorphic if and only if $\rank_{W^i} M^i=\rank_{W^i} N^i$ for all
  $i\in \Z/f\Z$.   
  Moreover, for any given non-negative even 
  integers $n_i$
  for $i\in \Z/f\Z$, there is a unique  
  self-dual skew-Hermitian $\calO_{\bfB}\otimes W$-module $M$ up to isomorphism 
  such that $\dim_{W^i} M^i=n_i$ for all $i\in \Z/f\Z$. \qed
\end{lemma}

Suppose $\bfB$ is the division algebra. Let $M=\oplus_{j} M^j$ be the
decomposition by the action of $\calO_{\bfL^{\rm nr}}$ as in 
(\ref{eq:48}). It is easy to
see using (\ref{eq:412}) that $\psi(M^{j_1},M^{j_2})=0$ if
$j_1-j_2\neq f$ in $\Z/2f\Z$ and hence $\psi$ is determined by its
restriction  
\[ \psi: M^j \times M^{j+f} \to W \]
for $0\le j<f$. Note that $\rank_{W^i}
M^j=\rank_{W^i} M^{j+f}$ (\ref{eq:49}), where $i:=j\!\mod f$.
Let $\calD^{-1}_{W^i/W}=W^i \delta_i$, where $\delta_i$ is a
generator. Then there is a unique $W^i$-bilinear 
pairing \[ \psi_i:M^j\times M^{j+f}\to W^i \] 
such that 
$\psi=\Tr_{W^i/W} (\delta_i \psi_i)$. 
Clearly, the module $M$ is self-dual with respect to the pairing $\psi$ 
if and only if each $\psi_i$ is a perfect pairing. 

Now we only consider the case where 
$M$ is {\it self-dual}. For any $x, y\in M^j$,
one easily sees 
\[ \psi_i(x,\Pi y)=\psi_i(\Pi^* x,y)=\psi_i(y, \Pi x), \]
so the pairing 
\[ \varphi_i:M^j\times M^j\to W^i, \quad \varphi_i(x,y):=\psi_i(x,\Pi
y), \]
is symmetric. Put $\ol
{M^j}:=M^j/\pi M^j$. 
Then $\psi_i$ induces the perfect pairing\footnote{To ease our
  notation, we still denote the induced pairing by $\psi_i$.} 
$\psi_i:  \ol {M^j}\times  \ol {M^{f+j}}\to k$.
Recall $a_j:=\dim_k {M^j}/\Pi M^{j+f}$.  
From the isomorphism
\[ \Pi: M^{j+f}/\Pi M^{j} \simeq \Pi M^{j+f}/\pi M^{j},\]
we get 
\[ a_j+a_{j+f}=\dim_k \ol {M^j}=\dim_k \ol{M^{j+f}}. \]

\begin{lemma}\label{42} Assume that $\bfB$ is a division algebra. Let
  $M=\oplus_{j\in \Z/2f\Z} M^j$ be a self-dual skew-Hermitian
  $\calO_{\bfB}\otimes_{\Zp} W$-module.  
  For each $j\in \Z/2f\Z$, 
there are $W^i$-bases 
\[ \{x^j_1,\dots, x^j_{a_j+a_{j+f}}\}, 
\quad \{x^{j+f}_1,\dots, x^{j+f}_{a_j+a_{j+f}}\} \]
for $M^j$ and $M^{j+f}$, respectively, where the positive integers
$a_j$ and $a_{j+f}$ are as above, such that 
\begin{equation}
  \label{eq:414}
  \begin{cases}
   \Pi(x^j_r)=x^{j+f}_r,  & \forall\, 1\le r\le a_j, \\  
 \Pi(x^{j+f}_{a_j+r})=x^j_{a_j+r}, & \forall\, 1\le r \le
a_{j+f},
  \end{cases}
\end{equation}
and for $1\le r, s \le a_j+a_{j+f}$, one has
\begin{equation}
  \label{eq:415}
  \psi_i(x^j_r,x^{j+f}_s)=\delta_{r,s}.
\end{equation}
\end{lemma}
\begin{proof}
Consider the induced symmetric pairing $\varphi_i:\ol {M^j}\times \ol
{M^j} \to k$. Since $\Pi \ol {{M^j}}$ and $\Pi \ol {M^{j+f}}$ are 
mutually orthogonally complemented with respect to the pairing $\psi_i$, 
one obtains a {\it non-degenerate} symmetric pairing
\begin{equation}
  \label{eq:416}
  \varphi_i :M^j/\Pi M^{j+f}\times M^j/\Pi M^{j+f} \to k.
\end{equation}
We prove the statement by induction on the rank of $M^j$ 
(or $M^{j+f}$). 
Suppose $a_j>0$, using (\ref{eq:416}) 
we may choose an element $x^j_1\in M^j$ such that
$\varphi_i(x^j_1,x^j_1)=1$, because
$W^{i \times}=({W^{i \times}})^2$. Put $x^{j+f}_1:=\Pi x^j_1\in
M^{j+f}$. Then we have 
\[ M^j\oplus M^{j+f}=N\oplus N^{\bot}, \]
where $N$ is the $W^i$-submodule generated by $x^j_1$ and $x^{j+f}_1$
and $N^{\bot}$ is the orthogonal complement of $N$.  
Clearly $N$ is stable under the $\calO_\bfB$-action and hence so is 
$N^{\bot}$. If $a_j=0$ then $a_{j+f}>0$ and 
we do the same for $M^{j+f}$.
As $N^{\bot}$ has lower rank, by induction we can choose bases $\{x^j_1,\dots,
x^j_{a_j+a_{j+f}}\}$ and $\{x^{j+f}_1,\dots, x^{j+f}_{a_j+a_{j+f}}\}$ 
for $M^j$ and $M^{j+f}$, respectively, satisfying (\ref{eq:414}) and 
(\ref{eq:415}). 
This completes the proof of the lemma. \qed


\end{proof}

We obtain the following classification result. 

\begin{cor}\label{43}
  Assume that $\bfB$ is the division algebra. Any two self-dual
  skew-Hermitian $\calO_{\bfB}\otimes W$-modules $M$ and $N$ are
  isomorphic if and only if \[ \dim_k M^j/\Pi M^{j+f}=\dim_k N^j/\Pi
  N^{j+f}, \quad \text{for all
  $j\in \Z/2f\Z$}.\] 
  Moreover, for any given non-negative integers $a_j$
  for $j\in \Z/2f\Z$, there is a unique up to isomorphism 
  self-dual skew-Hermitian $\calO_{\bfB}\otimes W$-modules $M$ such
  that $a_j=\dim_k M^j/\Pi M^{j+f}$for all $j\in \Z/2f\Z$. \qed
\end{cor}

\section{Quasi-polarized \dieu $\calO_{\bfB}$-modules}
\label{sec:05}

\subsection{}
\label{sec:51}

We keep the notation of Section~\ref{sec:04}. In particular, we have the decomposition $\calO\otimes_{\Zp} W=\oplus_{i\in \Z/f\Z} W^i$ as in \eqref{eq:46}.
All \dieu modules over $k$ in this paper are assumed to be finite and free as
$W$-modules. For basic theory of \dieu modules, we refer to 
Manin~\cite{manin:thesis} and Zink~\cite{zink:cartier}.
The Frobenius and Verschiebung operators of a \dieu module are denoted
by $\sfF$ and $\sfV$, respectively.
   

By a {\it \dieu $\calO_\bfB$-module} over $k$ we mean a \dieu module $M$ over $k$
together with a ring monomorphism $\calO_\bfB \to \End_{\rm DM}(M)$ of
$\Zp$-algebras. Recall that a {\it quasi-polarization}\footnote{The terminology "quasi-polarization" has been adopted in several papers of Oort e.g.~in \cite{norman-oort}. A reason for calling it quasi-polarization is that there is no notion of positivity for a biextension on a $p$-divisible group, unlike abelian varieties. Here we simply use the same terminology. So a quasi-polarization on a $p$-divisible group is really an {\it isogeny} to its Serre dual, not a quasi-isogeny.} 
on a \dieu module $M$
is a non-degenerate alternating $W$-bilinear form
\[ \<\ ,\, \>:M\times M\to W \]
such that $\<\sfF x,y\>=\<x,\sfV y\>^{\sigma}$ for all $x, y\in M$. When the pairing $\<\ ,\, \>$ is perfect, we call it {\it separable}\footnote{Some authors also call this {\it principal}.}.  
A quasi-polarization is
called {\it $\calO_\bfB$-linear} if it satisfies the 
condition 
\begin{equation}
  \label{eq:51}
  \<bx, y\>=\<x,b^*y\>, \quad \forall\, x,y\in M, \ b\in \calO_\bfB.
\end{equation}
A (resp.~{\it separably}) {\it  quasi-polarized \dieu $\calO_\bfB$-module} is a \dieu
$\calO_\bfB$-module 
$M$ together with an (resp.~separable) $\calO_\bfB$-linear quasi-polarization. 


We also recall that a {\it quasi-polarized $p$-divisible $\calO_\bfB$-module} 
is a triple $(H,\lambda,\iota)$ where $H$ is a $p$-divisible group,
$\iota:\calO_\bfB \to \End(H)$ is a ring monomorphism of $\Zp$-algebras 
and $\lambda:H\to H^t$
is a quasi-polarization 
(i.e.~an isogeny $\lambda$ satisfying $\lambda^t=-\lambda$) 
such that  $\lambda\circ \iota(b^*)=\iota(b)^t \circ \lambda$ 
for all $b\in \calO_\bfB$, where $H^t$ is the Serre dual of $H$. 
For a quasi-polarized $p$-divisible 
$\calO_\bfB$-module $(H,\lambda,\iota)$ over $k$, 
the associated (covariant) \dieu 
module $M=M(H)$ together with the additional
structures is a quasi-polarized \dieu $\calO_\bfB$-module over $k$. 

Clearly these notions can be defined for orders in  
general semi-simple $\Qp$-algebras with involutions (i.e.~for
general PEL data); cf. \cite[Section 2.1]{yu:cm}.   \\

Let $(M, \<\,,\>)$ be a (not necessarily separably) 
quasi-polarized \dieu $\calO_\bfB$-module. 

Suppose $\bfB$ is the matrix algebra. As in 
Section~\ref{sec:44}, we define the
symmetric pairing $(x,y):=\<x,Cy\>$, where $C$ is the Weil element. 
Then we obtain the decomposition
$M=e_{11}M\oplus e_{22} M=:M_1\oplus M_2$, which  
respects the pairing $(\,,\,)$. 
By the
Morita equivalence, one may consider the {\it
anti-quasi-polarized} \dieu $\calO$-module $M_1$ instead, i.e. for all $a\in \calO$ and $x,y\in M_1$, one has the
properties 
\begin{equation}\label{eq:52}
  (y,x)=(x,y),\quad (ax,y)=(x,ay), \ \text{and\ } (\sfF x,y)=(x,\sfV y)^\sigma. 
\end{equation}
 
Let
$M_1=\oplus_{i\in \Z/f\Z} M^i_1$ be the decomposition by the action of
$\calO^{\rm nr}$ as in \eqref{eq:44}. Each component $M_1^i$ is then a $W^i$-module. Then one has 
\[ \sfF : M^i_1 \to M^{i+1}_1, \quad \sfV : M^{i+1}_1 \to M^i_1. \]
It follows that the ranks $\rank_{W^i} M_1^i$ (or equivalently $\rank_{W^i} M^i$) for $i\in \Z/f\Z$ are constant. As a result, the module $M_1$ (or $M$) is free as an $\calO\otimes_{\Zp} W$-module. (This only uses the property that $M$ is a \dieu $\calO$-module.)
One also has the orthogonal property $(M_1^{i_1}, M^{i_2}_1)=0$ if $i_1\neq i_2$ in $\Z/f\Z$.


\begin{lemma}\label{51}
  Assume that $\bfB$ is the matrix algebra and $p\neq 2$. 
  If $M$ and $N$ are two
  {\rm separably} quasi-polarized \dieu $\calO_\bfB$-modules of the same rank,    
  then $M$
  and $N$ are isomorphic as skew-Hermitian $\calO_\bfB\otimes
  W$-modules.
\end{lemma}
\begin{proof}
  This follows from the fact that $M_1$ (or $M$) is
free as an $\calO\otimes_{\Zp} W$-module and Lemma~\ref{41}. \qed
\end{proof}

Suppose $\bfB$ is the division algebra. Let $M=\oplus_{j\in \Z/2f\Z}
M^j$ be the decomposition as in (\ref{eq:488}). Then we have 
\begin{equation}\label{eq:shift}
\sfF : M^j\to M^{j+1}, \quad \sfV :M^{j+1}\to M^{j},
\end{equation}
and $\< M^{j_1}, M^{j_2}\>=0$ if $j_1-j_2\neq f$ in $\Z/2f\Z$.  
Similarly, it follows from \eqref{eq:shift} that 
the ranks $\rank_{W^i} M^j$ for $j \in \Z/2f\Z$ are constant, and that
$M$ is free as an $\calO_{\bfL}\otimes_{\Zp} W$-module. (This
uses only the property that $M$ is a \dieu $\calO_\bfL$-module.)

\subsection{}
\label{sec:52}

From now on until Section~\ref{sec:09} we
assume that $\rank_W M=m [\bfB:\Qp]$ for some positive integer $m$.
As in \eqref{eq:275}, let $\char (a)\in \Zp[T]$ be the reduced characteristic
polynomial of $a$ from $\bfB$ to $\Qp$, which is of degree $[\bfB:\Q_p]/2$.
We say a \dieu $\calO_{\bfB}$-module $M$ satisfies the {\it determinant
  condition} 
if one has the equality of the following characteristic polynomials, taken in $k[T]$, cf. (\ref{eq:25})   
\begin{equation}
  \label{eq:53}
  {\rm (K)} \quad \char (\iota(a) | M/\sfV M)= \char(a)^m, \quad
  \forall\, a\in \calO_\bfB.  
\end{equation} 
If we set $d=[\bfF:\Qp]$, then the above polynomials are of degree $2dm$.  

\begin{lemma}\label{52}
Let $M$ be a \dieu $\calO_{\bfB}$-module. 
\begin{enumerate}
  \item If $\bfB$ is the matrix algebra, then $M$ satisfies the
    determinant condition (K) if and only if for all $i\in \Z/f\Z$,
    one has $\dim_k (M_1/\sfV M_1)^i=em$.
  \item If $\bfB$ is the division algebra, then $M$ satisfies the
    determinant condition (K) if and only if for all $j\in \Z/2f\Z$,
    one has $\dim_k (M/\sfV M)^j=em$.
\end{enumerate}
\end{lemma}
\begin{proof}
  (1) Using the Morita equivalence, $M$ satisfies the condition (K) if
  and only if 
  $M_1$ satisfies (K) for all $a\in \calO$. Choose an algebraic closure
  $B(k)^{\rm alg}$ of $B(k):=W(k)[1/p]$ and put
  $\Sigma_{\bfF}:=\Hom(\bfF, B(k)^{\rm alg})$. For $a\in \calO$,  the
left hand side of the equation (\ref{eq:53}) is 
\[ \prod_{i\in \Z/f\Z} (T-\wt
    \sigma_i(a))^{\dim_k (M_1/\sfV M_1)^i} , \]
where $\wt \sigma_i
  \in \Sigma_\bfF$ is any lift of $\sigma_i\in \Sigma_0$. 
The right hand side of the equation
  (\ref{eq:53})
  is equal to 
\[ \prod_{ \sigma\in \Sigma_\bfF} (T-
    \sigma'(a))^m=\prod_{i\in \Z/f\Z} (T-\wt
    \sigma_i(a))^{em}\ . \]
Therefore, the condition (K) is satisfied if and only if 
$\dim_k (M_1/\sfV M_1)^i=em$
for all $i\in \Z/f\Z$. 

(2) As $\bfB$ is generated by the element $\Pi$ over $\bfL$, by 
\cite[Proposition 2.1.1]{hartwig:kr}
one may check the equality (\ref{eq:53}) for $a=\Pi$ and all $a\in
\calO_\bfL$. Since 
$\iota(\Pi)$ on $M/\sfV M$ is nilpotent, its \ch polynomial is $T^{2dm}$,
where $d=[\bfF:\Qp]$. Since $\Pi$ satisfies $\Pi^2+\pi=0$, the (reduced) \ch polynomial of $\Pi$ 
is $(T^2+\pi)^{d}$, whose image in $k[T]$ is equal to $T^{2d}$. 
This verifies (K) for $a=\Pi$. We then must show that (K) holds for all $a\in \calO_\bfL$ if and only if 
$\dim_k (M/\sfV M)^j=em$ for all $j\in \Z/2f\Z$.
But the same proof of (1) also proves this. \qed

\end{proof}
  

\begin{lemma}\label{53}
  Suppose $M$ is a {\it separably} 
  quasi-polarized \dieu $\calO_\bfB$-module.
  \begin{enumerate}
  \item If $\bfB$ is the matrix algebra, then for any $i\in \Z/f\Z$,
    one has $\dim_k (M_1/\sfV M_1)^i=em$.
  \item If $\bfB$ is the division algebra, then for any $i\in \Z/f\Z$,
    one has $\dim_k (M/\sfV M)^i=2em$ (recall that $(M/\sfV M)^i$ denotes the
    $\sigma_i$-component of $M/\sfV M$).
  \end{enumerate}
\end{lemma}
\begin{proof}
(1) The proof is similar to that of \cite[Lemma 2.6 (2)]{yu:reduction}.
Consider $M_1$ as a separably anti-quasi-polarized \dieu
  $\calO$-module. 
 Since the restriction of $(\ ,\, )$ to each component $M^i_1$ is
 perfect, for any $W$-basis $\{x_r\}$ of $M^i_1$, the discriminant
 $\det (x_r,x_s)$ is a unit. Consider the map $\sfV :M^{i+1}_1\to
 M^i_1$ of modules of $W$-rank $2em$; one has
 $M^i_1/\sfV M^{i+1}_1=(M_1/\sfV M_1)^i$. 
 By the elementary divisor theorem,
 there exist bases $\{x_1,\dots, x_{2em}\}$ and $\{y_1,\dots,
 y_{2em}\}$ of $M^{i+1}_1$ and $M^i_1$, respectively, such that for
 all $r$, one has $Vx_r=p^{n_r} y_r$ for some integer $n_r\ge 0$. Then 
\[ p^{2em} \det (x_r,x_s)^{\sigma^{-1}}=\det(\sfV x_r,\sfV x_s)
=\det p^{n_r+n_s}(y_r,y_s)=p^{2\sum n_r}\det (y_r,y_s). \]
It follows that $\sum n_r=em$ and $\dim_k (M_1/\sfV M_1)^i=em$.

(2) Consider $M$ as a separably quasi-polarized \dieu
  $\calO$-module (the Hilbert-Siegel analogue). The same proof of (1) with $(\ , \,)$ and $M_1$ replaced by $\<\ ,\,\>$ and $M$, respectively, will prove the result. \qed  
\end{proof}

\begin{remark}\label{54}
According to Lemmas~\ref{52} and \ref{53}, 
if $\bfB$ is the matrix algebra, then any
separably quasi-polarized \dieu $\calO_\bfB$-module 
satisfies the determinant condition. 
However, in the case where $\bfB$ is the division algebra, 
there are a few
possibilities for the numbers $\dim_k (M/\sfV M)^j$. Therefore, 
the determinant condition
would impose a further condition for separably quasi-polarized \dieu
$\calO_\bfB$-modules. 
\end{remark}

\subsection{}
\label{sec:53}

We discuss the relationship between the numbers $a_j:=\dim_k (M/\Pi
M)^j$ and the numbers $\dim_k (M/\sfV M)^j$ for a \dieu $\calO_B$-module $M$ when $\bfB$ is the division
algebra. Put 
\begin{equation}
  \label{eq:535}
  c_j:=\dim_k (M/\sfV M)^j\quad \text{ for $j\in \Z/2f\Z$}.
\end{equation}
Write 
$\sfV _j: M^{j+1}\to M^j$ for the restriction of $\sfV $ on $M^{j+1}$, and
$\Pi_j:M^j\to M^{j+f}$ for that of $\Pi$ on $M^j$. We have the
commutative diagram 
\begin{equation}
  \label{eq:54}
  \begin{CD}
    M^j @<{\sfV _j}<< M^{j+1} \\
    @V{\Pi_j}VV @V{\Pi_{j+1}}VV \\
    M^{j+f} @<{\sfV _{j+f}}<< M^{j+f+1}.
  \end{CD}
\end{equation}

Let $\ord$ be the normalized valuation on $W^i$, that is,
one has $\ord(\pi)=1$. 
Let $\ord \det \Pi_j$ denote the
valuation of $\det(A_j)$, where $A_j$ is the matrix representing
the map $\Pi_j$ with respect to a set of $W^i$-bases for $M^j$ and 
$M^{j+f}$ respectively; this is well-defined. Similarly we define
$\ord \det \sfV _j$ for  
suitable bases of $M^{j+1}$ and $M^{j}$.  
It is easy to
see that 
\begin{equation}
  \label{eq:55}
  \ord \det \sfV _j=c_j, \ \text{and\ } \ord \det \Pi_{j}=a_{j+f} , \quad
  \forall\, 
  j\in \Z/2f\Z. 
\end{equation}
As $\Pi_{j+1}=\sfV _{j+f}^{-1} \circ \Pi_{j} \circ \sfV _j$, one has the relation
\[ a_{j+f+1}=a_{j+f}+c_j-c_{j+f}, \]
or equivalently
\begin{equation}\label{eq:56}
  a_{j+1}=a_{j}+c_{j+f}-c_{j}, \quad \forall\, j\in \Z/2f\Z. 
\end{equation}
Since $\rank_{W^i} M^j=2m$, it follows from $\Pi_j\circ
\Pi_{j+f}=-\pi$ that 
\begin{equation}
  \label{eq:57}
  a_j+a_{j+f}=2m. 
\end{equation}
Since $a_j$'s are integers between $0$ and $2m$, it
follows from (\ref{eq:56}) that 
\begin{equation}
  \label{eq:58}
  {\Big |}\sum_{i=j}^{j+r} (c_{i+f}-c_i) {\Big |}\le 2m, \quad
   \forall\, j\in 
   \Z/2f\Z,\  0\le r\le f-1.
\end{equation}
On the other hand, the collection $\{c_j\}$ satisfies the condition 
\begin{equation}
  \label{eq:59}
  \sum_{j\in \Z/2f\Z} c_j=2dm. 
\end{equation}
We say $M$ is {\it separably quasi-polarizable} if it admits a separable $\calO_\bfB$-linear quasi-polarization. 
If $M$ is separably quasi-polarizable, then by Lemma~\ref{53} (2), one has
\begin{equation}
  \label{eq:510}
  c_j+c_{j+f}=2em, \quad \forall\, j\in \Z/2f\Z,
\end{equation}
because $(M/\sfV M)^i=(M/\sfV M)^j\oplus (M/\sfV M)^{j+f}$ and $\dim_k (M/\sfV M)^i=2em$, with $i=j \mod f$.


\begin{lemma}\label{55} 
Suppose that $\bfB$ is the division algebra and let the notation be as above. 
The sets of numbers $\{a_j\}$ and $\{c_j\}$ 
satisfy the conditions
  (\ref{eq:56})--(\ref{eq:59}). Moreover, if $M$ is separably
  quasi-polarizable, then one has in addition the condition
  (\ref{eq:510}). 
\end{lemma}

It is interesting to know whether the necessary conditions in Lemma~\ref{55} are also sufficient, in the sense that these invariants can be realized by a (separably quasi-polarizable) \dieu $\calO_\bfB$-module. 




\begin{prop}\label{56} 
Let $M$ be a \dieu $\calO_\bfB$-module of $W$-rank $4dm$, where $d=[\bfF:\Qp]$. 
\begin{enumerate}
\item Suppose $\bfB$ is the matrix algebra.
  \begin{enumerate}
  \item The module $M$ is free as an
  $\calO_\bfB\otimes W$-module.
\item The module $M$ is separably quasi-polarizable if and only 
  if it satisfies the determinant condition.
  \end{enumerate}

\item Suppose $\bf B$ is the division algebra. 
  \begin{enumerate}
  \item If $M$ satisfies the
  determinant condition, then it is free as an
  $\calO_\bfB\otimes W$-module.
  \item If $M$ is separably quasi-polarizable and free 
  as an $\calO_\bfB\otimes W$-module, then it satisfies the
  determinant condition.
  \end{enumerate}

\end{enumerate}
\end{prop}
\begin{proof}
(1) Part (a) is proved in the paragraph before Lemma~\ref{51}. 
For part (b), the only if part follows from
Lemmas~\ref{52} and \ref{53}. For the other direction, we refer to the
discussion in \cite[Lemma 2.6]{yu:reduction}.

(2) (a) Since $M$ satisfies the
  determinant condition, by Lemma~\ref{52} each $c_j$ is equal to
  $em$. By (\ref{eq:56}) and (\ref{eq:57}), each $a_j$ is
  equal to $m$. Therefore, $M$ is free as an
  $\calO_\bfB\otimes W$-module. (b) Note that $M$ is free as an
  $\calO_\bfB\otimes W$-module if and only if $a_j=m$ for all $j$. 
  By (\ref{eq:56}), one has $c_j=c_{j+f}$ for all $j$.
  Since $M$ is separably quasi-polarizable, by
  (\ref{eq:510}) each $c_j$ is equal to
  $em$. By Lemma~\ref{52},  
  the \dieu $\calO_\bfB$-module $M$ satisfies the
  determinant condition. \qed 



\end{proof}

\begin{cor}\label{57} Assume that $p\neq 2$.
  Let $\Lambda$ be a $\Zp$-valued unimodular skew-Hermitian
  $\calO_\bfB$-module 
  of $\Zp$-rank $4dm$. For any separably quasi-polarized \dieu
  $\calO_{\bfB}$-module $M$ of rank $4dm$ that satisfies the
  determinant 
  condition, one has an isomorphism $M\simeq
  \Lambda\otimes_{\Zp} W$ as 
  skew-Hermitian $\calO_\bfB\otimes W$-modules. 
\end{cor}
\begin{proof}
  By Proposition~\ref{56}, $M$ is free as an $\calO_\bfB\otimes
  W$-module. When $\bfB$ is the matrix algebra, the assertion follows
  from Lemma~\ref{41}. When $\bfB$ is the division algebra, the
  assertion then follows 
  from Lemma~\ref{42}. \qed 
\end{proof}

\begin{remark}\label{58}
The converse of Proposition~\ref{56} (2) (a) does not hold. 
We give a counterexample.
Let $\calO=\Z_{p^2}$. Then $\calO_\bfB=\Z_{p^4}[\Pi]$ with
$\Pi^2=-p$ and $\Pi a=\sigma^2(a) \Pi$ for $a\in \Z_{p^4}$. Put
$M=(\calO_\bfB\otimes W)^{\oplus 2}$, which is obviously a 
free $\calO_\bfB\otimes W$-module. We have the decomposition 
$M=M^0\oplus M^1 \oplus M^2 \oplus M^3$ by the action of $\Z_{p^4}$. 
Choose a basis $\{e^j_1,e^j_2\}$ of $M^j$
for each $i$ such that the representative matrix for each $\Pi_i: M^i\to
M^{j+2}$ is given by $\left[\begin{smallmatrix} 0 & -p \\ 1 & 0
     \end{smallmatrix}\right]$. With respect to these bases, we define
   the
map $\sfV_j:M^{j+1}\to M^j$ by $\sfV_j=\left[\begin{smallmatrix} 
1 & 0 \\ 0 & 1
     \end{smallmatrix}\right]$ for $j=0,2$, 
and $\sfV_j=\left[\begin{smallmatrix} 
p & 0 \\ 0 & p \end{smallmatrix}\right]$ for $j=1,3$. Then we have the following
commutative diagram 
\[ 
\begin{CD}
  M^0 @<\sfV_0<< M^1 @<\sfV_1<< M^2 \\
  @VV\Pi_0V @VV\Pi_1V @VV\Pi_2V \\
  M^2 @<\sfV_2<< M^3 @<\sfV_3<< M^0 \\
  @VV\Pi_2V @VV\Pi_3V @VV\Pi_0V \\
  M^0 @<\sfV_0<< M^1 @<\sfV_1<< M^2. \\
\end{CD}
\]
This defines a \dieu
$\calO_\bfB$-module $M$ and one has $c_0=c_2=0$ and $c_1=c_3=2$. Thus,
$M$ does not satisfy the determinant condition by Lemma~\ref{52} (2). 
\end{remark}

 

  


\section{Isogeny classes of $p$-divisible $\calO_\bfB$-modules} 
\label{sec:06}

\subsection{}
\label{sec:61}

Keep the notation of Sections~\ref{sec:04} and \ref{sec:05}. Our goal is to
classify the isogeny classes of quasi-polarized $p$-divisible
$\calO_\bfB$-modules $\ul H=(H,\lambda,\iota)$ over an \ac field $k$
of \ch $p$. 
Let $(M,\<\,, \>,\iota)$ be the associated \dieu module with the
additional structures. Assume that $\rank_W M=m [\bfB:\Qp]$ for some
integer $m\ge 1$. 
Note that this is the general type D
case (in the local situation). Let $d:=[\bfF:\Qp]$. 
So the $p$-divisible group $H$ has height $4dm$. 

\def\slope{{\rm slope}}

The slope sequence (or Newton polygon) 
of a $p$-divisible group $H$ is denoted by  
$\nu(H)$. Write 
\[ \{\, \beta_i^{m_i}\, \}_{1\le i\le t}  \]
for the slope sequence with each slope $\beta_i$ of multiplicity
$m_i$. 

Two quasi-polarized $p$-divisible $\calO_\bfB$-modules
$\ul {H}=(H,\lambda,\iota)$ and $\ul
{H'}=(H',\lambda',\iota')$ are said to be {\it isogenous} if there is
an $\calO_\bfB$-linear quasi-isogeny $\varphi:H\to H'$ such that
$\varphi^* \lambda'=\lambda$. This is equivalently saying that the
associated $F$-isocrystals $\ul M\otimes_{\Zp} \Qp \simeq  \ul
M'\otimes_{\Zp} \Qp$ are isomorphic compatible with the additional
structures. Similarly, one defines the isogeny relation for $p$-divisible
$\calO_\bfB$-modules $(H,\iota)$. 
Clearly, the slope sequence of a (quasi-polarized) $p$-divisible
$\calO_\bfB$-module is determined by its isogeny class. 
The relationship between the Newton polygon and isogeny class of
$p$-divisible groups with additional structures (in the general
setting of $F$-isocrystals with $G$-structure) is known due to the
works of Kottwitz \cite{kottwitz:isocrystals,kottwitz:isocrystals2} 
and Rapoport-Richartz \cite{rapoport-richartz}. Here we describe the
image of the Newton map $\nu$ for $p$-divisible
$\calO_\bfB$-modules.     

\subsection{}
\label{sec:62}
We describe the isogeny classes of $p$-divisible
$\calO_\bfB$-modules. 

\begin{lemma}\label{61}
Let $(H,\iota)$ and $(H',\iota')$ be two $p$-divisible
$\calO_\bfB$-modules of the same height. Then $(H,\iota)$ is isogenous to
$(H',\iota')$ if and only if $\nu(H)=\nu(H')$.  
\end{lemma}
\begin{proof}
  The direction $\implies$ is obvious and we show the other
  direction. Since $\nu(H)=\nu(H')$, we may choose an
  isogeny $\varphi:H\to H'$. Then we have an isomorphism
  $j:\End^0(H')\simeq \End^0(H)$ of $\Qp$-algebras, 
  sending $a\mapsto \varphi^{-1} a \varphi$, where
  $\End^0(H):=\End(H)\otimes_{\Zp} \Qp$.  
  Put $\iota_1:=j\circ \iota': \bfB\to \End^0(H)$. 
  Since the center of $\End^0(H)$ is a product of copies of
  $\Q_p$, by the Noether-Skolem theorem, there is an element
  $\alpha \in \End^0(H)^\times$ such that $\iota_1=\Int(\alpha) \circ
  \iota$. For $b\in B$, one has 
\[ \varphi^{-1}\,\iota' (b)\, \varphi =\iota_1(b)= \alpha\, 
\iota(b)\, \alpha^{-1}. \]
Therefore, $\varphi\circ \alpha:H\to H'$ is an $\calO_\bfB$-linear
quasi-isogeny. \qed 
\end{proof}
 


\subsection{The isoclinic case.}
\label{sec:63}
Let $H$ be an isoclinic $p$-divisible $\calO_\bfB$-module of
height $h>0$ and of slope $\beta$, and let $M$ be the associated
\dieu $\calO_\bfB$-module. Let $E:=\End^0(H)$ be the endomorphism
algebra. Write $E=\Mat_{n}(\Delta)$, where
$\Delta$ is a central division $\Qp$-algebra with
$\inv(\Delta)=\beta$.\\


\subsubsection{}
\label{sec:631}
Suppose $\bfB$ is the matrix algebra. Then $H=H_1\oplus H_2$,  $H_1$
has height $h_1:=h/2$ and $\deg (\End^0(H_1))=h_1$. It is known
that there is a monomorphism $\bfF\to \End^0(H_1)$ 
if and only if  $d|h_1$, or
equivalently $2d|h$.  
As $\deg (\End^0(H_1))=h_1$, one has $\beta=a/h_1$ for some integer
$0\le a\le h_1$. Therefore,   
\[ \nu(H)=\left \{\left (\frac{a}{h_1}\right )^{h}\right \} \]
for some integer $0\le a\le h_1$.  
Conversely, suppose $2d | h$ and we are given a slope sequence 
$\nu=\{(2a/h)^{h}\}$ for some integer $0\le
a \le h/2$. Then there is a $p$-divisible group $H$ of height $h$ 
together
with an algebra monomorphism $\bfB \to \End^0(H)$ 
such that $\nu(H)=\nu$. 
Replacing $H$ by another 
$p$-divisible group in its isogeny class, 
we have a ring monomorphism $\iota: \calO_\bfB\to \End(H)$.\\


\subsubsection{}
\label{sec:632}
Suppose that $\bfB$ is the division algebra. Write $\beta=a/h$ for
some integer $0\le a \le h$.  As the
field $\bfL$ can be embedded into $E$, one has $2d|h$. Therefore,
we can write $h=2dh'$ for some integer $h'$. 

\begin{lemma}\label{62} Notations and assumptions are as above. 
There is an embedding of $\bfB$ into $E$ if and only if  
  \begin{equation}
    \label{eq:62}
    a\equiv h' \pmod 2.
  \end{equation}   
\end{lemma}
\begin{proof}
We write
\[ \Delta\otimes_{\Qp} \bfB^{\rm op}=\Delta_\bfF \otimes_{\bfF}
\bfB^{\rm op}=\Mat_c(\Delta'), \] 
where $\Delta_\bfF:= \Delta\otimes_{\Qp} \bfF$, $\bfB^{\rm op}$ is
the opposite algebra of $\bfB$, $\Delta'$ is a central division
$\bfF$-algebra.

By an embedding theorem for general simple algebras 
\cite[Theorem 2.7]{yu:embed}, there is an embedding of 
$\bfB$ into $E=\Mat_n(\Delta)$ if and only if 
the following condition holds
\begin{equation}
  \label{eq:61}
  [\bfB:\Q_p] \mid  n c. 
\end{equation} 
Write $\delta:=\deg (\Delta)$ and $\delta':=\deg(\Delta')$.
  We have $c\delta'=2 \delta$ and $h=n\delta$. 
Then one has 
\[ 4d|nc \iff 4d\delta'|nc \delta' \iff 2d \delta'|n\delta=h, \] 
and this is equivalent to the condition
\begin{equation}
  \label{eq:63}
  \delta'|h'.
\end{equation}
As $\delta'$ is the denominator of $\inv(\Delta')$, the condition 
(\ref{eq:63}) holds if and only if $h'\cdot \inv(\Delta')\in \Z$. We
compute 
\[ \inv(\Delta')=d\cdot \frac{a}{h}-\frac{1}{2}=\frac{a-h'}{2h'} \quad
\text{and hence} \quad  h'\cdot \inv(\Delta')=\frac{a-h'}{2}. \]
Therefore, the condition (\ref{eq:63}) holds if and only if 
the condition (\ref{eq:62}) holds. This proves the lemma. \qed
\end{proof}
 



By Lemma~\ref{62} one has
\[ \nu(H)=\left \{\left (\frac{a}{h}\right )^{h}\right \} \]
for some integer $0\le a\le h$ with $a\equiv h' \pmod 2$.  
Conversely, suppose $h=2dh'$ for some $h'\in \Z_{\ge 1}$ 
and we are given a slope sequence 
$\nu=\{(a/h)^{h}\}$ for some integer $0\le
a \le h$ with $a\equiv h' \pmod 2$. 
Then by adjusting $H$ in its isogeny class if necessary, 
there is a $p$-divisible 
$\calO_{\bfB}$-module $H$ of 
height $h$ such that $\nu(H)=\nu$. \\

We conclude our discussion by the following proposition. 
\begin{prop}\label{63}
  Let $h=2dh'$ with $h'\in \Z_{\ge 1}$, and let $(H,\iota)$ be an
  isoclinic $p$-divisible $\calO_\bfB$-module of height $h$ over $k$.
\begin{enumerate}
\item If $\bfB$ is the matrix algebra, then 
\[ \nu(H)=\left \{\left (\frac{a}{dh'}\right )^{h}\right \}, \]
where $a$ is an integer with $0\le a \le dh'$. Conversely, any slope sequence of this form is realized by a p-divisible $\calO_\bfB$-module of height $h$ over $k$.
\item  If $\bfB$ is the division algebra, then 
\[ \nu(H)=\left \{\left (\frac{a}{h}\right )^{h}\right \}, \]
where $a$ can be any integer with $0\le a \le h$ and 
$a\equiv h' \pmod 2$. Conversely, any slope sequence of this form is realized by a p-divisible $\calO_\bfB$-module of height $h$ over $k$.
\qed
\end{enumerate}
\end{prop}
    
\subsection{The general case}
\label{sec:64}
It is not hard to state the general case for possible slope 
sequences of $p$-divisible $\calO_\bfB$-modules 
based on the isoclinic case. One simply considers the decomposition
$H\sim H_1\times H_2\times \dots \times H_t$ into the isoclinic
components in the isogeny class.  

\begin{thm}\label{64}
  Let $h=2dh'$ with $h'\in \Z_{\ge 1}$, and let $(H,\iota)$ be a 
  $p$-divisible $\calO_\bfB$-module of height $h$ over $k$. 
\begin{enumerate}
\item If $\bfB$ is the matrix algebra, then
  \begin{equation}\label{eq:64}
 \nu(H)=\left \{\left (\frac{a_i}{dh_i'}\right )^{2dh'_i}\right
\}_{1\le i\le t},    
  \end{equation}
where $h_1'+\dots +h'_t=h'$ is a partition of the integer $h'$ and
$a_i$ is an integer with $0\le a_i \le dh'_i$. Moreover, after
combining the indices $i$ with same slope $a_i/dh_i'$ together and rearranging the indices, we may assume that 
\[ \frac{a_{i}}{dh_i'}< \frac{a_{i+1}}{dh_{i+1}'}, 
\quad  i=1,\dots, t-1. \]  
Conversely, any slope sequence of the form \eqref{eq:64} subject to the condition above arises from a $p$-divisible $\calO_\bfB$-module of height $h$ over $k$. 
\item  If $\bfB$ is the division algebra, then 
  \begin{equation}
    \label{eq:65}
    \nu(H)=\left \{\left (\frac{a_i}{2dh_i'}\right )^{2dh'_i}
 \right \}_{1\le i\le t},
  \end{equation}
where $h_1'+\dots +h'_t=h'$ is a partition of the integer $h'$ and 
$a_i$ is an integer with $0\le a_i \le 2dh'_i$ and $a_i\equiv
h_i'\pmod 2$. Similarly, we may rearrange the indices 
such that 
\[ \frac{a_{i}}{2dh_i'}< \frac{a_{i+1}}{2dh_{i+1}'},\quad i=1,\dots,
t-1. \] 
Conversely, any slope sequence of this form arises from a $p$-divisible  
$\calO_\bfB$-module of height $h$ over $k$. \qed
\end{enumerate}
\end{thm}

\section{Slope sequences of quasi-polarized \dieu $\calO_\bfB$-modules}
\label{sec:07}
Keep the notation of Section 6.
Let $h=4dm$ with $m\in \Z_{\ge 1}$, and let $(H,\lambda, \iota)$ be a 
quasi-polarized 
$p$-divisible $\calO_\bfB$-module of height $h$ over $k$. Then the
slope sequence $\nu(H)$ of $H$ is given as Theorem~\ref{64} together 
with the symmetric condition: for all $0 \le i, j\le t$ 
with $i+j=t+1$, one has $h_i'=h_j'$ and 
\begin{equation}
  \label{eq:71}
  \begin{split}
    & \frac{a_{i}}{dh_i'}+ \frac{a_{j}}{dh_j'}=1 \quad \text{in the
      matrix algebra case,}\\
    &  \frac{a_{i}}{2dh_i'}+\frac{a_{j}}{2dh_j'}=1 \quad \text{in the
      division algebra case,}
  \end{split}
\end{equation}
where $h_1'+\dots +h_t'=2m$ is a partition of $2m$.

Note that the integer $t$ is even if and only if $H$ has no
supersingular component. For any symmetric slope sequence $\nu$ which
has the form in Theorem~\ref{64}, 
we write 
\begin{equation}
  \label{eq:nu}
  \nu=\nu_n \cup \nu_s,
\end{equation}
where $\nu_n$, the {\it non-supersingular part},
consists of all slopes in $\nu$ 
which are not $1/2$, and $\nu_s$, the {\it supersingular part},  
consists of all slopes $1/2$ in $\nu$. 

\begin{lemma}\label{71}
Let $h$ be a positive integer and $\beta$ a positive rational number
such that there exists an isoclinic $p$-divisible $\calO_\bfB$-module of
height $h$ and with slope $\beta$. Then there exists a separably
quasi-polarized $p$-divisible $\calO_{\bfB}$-module
$(H,\lambda,\iota)$ of height $2h$ such that $\nu(H)=\{\beta^h,
(1-\beta)^h\}$.    
\end{lemma}
\begin{proof}
  Choose an isoclinic $p$-divisible $\calO_\bfB$-module 
  $(H_1,\iota_1)$ of
  height $h$ and with the slope $\beta$. Put $H_2:=H_1^t$, which is 
  an isoclinic $p$-divisible $\calO_\bfB$-module  of 
  height $h$ and with the slope $1-\beta$. The monomorphism
  $\iota_2:\calO_\bfB \to \End(H_2)$ is given by
  $\iota_2(a):=\iota_1(a^*)^t$ for $a\in \calO_\bfB$. 
  Put $H:=H_1\times H_2$, and then
  $H^t=H_1^t\times H_2^t$. Let $\lambda=
  \left[\begin{smallmatrix}
   & \lambda_2 \\ \lambda_1 & 
  \end{smallmatrix}\right]:H\to
  H^t$ be an $\calO_\bfB$-linear isogeny, where
$\lambda_1:H_1\to H_2^t$ and
$\lambda_2:H_2 \to H_1^t$ are $\calO_\bfB$-linear isogenies. Then
$\lambda^t=\left[\begin{smallmatrix}
   & \lambda_1^t \\ \lambda_2^t & 
  \end{smallmatrix}\right]$, so $\lambda^t=-\lambda$ if and
only if $\lambda_2=-\lambda_1^t$. Choose $\lambda_1$ an
$\calO_\bfB$-linear isomorphism and put
$\lambda=\left[\begin{smallmatrix}
   & -\lambda_1^t \\ \lambda_1 & 
  \end{smallmatrix}\right]$. Then $\lambda$ is a separably
$\calO_\bfB$-linear quasi-polarization. \qed
\end{proof}

Note that the construction in Lemma~\ref{71} 
works for any finite-dimensional 
simple $\Qp$-algebra $\bfB'$ with involution, not just for the
quaternion algebra treated in this paper. This method of 
construction appears quite often in dealing with
symmetric slope sequences with two slopes and we call 
this {\it the double construction}.

For the supersingular case, we have the following result. 
 
\begin{thm}\label{72}
  For any positive integer $m$, there exists a superspecial separably
  quasi-polarized \dieu $\calO_\bfB$-module $M$ of rank $4dm$ that
  satisfies the determinant condition (K).  
\end{thm}


The proof of this theorem is given by construction. Namely, we
directly write down such a \dieu module. The construction is lengthy
and we omit the
details of the 
construction\footnote{The reader can find the detailed construction in
  older versions of arXiv:1306.1400.}.  
We conclude the main result of this section.

\begin{thm}\label{73}
  Let $h=4dm$ with any $m\in \Z_{\ge 1}$. Let $\nu$ be a slope sequence
  of the form in Theorem~\ref{64} 
  that satisfies the symmetric condition (\ref{eq:71}), where 
  $h_1'+\dots +h_t'=2m$ is any partition of $2m$. Then there exists a
  separably quasi-polarized \dieu $\calO_\bfB$-module $M$ of rank $h$ 
  and with $\nu(M)=\nu$.   
\end{thm}

\begin{proof}
  Write $\nu=\nu_n+\nu_s$ into non-supersingular part and
  supersingular part, say of length $4dm_n$ and $4dm_s$,
  respectively. As each isoclinic component $\{\beta_i^{m_i}\}$ of
  $\nu_n$ can
  be realized by a $p$-divisible $\calO_\bfB$-module
  (Theorem~\ref{64}), by 
Lemma~\ref{71} there is a separably quasi-polarized $p$-divisible
$\calO_\bfB$-module $(H_n,\lambda_n,\iota_n)$ of height $4dm_n$ such that
$\nu(H_n)=\nu_n$. 
On the other hand, by Theorem~\ref{72}, there
is a superspecial separably
  quasi-polarized $p$-divisible $\calO_\bfB$-module
  $(H_s,\lambda_s,\iota_s)$ of height $4dm_s$ that 
  satisfies the determinant condition. The product
  $(H_n,\lambda_n,\iota_n)\times (H_s,\lambda_s,\iota_s)$ satisfies
  the desired properties. \qed 
\end{proof}

\begin{remark}\label{74} \ 

  (1) In Lemma~\ref{71} one may 
  choose $H$ such that $H$ is a {\it minimal} $p$-divisible
  group in the sense of Oort, that is, the endomorphism ring 
  $\End(H)$ of $H$
  is a maximal $\Zp$-order in the semi-simple $\Qp$-algebra
  $\End^0(H)$. This follows from the construction of the minimal
  isogeny; see Section 4 (particularly Proposition 4.8) in 
  \cite{yu:endo}. 
  Therefore, the $p$-divisible $\calO_\bfB$-module in
  Theorem~\ref{73} can be chosen to be minimal. 
  
  (2) We shall see that when $\bfB$ is the division algebra, 
      the determinant condition (K) will rule
      out some possibilities of the slope sequences that are realized
      by separably quasi-polarized \dieu $\calO_\bfB$-module in
      Theorem~\ref{73}. That is, not all symmetric slope sequences in
      Theorem~\ref{73} occurring as those of separably quasi-polarized
      \dieu $\calO_\bfB$-modules that satisfy {\it the determinant
      condition}. We refer to Section~\ref{sec:12} for more details in
      the case of rank $4d$. 

  (3) When $\bfB$ is a division algebra and $d$ is odd, 
  the double construction as in Lemma~\ref{71}
  provides an alternative way to produce a separably quasi-polarized
  superspecial \dieu $\calO_\bfB$-module $M$. However, we checked that
  such a \dieu module $M$ rarely satisfies the determinant condition. 

  (4) We refer to \cite{yu:mass_hb} for a classification of
      superspecial quasi-polarized \dieu $O_F\otimes \Zp$-modules of
      HB type.       
\end{remark}

\begin{cor}\label{75}
  There is an ordinary separably quasi-polarized $p$-divisible 
  $\calO_\bfB$-module of height $4dm$ 
  if and only if one of the following holds:
  \begin{enumerate}
  \item $\bfB$ is the matrix algebra;
  \item $\bfB$ is the division algebra and $m$ is even. 
  \end{enumerate}
\end{cor}
\begin{proof} Any ordinary $p$-divisible $\calO_\bfB$-module always 
  admits an $\calO_\bfB$-linear separable quasi-polarization. 

  If $\bfB$ is the matrix algebra, then the ordinary slope sequence  
  appears  
  in Theorem~\ref{64} (or in Proposition~\ref{63}). Therefore, 
  by Theorem~\ref{73} there is an ordinary
  separably quasi-polarized $p$-divisible $\calO_\bfB$-module of
  height $4dm$. 

  Suppose $\bfB$ is the division
  algebra. Then 
  a slope sequence with 2 slopes $\{ (a_1/2dh'_1)^{2dh'_1},  
  (a_2/2dh'_2))^{2dh'_2}\}$ which is of the form (\ref{eq:65}) can be the  
  ordinary slope sequence 
  if and only if $h_1'=h_2'=m$ and
  $m\equiv 2dm \pmod 2$ (taking $a_2=2 dh'_2$ and $h_2'=m$). 
  That is, $m$ is even. 
  \qed 
\end{proof}

\begin{remark}\label{7.55}
For a smooth PEL-type moduli space $\calM'$, the ordinary locus of
$\calM'\otimes \ol{k(v)}$ is non-empty if and only if $E_v=\Qp$, where
$E$ is the reflex field and $v$ is a finite place of $E$ lying over $p$; see Wedhorn~\cite{wedhorn:ordinary}. 
Corollary~\ref{75} shows that the latter
condition is not sufficient for the non-emptiness of the ordinary
locus in some ramified cases.    
\end{remark}

We give a few examples of possible slope sequences of quasi-polarized
\dieu $\calO_\bfB$-modules. 

\begin{cor}[$m=1$]\label{76}
  Let $(H,\lambda,\iota)$ be a quasi-polarized $p$-divisible
  $\calO_\bfB$-module of height $4d$. 
  \begin{enumerate}
  \item If $\bfB$ is the matrix algebra, then 
\[  \nu(H)=\left \{\left (\frac{a}{d}\right )^{2d}, \left
 (\frac{d-a}{d}\right )^{2d}\right \},  \]
where $a$ can be any integer with $0\le a< d/2$, or 
\[\nu(H)=\left \{\left (\frac{1}{2}\right )^{4d}\right \}. \]    
  \item If $\bfB$ is the division algebra, then 
\[  \nu(H)=\left \{\left (\frac{a}{2d}\right )^{2d}, \left
 (\frac{2d-a}{2d}\right )^{2d}\right \},  \]
where $a$ can be any integer with $0\le a< d$ 
with $a\equiv 1 \pmod 2$, or 
\[ \nu(H)=\left \{\left (\frac{1}{2}\right )^{4d}\right \}. 
\quad \text{\qed} \]    
\end{enumerate}
\end{cor}

\begin{cor}[$m=2$]\label{77}
  Let $(H,\lambda,\iota)$ be a quasi-polarized $p$-divisible
  $\calO_\bfB$-module of height $8d$. 
  \begin{enumerate}
  \item If $\bfB$ is the matrix algebra, then we have the following
    possibilities of $\nu(H)$:
  \begin{enumerate}
    \item one slope case:
  \[\nu(H)=\left \{\left (\frac{1}{2}\right )^{8d}\right \}. \] 
  \item two slopes case:
  \[  \nu(H)=\left \{\left (\frac{a}{2d}\right )^{4d}, \left
  (\frac{2d-a}{2d}\right )^{4d}\right \}, 
  \quad 0\le a <d,\ a\in \Z.  \]
  \item three slopes case:
  \[  \nu(H)=\left \{\left (\frac{a}{d}\right )^{2d}, 
  \left (\frac{1}{2}\right )^{4d},
  \left (\frac{d-a}{d}\right )^{2d}\right \}, 
  \quad 0\le a <\frac{d}{2},\ a\in \Z.  \]
  \item four slopes case:
  \[  \nu(H)= \left \{
  \left (\frac{a}{d}\right )^{2d}, 
  \left (\frac{b}{d}\right )^{2d}, 
  \left (\frac{d-b}{d}\right )^{2d}, 
  \left (\frac{d-a}{d}\right )^{2d}\right \}, 
  \]
  where $a$ and $b$ can be any integers with $0\le a<b  < d/2$. 
\end{enumerate}

\item If $\bfB$ is the division algebra, then we have the following
    possibilities of $\nu(H)$:
    \begin{enumerate}
     \item one slope case:
  \[\nu(H)=\left \{\left (\frac{1}{2}\right )^{8d}\right \}. \] 
  \item two slopes case:
  \[  \nu(H)=\left \{\left (\frac{a}{4d}\right )^{4d}, \left
  (\frac{4d-a}{4d}\right )^{4d}\right \}, 
  \quad 0\le a <2d,\ a\in \Z, \ a\equiv 0 \pmod 2.  \]
  \item three slopes case:
  \[  \nu(H)=\left \{\left (\frac{a}{2d}\right )^{2d}, 
  \left (\frac{1}{2}\right )^{4d},
  \left (\frac{2d-a}{2d}\right )^{2d}\right \},  \]
  where $a$ can be any integer with $0\le a < d$ and $a\equiv 1 \pmod 2$.
  \item four slopes case:
  \[  \nu(H)= \left \{
  \left (\frac{a}{2d}\right )^{2d}, 
  \left (\frac{b}{2d}\right )^{2d}, 
  \left (\frac{2d-b}{2d}\right )^{2d}, 
  \left (\frac{2d-a}{2d}\right )^{2d}\right \}, 
  \]
  where $a$ and $b$ can be any integers with $0\le a<b  < d$ and $a, b
  \equiv 1 \pmod 2$. \qed 
    \end{enumerate}
\end{enumerate} 
\end{cor}

\section{Isogeny classes of quasi-polarized $p$-divisible
  $\calO_\bfB$-modules}
\label{sec:09}

In this section we would like to classify  
isogeny classes of quasi-polarized $p$-divisible
$\calO_\bfB$-modules of height $4dm$ over an \ac field $k$ of \ch $p>0$.
Consider rational quasi-polarized \dieu $\bfB$-modules $N$ with
$B(k)$-$\rank=4dm$, or quasi-polarized $\bfB$-linear $F$-isocrystals. 
When all slopes of $N$ are between $0$ and $1$, there is a \dieu
$\calO_\bfB$-lattice $M$ in $N$, so $N=M\otimes_{W(k)} B(k)$ for some
quasi-polarized \dieu $\calO_\bfB$-module. Thus, classifying isogeny classes in question may be reduced to classifying rational quasi-polarized \dieu $\bfB$-modules $N$.
Let $\nu$ be a symmetric slope sequence as in Theorem~\ref{73}. Let
$I(\nu)$ denote the set of 
isogeny classes of quasi-polarized $p$-divisible 
$\calO_\bfB$-modules $(H,\lambda,\iota)$ of height $4dm$ such that
$\nu(H)=\nu$.
By \dieu theory, the set $I(\nu)$ is isomorphic to the set of isomorphism  classes of rational quasi-polarized \dieu $\bfB$-modules $N$ of $B(k)$-$\rank=4dm$ and with slope sequence $\nu$.


Rapoport and Richartz \cite{rapoport-richartz} have
classified the isomorphism classes of isocrystals with $G$-structure with a fixed Newton vector by the Galois cohomology $H^1(\Q_p, J)$ for a certain reductive group $J$ over $\Qp$ when $G$ is connected. For the present case one needs to work a bit more, because
our group $G$ is not connected.
The description by Galois cohomology helps us to
understand the classification problem. 
However, in order to obtain more explicit results, 
one also needs to compute the set $H^1(\Qp, J)$.


We shall work along with \dieu modules and translate the 
classification problem into the theory of (skew-)Hermitian 
forms over local fields; see Theorem~\ref{92}.

Write $\nu=\nu_n + \nu_s$ into the non-supersingular and supersingular
parts. 

\begin{lemma}\label{91}
  We have $I(\nu)=I(\nu_n)\times I(\nu_s)$, and $I(\nu_n)$ consists of
  one isogeny class. 
\end{lemma}
\begin{proof}
Suppose $N_1$ and $N_2$ are rational quasi-polarized \dieu
$\bfB$-modules with $\nu(N_1)=\nu(N_2)=\nu$. Write 
\[ N_1=N_1^{ns}\oplus N_1^{ss}, \quad N_2=N_2^{ns}\oplus N_2^{ss} \]
for the decomposition of $N_1$ and $N_2$ into the non-supersingular
component and supersingular component, 
respectively. Clearly, one has $N_1\simeq N_2$ if and only if 
$N_1^{ns}\simeq N_2^{ns}$ and $N_1^{ss}\simeq N_2^{ss}$. This shows
the first part. 

For the second part, we decompose the rational \dieu modules 
\[ N_1^{ns}=\oplus_{\beta<1/2} (N_{1,\beta}\oplus N_{1,\beta}^t),
\quad N_2^{ns}=\oplus_{\beta<1/2} (N_{2,\beta}\oplus
N_{2,\beta}^t) \]
into isotypic components. By Lemma~\ref{61}, we have
an isomorphism $N_{1,\beta}\simeq N_{2,\beta}$ as rational \dieu
$\bfB$-modules for each $\beta<1/2$. It follows that $N_1^{ns}\simeq
N_2^{ns}$ as rational quasi-polarized \dieu $\bfB$-modules. \qed
\end{proof}

By Lemma~\ref{91}, one reduces the classification to classifying the
isomorphism classes of supersingular rational quasi-polarized \dieu
$\bfB$-modules of $B(k)$-rank $4dm_s$, where $4dm_s$ is the length of
the supersingular part $\nu_s$.

Let $N$ be a supersingular rational quasi-polarized 
\dieu $\bfB$-module of $B(k)$-rank $4dm_s$. Put 
\[ \wt N:=\{x\in N | \sfF^2 x=px\}. \] 
This is a $B(\F_{p^2})$-vector space of dimension $4dm_s$ such that
\begin{itemize}
\item $W(k)\otimes_{B(\F_{p^2})} \wt N=N$,
\item $\sfF=\sfV $ on $\wt N$, and 
\item the action of $\bfB$ leaves $\wt N$ invariant.
\end{itemize}

Let $\bfD$ be the quaternion division algebra over $\Qp$. We can write 
$\bfD=B(\F_{p^2})[\sfF]$ with relations $\sfF^2=p$ and $\sfF
a=\sigma(a) \sfF$ for 
all $a\in B(\F_{p^2})$. Then $\wt N$ naturally becomes a left
$\bfD$-module of $\Q_p$-rank $8dm_s$. Define the involution $*_\bfD$
on $\bfD$ by 
\[ (a+b\sfF)^{*_\bfD}:=\sigma(a)+b\sfF. \] 
This is an orthogonal
involution as the fixed subspace is $3$-dimensional.  
As the actions of
 $\bfB$ and $\bfD$ commute, $\wt N$ becomes a left $\bfB \otimes_{\Qp}
 \bfD$-module. Write  
\[ \bfB \otimes_{\Qp} \bfD=\bfB \otimes_\bfF (\bfF \otimes_{\Qp}
\bfD)\simeq 
 \Mat_2 (\bfB'), \]
where $\bfB'$ is a quaternion algebra over $\bfF$, which is determined
by the relation
\[ \inv (\bfB')=1/2 [\bfF:\Qp]-\inv (\bfB). \]

The alternating pairing 
\[ \<\, , \>:\wt N\times \wt N\to B(\F_{p^2}) \]
has values in $B(\F_{p^2})$ satisfying $\<\sfF x,y\>=\<x,\sfV y\>^\sigma$. 
Define 
\[ \psi(x,y):= \Tr_{B(\F_{p^2})/\Qp} \<x,\sfF y\>. \]
For all $x, y\in \wt N$, $a\in
\bfD$, and $b\in \bfB$, one has
\begin{itemize}
\item [(i)] $\psi(y,x)=-\psi(x,y)$, 
\item [(ii)] $\psi(ax,y)=\psi(x, a^{*_{\bfD}} y)$, and 
\item [(iii)] $\psi(bx,y)=\psi(x, b^* y)$.  
\end{itemize}
That is,  $\wt N$ is a $\Q_p$-valued skew-Hermitian 
$\bfB \otimes_{\Qp} \bfD$-module with respect to the product
involution $*\otimes {*_\bfD}$. We check (i)--(iii). For (i), 
\[ \psi(y,x)=\Tr \<y,\sfF x\>=\Tr \<\sfF y,x\>^\sigma=-\Tr  \<x,\sfF
y\>=-\psi(x,y). \] 
For (ii), one has for $a\in B(\F_{p^2})\subset \bfD$ 
\[ \psi(ax,y)=\Tr \<ax,\sfF y\>=\Tr \<x, \sfF  a^\sigma y\>=\psi(x,a^\sigma
y),\]
\[ \psi(\sfF x,y)=\Tr \<\sfF x,\sfF y\>=\Tr p\<x,y\>^\sigma=\Tr
\<x,\sfF^2y\>=\psi(x,\sfF y). \]
For (iii), 
\[ \psi(bx,y)=\Tr \<bx,\sfF y\>=\Tr \<x, \sfF b^* y\>=\psi(x,b^* y). \]
Note that if we replace $\wt N$ by $\wt N':=\{x\in
N|\sfF ^2x+px=0\}$, then $\sfF=-\sfV $ on $\wt N'$ and the pairing
$\psi'(x,y):=\Tr\<x,\sfF y\>$ becomes Hermitian instead of
skew-Hermitian. Moreover,
the adjoint involution $*'_\bfD$ on $\bfD$ is the canonical involution. 

Since the canonical involution $*$ is symplectic and $*_\bfD$ is
orthogonal, the product involution $*\otimes *_\bfD$ is
symplectic. Therefore, we can choose an $\bfF$-algebra isomorphism 
$\bfB \otimes_{\Qp} \bfD\simeq \Mat_2 (\bfB')$ such that the induced
involution is the map $(b_{ij})\mapsto (b_{ji}^{*'})$, where $*'$ is the
canonical involution on $\bfB'$. 

Let $e_{11}$ and $e_{22}$ be the standard idempotents of $\Mat_2
(\bfB')$ and let $\wt N=\wt N_1\oplus \wt N_2$ be the corresponding
decomposition. We have proven

\begin{thm}\label{92}
  The association $(N,\<\, ,\>)\mapsto (\wt N_1,\psi)$ 
  gives rise to a bijection between
  the set $I(\nu_s)$ and the set of isomorphism classes of
  $\Qp$-valued skew-Hermitian free $\bfB'$-modules of $\bfB'$-rank
  $m_s$, where $\bfB'$ is the quaternion algebra (unique up to
  isomorphism) over $\bfF$ with 
  $\inv(\bfB')=1/2 [\bfF:\Qp]-\inv (\bfB)$. \qed 
\end{thm}

\begin{cor}\label{93} \ 
\begin{enumerate}
  \item If $\bfB'$ is the matrix algebra, then there is a natural
    bijection between the set $I(\nu_s)$ and the set of isomorphism 
    classes of
    non-degenerate symmetric spaces over $\bfF$ of dimension $2m_s$.
  \item If $\bfB'$ is the quaternion division algebra, 
   then there is a natural
    bijection between the set $I(\nu_s)$ and the set of isomorphism
    classes of
    non-degenerate skew-Hermitian $\bfB'$-modules  of $\bfB'$-rank $m_s$.
\end{enumerate}
\end{cor}
\begin{proof}
  For the matrix algebra case, we do the Morita equivalence again as
  before. The corollary follows from Theorem~\ref{92}. \qed  
\end{proof}

In the following we use the theory of quadratic forms 
and the skew-Hermitian quaternionic forms over local fields; see O'Meara
\cite[Chapter IV]{omeara:book} and Tsukamoto~\cite{tsukamoto:1961}.



Consider non-degenerate symmetric spaces $V$ of dimension $n_0$ over a
non-Archimedean local field $k_0$ of \ch different from $2$. Recall
the discriminant $\delta V\in k_0^\times/k_0^{\times 2}$ of $V$ 
is defined by 
\[ \delta V:= (-1)^{\lfloor n_0/2 \rfloor} \det V. \] 
Note that we have $\delta V=[1]$ when $V$ is the hyperbolic plane.  
Let $S\, V\in \{\pm 1\}$ denote
the Hasse symbol of $V$ (see \cite[p.~167]{omeara:book}). 
Denote by $Q(n_0)$
the set of isomorphism classes of non-degenerate symmetric spaces $V$
of dimension $n_0$ over $k_0$. 

\begin{thm}\label{94} 
  \begin{enumerate}
  \item For any $n_0\ge 1$, the map $(\delta, S): Q(n_0)\to
    k_0^\times/k_0^{\times 2}\times \{\pm 1\}$ is injective. This map
    is also surjective for any $n_0\ge 3$. 
  \item For $n_0=1$, the map $\delta: Q(1) \simeq
    k_0^\times/k_0^{\times 2}$ is a bijection.
  \item For $n_0=2$, the image of the map $(\delta,S)$ is
\[ \big \{([a],\pm1); [a]\neq [1]\big  \}\cup \left \{\left([1], \left
  (\frac{-1,-1}{k_0}\right) \right) \right \}, \]
  where $(-1,-1/k_0)$ is a Hilbert symbol.
  \end{enumerate}
\end{thm}
\begin{proof}
  See Theorems 63:20, 63:22 and 63:23 of
  \cite[p.~170-171]{omeara:book}. \qed
\end{proof}

\begin{cor}\label{95}
  If $k_0$ is non-dyadic, then one has
\[ |Q(1)|=4, \quad |Q(2)|=7, \quad \text{and}\quad |Q(n_0)|=8, \quad
\forall\, n_0\ge 3.\ \text{\qed} \]
\end{cor}

Let $B_0$ be the quaternion division algebra over $k_0$ together with
the canonical involution $*$. Denote 
by $SQ(n_0)$
the set of isomorphism classes of skew-Hermitian $B_0$-modules
$(V,\psi)$ of rank $n_0$ for $n_0\ge 1$. The discriminant $\delta V\in
k_0^\times/k_0^{\times 2}$ is defined by 
\[ \delta V:=(-1)^{\lfloor n_0/2 \rfloor}
\Nr \left (\psi(e_i,e_j)\right ), \]
where $\{e_i\}$ is a basis for $V$ over $B_0$ and
$\Nr:\Mat_{n_0}(B_0)\to k_0$ is the reduced norm. 

\begin{thm}\label{96}\
  \begin{enumerate}
  \item For $n_0\ge 2$, the map $\delta:SQ(n_0)\to
    k_0^\times/k_0^{\times 2}$ is a bijection. 
  \item For $n_0=1$, the map $\delta:SQ(n_0)\to
    k_0^\times/k_0^{\times 2}$ is injective and its  image is equal to
    $\{[a]; [a]\neq [1]\}$.
  \end{enumerate}
\end{thm}
\begin{proof}
  This is Theorem 3 of \cite{tsukamoto:1961}. \qed 
\end{proof}

\begin{cor}\label{97}
  If $k_0$ is non-dyadic, then one has
\[ |SQ(1)|=3, \quad \text{and}\quad |SQ(n_0)|=4, \quad
\forall\, n_0\ge 2. \ \text{\qed} \]
\end{cor}

\begin{prop}\label{98}  Let the notation be as above and assume that $p\neq 2$. 
  \begin{enumerate}
  \item If $\bfB'$ is the matrix algebra, then we have 
    \begin{equation}
      \label{eq:91}
      |I(\nu_s)|=
      \begin{cases}
        7 & \text{if $m_s=1$,}\\
        8 & \text{if $m_s\ge 2$.}
      \end{cases}
    \end{equation}
  \item If $\bfB'$ is the quaternion division algebra, then we have 
    \begin{equation}
      \label{eq:92}
      |I(\nu_s)|=
      \begin{cases}
        3 & \text{if $m_s=1$,}\\
        4 & \text{if $m_s\ge 2$}. 
      \end{cases}
    \end{equation}
  \end{enumerate}
\end{prop}
\begin{proof}
  These follow from Corollaries~\ref{93}, \ref{95} and \ref{97}. \qed
\end{proof}

Combining Lemma~\ref{91}, Theorems~\ref{92}, \ref{94} and \ref{96},
we obtain an explicit classification of isogeny
classes of quasi-polarized $p$-divisible $\calO_\bfB$-modules over $k$. 
 

    

\section{Local model for $\calMpK$}
\label{sec:10}

We shall use the notation of
Section~\ref{sec:02} and Section~\ref{sec:41}. 
For the remainder of this paper, we assume $p\neq 2$.

\subsection{Local models}
\label{sec:10.1}

Let $\Lambda$ be a free $O_B\otimes_\Z \Zp$-module of rank $m$
together with a perfect $\Zp$-valued skew-Hermitian pairing
\[ \psi: \Lambda\times \Lambda\to \Zp. \]

For such a lattice $\Lambda$, we define, following Rapoport and Zink
\cite{rapoport-zink}, a projective $\Z_p$-scheme $\bfM_\Lambda$, called
the {\it local model associated to $\Lambda$} (and $\psi$), 
which represents the following functor. For any
$\Zp$-scheme $S$, $\bfM_\Lambda(S)$ is the set of locally free
$\calO_S$-submodules $\scrF\subset \Lambda\otimes_{\Zp} \calO_S$ of
rank $m[B:\Q]/2=2dm$ such that 
\begin{itemize}
\item [(i)] $\scrF$ is isotropic with respect to the pairing $\psi$;
\item [(ii)] locally for Zariski topology on $S$, $\scrF$ is a direct
  summand of $\Lambda\otimes_{\Zp} \calO_S$;
\item [(iii)] $\scrF$ is invariant under the $O_B$-action;
\item [(iv)] $\scrF$ satisfies the determinant condition
  (cf. Section~\ref{sec:23}):
\[ (K)\quad \char(a|\Lambda\otimes \calO_S/\scrF)=\char (a)^m\in
  \calO_S[T], \quad \forall\, a\in O_B. \]
\end{itemize}

Recall that for an abelian scheme $A$ over a base scheme $S$, 
we have the Hodge filtration
\[ 0\to \omega_{A/S} \to H^1_{\rm DR} (A/S)\to \Lie(A^t/S) \to 0. \]
Taking the dual one obtains the short exact sequence
\[ 0\to \omega_{A^t/S} \to H_1^{\rm DR} (A/S)\to \Lie(A/S) \to 0. \]
If $M$ is the covariant \dieu module of an abelian variety $A$ over a
perfect field $k_0$, then there is a canonical isomorphism $M/pM\simeq 
H_1^{\rm DR}(A/k_0)$ with the Hodge filtration $\sfV M/pM$ corresponding
to $\omega_{A^t}$. 
This justifies the definition of the determinant condition for objects
in the local model $\bfM_\Lambda$ in (iv). 

By an automorphism of the lattice $\Lambda\otimes \calO_S$, where $S$
is a $\Zp$-scheme, we mean an $O_B\otimes \Zp$-linear automorphism 
of the $\calO_S$-module $\Lambda\otimes \calO_S$ 
that preserves the pairing $\psi$. 
We denote by 
$\Aut_{O_B\otimes \calO_S}(\Lambda\otimes \calO_S,\psi)$ 
the group of automorphisms of $\Lambda\otimes \calO_S$. 

Let $\calG=\Aut_{O_B\otimes \Zp}(\Lambda,\psi)$ 
be the group scheme over $\Zp$ that represents the 
group functor 
\[ S\mapsto \Aut_{O_B\otimes \calO_S} 
(\Lambda\otimes \calO_S, \psi). \]
By \cite[Theorem 3.16]{rapoport-zink} and \cite[Theorem 2.2(a)]{pappas:rz}, $\calG$ is an affine smooth group scheme over $\Zp$. The
generic fiber $\calG_{\Qp}$ of $\calG$ is a $\Qp$-form of $(\Res_{F/\Q}
O_{2m,F})\otimes_\Q \Qp$; see \eqref{eq:GQ}. 
The group scheme $\calG$ acts naturally on $\bfM_{\Lambda}$ on the
left. 

\subsection{Local model diagrams}
\label{sec:10.2}
We impose a level structure away from $p$ on $\calM^{(p)}_K$ and consider a fine moduli scheme $\calM^{(p)}_{K,*}$ over $\Z_{(p)}$, More precisely, write $\calM^{(p)}_K=\coprod_{p\nmid D}\calM_{K,D}$, where $\calM_{K,D}=\calM_{K}\cap \calM_{D}$ and $\calM_D$ is defined in \eqref{eq:21}. For each $D$, we fix a prime-to-$pD$ integer $N_D\ge 3$. Set $\calM^{(p)}_{K,*}:=\coprod_{p\nmid D} \calM_{K,D,N_D}$, where $\calM_{K,D,N_D}$ is the moduli scheme over $\Z_{(p)}$ which parameterizes objects in $\calM_{K,D}$ with a level $N_D$ structure. 

Let $S$ be a $\Zp$-scheme and
$\ul A=(A,\lambda,\iota,\eta)$ be an object in
$\calM^{(p)}_{K,*}(S)$. 
A {\it trivialization} $\gamma$ of the de Rham homology $H_1^{\rm
  DR}(A/S)$ by $\Lambda \otimes_{\Zp} \calO_S$ is an 
$O_B\otimes \Zp$-linear isomorphism $\gamma: H_1^{\rm
  DR}(A/S) \to \Lambda \otimes_{\Zp} \calO_S$ of $\calO_S$-modules  
such that $\psi(\gamma(x),\gamma(y))=\<x,y\>_\lambda$ for $x, y \in H_1^{\rm
  DR}(A/S)$, where 
\[ \<\,,\>_\lambda: H_1^{\rm DR}(A/S) \times H_1^{\rm DR}(A/S) \to
\calO_S \]
is the perfect alternating pairing induced by $\lambda$. 

Let $\wt \calM=\wt {\calM^{(p)}_{K,*}}$
denote the moduli space over $\Z_p$ that parametrizes
equivalence classes of objects
$(\ul A, \gamma)_S$, where
\begin{itemize}
\item $\ul A=(A,\lambda,\iota,\eta)$ is an object over a $\Zp$-scheme $S$ in
$\calM^{(p)}_{K,*}\otimes \Zp$, and
\item  $\gamma$ is a trivialization of $H_1^{\rm
  DR}(A/S)$ by $\Lambda \otimes \calO_S$. 
\end{itemize}

The moduli scheme $\wt \calM$ has two natural
projections $\varphi^{\rm mod}$ and $\varphi^{\rm loc}$.
The morphism 
\[ \varphi^{\rm mod}:\wt \calM \to
\calM^{(p)}_{K,*}\otimes \Zp \] 
forgets the trivialization. The morphism 
\[ \varphi^{\rm
  loc}: \wt \calM \to \bfM_\Lambda \] 
sends any object $(\ul A, \gamma)$ to
$\gamma(\omega_{A^t/S})$, where $\omega_{A^t/S}\subset 
H_1^{\rm DR}(A/S)$ is the 
$\calO_S$-submodule in the Hodge filtration.
Thus, we have the so called {\it local model diagram}:   
\begin{equation}
  \label{eq:10.1}
  \begin{CD}
  \calM^{(p)}_{K,*}\otimes \Z_p @<{\varphi^{\rm mod}}<< \wt \calM 
  @>{\varphi^{\rm loc}}>>   \bfM_\Lambda.  
  \end{CD}
\end{equation}






The local model diagram above was introduced by Rapoport and Zink
\cite{rapoport-zink} in a more general setting. 
The moduli scheme $\wt \calM$ also admits a left action
by the group scheme $\calG$. 
Recall that  
$k$ denotes an \ac field of \ch $p>0$ and $W=W(k)$ the ring of
Witt vectors over $k$. 
Using Corollary~\ref{57}, for any $k$-valued point $\ul A$ in
$\calM^{(p)}_{K,*}$, there is an $O_B\otimes \Zp$-linear 
isomorphism of $W$-modules
\[ M(A) \simeq \Lambda\otimes W \]
which is compatible with the alternating pairings. 
This shows that the morphism
$\varphi^{\rm mod}$ is surjective. 
It follows that $\varphi^{\rm mod}$ is a left $\calG$-torsor, 
and hence this morphism is affine and smooth. 

By the Grothendieck-Messing deformation theory 
for abelian schemes (\cite{grothendieck:bt, messing:bt} and \cite[3.29]{rapoport-zink}),
for any $k$-valued point
$x$ of $\calM^{(p)}_{K,*}$, 
there is a $k$-valued point $y$ in $\bfM_\Lambda$ such that there is
a (non-canonical) isomorphism 
\begin{equation}
  \label{eq:10.2}
  {\calM^{(p)}_{K,*}}|^\wedge_{x}
\simeq \bfM_{\Lambda}|^\wedge_y 
\end{equation}
of formal local moduli spaces. This shows, in particular,
that if the local model $\bfM_\Lambda$ is flat over
$\Spec \Zp$, then the integral model $\calM^{(p)}_{K,*}\otimes \Zp$ is flat over
$\Spec \Zp$ \cite[p.~95]{rapoport-zink}. 

The morphism $\varphi^{\rm loc}$ is smooth, 
$\calG$-equivariant, and of relative dimension the same as 
$\varphi^{\rm mod}$. 
However, at this moment 
we do not know whether the morphism $\varphi^{\rm
  loc}$ is surjective. If this is so, then the integral model
$\calM^{(p)}_{K,*}\otimes \Zp$ is flat over $\Spec \Zp$ if and only if 
the local model $\bfM_\Lambda$ is flat over $\Spec \Zp$.


\subsection{A reduction step}
\label{sec:11.1}
Let $\Lambda$ and $\bfM_{\Lambda}$ be as above. 
Let 
\[ \Lambda=\oplus_{v|p} \Lambda_v \]
be the decomposition of $\Lambda$ obtained from the decomposition
$O_F\otimes \Zp=\prod_{v|p} \calO_v$, where $\calO_v$ is the ring of
integers in the local field $F_v$ of $F$ at $v$. Then we have 
\[ \bfM_{\Lambda}=\prod_{v|p} \bfM_{\Lambda_v}, \]
where the product $\Pi$ means the fiber product of the schemes
$\bfM_{\Lambda_v}$'s over $\Spec \Zp$ and 
$\bfM_{\Lambda_v}$ is the local model defined by the
lattice $\Lambda_v$ in the same  way as $\bfM_{\Lambda}$; see
Section~\ref{sec:10.1}. 

Write $O_B\otimes_\Z \Z_p=\prod_{v|p} \calO_{B_v}$ for the decomposition
with respect to $O_F\otimes \Zp=\prod_{v|p} \calO_v$. Then
$\calO_{B_v}$ is a maximal order in $B_v$.
Similarly we have the automorphism group
scheme $\calG_v=\Aut_{\calO_{B_v}}(\Lambda_v,\psi_v)$ associated to
the local lattice $(\Lambda_v,\psi_v)$, and we have the fiber product
decomposition
\[ \calG=\prod_{v|p} \calG_v. \] 

Now we fix a place $v$ of $F$ over $p$. Let 
$\calO_v^{\rm nr}\subset \calO_v$ be the maximal \'etale extension of
$\Zp$ in $\calO_v$ and put 
\[ I_v:=\Hom_{\Zp}(\calO_v^{\rm nr}, W). \] 
Let $e=e_v$ be the ramification index and $f=f_v$ be the inertia degree. 
Let $\pi$ be a uniformizer of $\calO_v$ and let $P(T)$ be the minimal
polynomial of $\pi$ over $\calO_v^{\rm nr}$. For any $\sigma\in I_v$
put $W_\sigma:=W[T]/(\sigma(P(T)))$ and denote by $\pi$ again the
image of $T$ in $W_\sigma$. One has $W_\sigma=W[\pi]$ and the element 
$\pi$ satisfies the equation $\sigma(P(T))=0$. We have the decomposition
\[ \Lambda_v\otimes_{\Zp} W=\oplus_{\sigma\in I_v} \Lambda_\sigma,
\quad \Lambda_\sigma:=\Lambda_v \otimes_{\calO_v^{\rm nr},\sigma}
W. \]
Write \[ \psi_\sigma:\Lambda_\sigma\times \Lambda_\sigma\to W \]
for the induced alternating pairing. 

Similarly we define the local model $\bfM_{\Lambda_\sigma}$ over
$\Spec W$ attached to each skew-Hermitian lattice
$(\Lambda_\sigma,\psi_\sigma)$. 
If $\scrF_v\subset \Lambda_v\otimes \calO_S$ is an object in
$\bfM_{\Lambda_v}$ and let $\scrF_v=\oplus_{\sigma\in I_v} \scrF_\sigma$
  be the natural decomposition, then every factor $\scrF_\sigma$ is a
  locally free $\calO_S$-module of rank $2me$; this follows from the
  determinant condition (K).
Therefore we have a natural isomorphism 
\[ \bfM_{\Lambda_v}\otimes W  \simeq \prod_{\sigma\in I_v}
\bfM_{\Lambda_\sigma}, \quad \scrF_v \mapsto (\scrF_\sigma)_{\sigma\in
  I_v}, \]
where the product $\Pi$ means the fiber product of the schemes
$\bfM_{\Lambda_\sigma}$'s over $\Spec W$.



\subsection{Maps to the local model}\label{sec:10.4}
A basic question in the local model diagram asks whether the morphism $\varphi^{\rm loc}$  is surjective. 
The local model diagram gives rise to a morphism of Artin stacks
\begin{equation}
  \label{eq:10.3}
  \theta: \calM^{(p)}_{K,*}\otimes \Zp \to [\calG\backslash
  \bfM_\Lambda],
\end{equation}
and this amounts to asking about the surjectivity of the map  
of the sets of geometric points
\begin{equation}\label{eq:10.4}
  \theta_k: \calM^{(p)}_{K,*}(k) \to \calG(k)\backslash
\bfM_\Lambda(k).
\end{equation}
Let ${\rm Dieu}^{O_B\otimes \Zp}_m(k)$ 
denote the set of isomorphism classes of separably quasi-polarized \dieu $O_B\otimes \Zp$-modules  
of $W$-rank $4md$ satisfying the determinant condition. 
The map $\theta_k$ factors through the natural map 
$\calM^{(p)}_{K,*}(k) \to {\rm Dieu}^{O_B\otimes \Zp}_m(k)$, $\ul A\mapsto M(\ul A)$, and let 
\begin{equation}
  \label{eq:10.5}
  \alpha: {\rm Dieu}^{O_B\otimes \Zp}_m(k)\to
  \calG(k)\backslash \bfM_\Lambda(k) 
\end{equation}
be the induced map. We will show in Section~\ref{sec:12} that $\alpha$ is surjective when $m=1$.

\section{Computation of local models for $m=1$}
\label{sec:11}


Keep the notation of the previous section.
For the remainder of this paper, assume that $m=1$.
We shall compute the geometric special fiber
$\bfM_{\Lambda_\sigma}\otimes k$ of $\bfM_{\Lambda_\sigma}$.  
Put $\calO_{B_\sigma}:=\calO_{B_v} \otimes_{\calO_v^{\rm nr}, \sigma} W$
for $\sigma\in I_v$.

\subsection{Unramified case.} 
\label{sec:11.2}

Suppose $v$ is unramified in $B$. Then
$\calO_{B_\sigma}=\Mat_2(W_\sigma)$. By the Morita equivalence
reduction as before, we have $\Lambda_\sigma= \Lambda_{\sigma,1}\oplus
\Lambda_{\sigma,2}$ and a unimodular Hermitian pairing   
\[ \varphi_\sigma: 
\Lambda_{\sigma,1}\times \Lambda_{\sigma,1}\to W. \]
Recall that $\varphi_\sigma(x,y)$ is the restriction of the symmetric
pairing $\psi(x,Cy)$ on the first factor $\Lambda_{\sigma,1}$, where
$C$ is the Weil element. 
The local model $\bfM_{\varphi_\sigma}$ associated to the symmetric
lattice $(\Lambda_{\sigma,1},\psi_\sigma)$ is defined by 
parametrizing the
$W_\sigma\otimes \calO_S$-submodules
$\scrF$ of  $\Lambda_{\sigma,1}\otimes \calO_S$ with the following
properties:
\begin{itemize}
\item [(i)] $\scrF$ is a locally free $\calO_S$-module of rank
$e$ and locally for Zariski topology on $S$
is a direct summand of $\Lambda_{\sigma,1}\otimes
\calO_S$;
\item [(ii)] $\scrF$ is isotropic with respect to the pairing
  $\varphi_\sigma$. 
\end{itemize}
Every object $\scrF_\sigma$  in $\bfM_{\Lambda_\sigma}$ has the
decomposition $\scrF_\sigma=\scrF_{\sigma,1}\oplus \scrF_{\sigma,2}$.
 
\begin{lemma}\label{11.1}
  The map which sends any object $\scrF_\sigma$ in
  $\bfM_{\Lambda_\sigma}$ to its first factor $\scrF_{\sigma,1}$
  induces an isomorphism of schemes
\[ \bfM_{\Lambda_\sigma} \simeq \bfM_{\varphi_\sigma}. \]
\end{lemma}
\begin{proof}
  It suffices to check that $\scrF_\sigma$ is isotropic with
  respect to the pairing $\psi_\sigma$ 
  if and only if $\scrF_{\sigma,1}$ is
  isotropic with respect to the pairing
  $\varphi_\sigma(x,y)=\psi_\sigma(x,Cy)$.
  Using $e_{11}^*=e_{22}$, we
  get $\psi_\sigma(\scrF_\sigma, \scrF_\sigma)=0$ if and only if
  $\psi_\sigma(\scrF_{\sigma,1}, \scrF_{\sigma,2})=0$. On the other
  hand, the isomorphism $C:\Lambda_{\sigma}\isoto \Lambda_\sigma$
  induces the isomorphism $C:\Lambda_{\sigma,1}\isoto
  \Lambda_{\sigma,2}$. Therefore,
  $\varphi_\sigma(\scrF_{\sigma,1},\scrF_{\sigma,1})=0$ if and only if
  $\psi_\sigma(\scrF_{\sigma,1}, \scrF_{\sigma,2})=0$. This shows the
  lemma. \qed 
\end{proof}

Put $\ol \Lambda_\sigma:=\Lambda_\sigma/p\Lambda_\sigma$ and $\ol
\Lambda_{\sigma,1}:=\Lambda_{\sigma,1}/p\Lambda_{\sigma,1}$.
Let $\scrD^{-1}_{W_\sigma/W}$ be the inverse different of the
extension $W_\sigma/W$ and choose a generator $\delta_\sigma$ of this
fractional ideal. Then there is a unique $W_\sigma$-valued 
$W_\sigma$-bilinear symmetric pairing 
\[ \varphi_\sigma':\Lambda_{\sigma,1}\times \Lambda_{\sigma,1}\to
W_\sigma \]   
such that 
$\varphi_\sigma(x,y)=\Tr \,[\delta_\sigma\cdot \varphi'_\sigma(x,y)]$. 
One can show that a $k[\pi]/(\pi^e)$-submodule 
$\scrF_{\sigma,1}\subset \ol \Lambda_{\sigma,1}$ is isotropic with
respect to the pairing 
$\varphi_\sigma$ if and only if so is it for the pairing
$\varphi'_\sigma$. 




Since $\Lambda_{\sigma,1}$ is a self-dual lattice and $k$ is
algebraically closed, we can choose a $W_\sigma$-basis 
$x_1,x_2$ for $\Lambda_{\sigma,1}$ such that 
\[ \varphi'_\sigma(x_1,x_1)=\varphi'_\sigma(x_2,x_2)=0\quad \text{and}\quad
\varphi'_\sigma(x_1,x_2)=\varphi'_\sigma(x_2,x_1)=1. \] Denote by $\bar x_i$, for $i=1,2$, the
image of $x_i$ in $\ol \Lambda_{\sigma,1}$.
Let $\scrF\subset \ol
\Lambda_{\sigma,1}$ be an object in $\bfM_{\varphi_\sigma}(k)$. As 
$\ol \Lambda_{\sigma,1}$ is a free $k[\pi]/(\pi^e)$-module of rank
two, one has
\[ \ol \Lambda_{\sigma,1}/\scrF \simeq k[\pi]/(\pi^{e_1}) \oplus
k[\pi]/(\pi^{e_2}) \]
for some integers $e_1,e_2$ with $0\le e_1\le e_2\le e$ and
$e_1+e_2=e$. The pair $(e_1,e_2)$ will be called the {\it Lie type} of
the object $\scrF$. We can write
\[ \scrF={\rm Span} \{\pi^{e_1} \bar y_1, \pi^{e_2} \bar y_2 \}, \]
where $\bar y_1$ and $\bar y_2$ generate $\ol  \Lambda_{\sigma,1}$ over
$k[\pi]/(\pi^e)$. Moreover, we can write either
\begin{itemize}
\item [(a)] $\bar y_1=\bar x_1+t \bar x_2$ and $\bar y_2=\bar x_2$, or
\item [(b)] $\bar y_1=t \bar x_1+\bar x_2$ and $\bar y_2=\bar x_1$, 
\end{itemize}
where $t\in k[\pi]/(\pi^e)$. We can represent $t$ as 
\[ t=t_0+t_1\pi+\dots +t_{e-2e_1-1} \pi^{e-2e_1-1}, \quad t_i\in k \]
because if $\ord_{\pi} (t)\ge e-2e_1$ then one can replace $\bar x_1+t
\bar x_2$ by $\bar x_1$ in the case $(a)$ (and the same for the case
(b)).  
Now one easily computes that
\[ \varphi'_\sigma(\scrF,\scrF)=0 \iff 2t \pi^{2e_1}=0. \]
This condition gives $t_0\pi^{2e_1}+\dots+t_{e-2e_1-1}\pi^{e-1}=0$ and
hence
\[ t_0=\dots=t_{e-2e_1-1}=0. \]
Therefore, we get two objects.
\[ \scrF={\rm Span}\{\pi^{e_1} \bar x_1, \pi^{e_2} \bar x_2\}, \quad 
\text{or}\quad \scrF={\rm Span}\{\pi^{e_1} \bar x_2, \pi^{e_2} \bar
x_1\}. \]
Notice that these two 
members are in the same orbit under the action of the
group $\calG_{\sigma}(k)$ as the automorphism of 
$\ol \Lambda_{\sigma,1}$
switching $\bar x_1$ and $\bar x_2$ lies in $\calG_\sigma(k)$, 
where 
\[ \calG_{\sigma}=
\Aut_{W_\sigma}(\Lambda_{\sigma,1},\varphi_\sigma) \] 
is the automorphism group scheme of the symmetric 
lattice $(\Lambda_{\sigma,1},\varphi_\sigma)$ over $W$. 
We obtain the following the result. 

\begin{prop}\label{11.2}
  Assume that $v$ is unramified in
  $B$ and let $\sigma\in I_v$. 
  Then $\bfM_{\varphi_\sigma}(k)$ consists of the
  $k[\pi]/(\pi^e)$-submodules 
\[ \scrF_{e_1}={\rm Span}\{\pi^{e_1} \bar x_1, \pi^{e-e_1} \bar
x_2\}, \quad \text{for\ } 0\le e_1\le e. \]
Moreover, two objects $\scrF_{e_1}$ and $\scrF_{e'_1}$ are in the same
orbit under the action of $\calG_\sigma(k)$ if and only if $e_1=e_1'$
or $e_1+e'_1=e$.  \qed
\end{prop}

\begin{prop}\label{11.3}  
  Assume that $v$ is unramified in
  $B$, and let $\sigma\in I_v$.  
\begin{enumerate}
  \item The special fiber $\bfM_{\varphi_\sigma}\otimes_W k$ is
    zero-dimensional and 
    two objects $\scrF$ and $\scrF'$ in $\bfM_{\varphi_\sigma}(k)$
    are in the same orbit under $\calG_\sigma(k)$ if and only if
    they have the same Lie type.
  \item The structure morphism
    $\bfM_{\varphi_\sigma}\to \Spec W$ is finite and flat. 
\end{enumerate}
\end{prop}
\begin{proof}
  (1) This follows immediately from Proposition~\ref{11.2}. 

  (2) Since the structure morphism is quasi-finite and projective, it is
      finite. We now show that any object $\scrF_0$ in
      $\bfM_{\varphi_\sigma}(k)$ can be lifted to an object $\scrF_R$
      over an integral domain $R$ with residue field $k$ and 
      fraction field $K$ of \ch
      zero. Then the coordinate ring of $\bfM_{\varphi_\sigma}$ is
      torsion-free as a $W$-module and hence is flat over $W$. 

      By Proposition~\ref{11.2}, write $\scrF_0=
     {\rm Span}\{\pi^{e_1} \bar x_1, \pi^{e_2} \bar
     x_2\}$ for two integers $e_1, e_2$ with $0\le
     e_1,e_2\le e$ and $e_1+e_2=e$. Write $W_\sigma=W[T]/(\sigma
     P(T))$. Let $R$ be the ring of integers in a finite separable
     field extension $K$ of $B(k)=\Frac(W)$ such that the polynomial
     $\sigma P(T)$ decomposes completely over $R$:
\[  \sigma P(T)=(T-\pi_1)\cdots (T-\pi_e)\in R[T]. \]
     Let $\pi_R$ be a uniformizer of $R$. We have $W_\sigma\otimes_W
     R=R[T]/(\sigma P(T))$. As $W_\sigma$ is a free $W$-module, we have
     an exact sequence:
\[ 
\begin{CD}
  0 @>>> W_\sigma\otimes_W (\pi_R) @>>> W_\sigma\otimes_W R @>>>
  W_\sigma\otimes_W k @>>>0. 
\end{CD} \]
  So an element $f(T)$ in $R[T]/(\sigma P(T))$ specializes to zero in
  $W_\sigma\otimes k=k[T]/(T^{e})$ if
  and only if $f(T)\in \pi_R \cdot R[T]/(\sigma P(T))$. We shall construct a
  $W_\sigma\otimes_W R$-submodule $\scrF_R\subset
  \Lambda_{\sigma,1}\otimes_W R$ such that 
  \begin{itemize}
  \item [(i)] $\scrF_R \otimes_R k=\scrF_0$;
  \item [(ii)] $\scrF_R$ and $(\Lambda_{\sigma,1}\otimes_W R)/\scrF_R$
    are both free of rank $e$ over $R$;
  \item [(iii)] $\scrF_R$ is isotropic with respect to the pairing
    $\psi_\sigma'$. 
  \end{itemize}
  Now we let $\scrF_R$ be the submodule generated by the elements 
$(T-\pi_1)\cdots (T-\pi_{e_1}) x_1$ and $(T-\pi_{e_1+1})\cdots
(T-\pi_{e}) x_2$. Clearly $\pi_i\in \pi_R R$ for all $i$ 
so one has (i). The
statement (ii) follows from (i) by the right exactness of the
tensor product. To check (iii), as 
$\scrF_R\subset \scrF_K:=\scrF_R\otimes K$,
it suffices to check (iii) for $\scrF_K$. Now 
we have 
\[ W_\sigma\otimes_W K=\prod_{i=1}^e  K\quad \text{ and }\quad
\scrF_K=(\scrF_{K,i})_{1\le i\le e}. \]
It is easy to see that each component $\scrF_{K,i}$ is one-dimensional
$K$-subspace generated by either $x_1$ or $x_2$ and hence 
$\scrF_K$ satisfies (iii). \qed 
\end{proof}

Let $\scrF_v$ be an object in $\bfM_{\Lambda_v}(k)$ and let
$\scrF_v=\oplus_{\sigma\in I_v} \scrF_\sigma$ be the natural
decomposition. The {\it reduced Lie type of $\scrF_v$} is defined to the
system of pairs $(e_{\sigma,1},e_{\sigma,2})$ indexed by $I_v$, where
$(e_{\sigma,1},e_{\sigma,2})$ is the Lie type of $\scrF_{\sigma,1}$. 
Proposition~\ref{11.3} 
immediately gives the following result.  

\begin{thm}\label{11.4}
  Suppose that $v$ is unramified in $B$. 
\begin{enumerate}
  \item The special fiber $\bfM_{\Lambda_v}\otimes_{\Zp} \Fp$ is
    zero-dimensional and 
    two objects $\scrF_v$ and $\scrF'_v$ in $\bfM_{\Lambda_v}(k)$
    are in the same orbit under the $\calG_v(k)$ if and only if
    they have the same reduced Lie type.  

  \item The structure morphism
    $\bfM_{\Lambda_v}\to \Spec \Zp$ is flat and finite. \qed
\end{enumerate}
\end{thm}

\subsection{Ramified case.} 

\label{sec:11.3}
Now we compute the local model $\bfM_{\Lambda_v}$ for the case where
$v$ is 
ramified in $B$. Recall that $\Lambda_v$ is a free
$\calO_{B_v}$-module of rank one together with a perfect $\Zp$-valued 
skew-Hermitian pairing $\psi_v:\Lambda_v\times \Lambda_v\to \Zp$. We
fix an unramified quadratic field extension $L_v\subset B_v$ as in
Section~\ref{sec:41}. Notice that the ring $\calO_{L_v}$ of
integers is 
contained in the unique maximal order $\calO_{B_v}$. 
We choose a presentation $\calO_{B_v}=\calO_{L_v}[\Pi]$ as in
(\ref{eq:41}) and (\ref{eq:42}).

Let
$\calO_{L^{\rm nr}_v}$ denote the maximal \'etale extension over $\Zp$
in $L_v$, and put $J_v:=\Hom_{\Zp}( \calO_{L^{\rm nr}_v},W)$. 
Let ${\rm pr}:J_v\to I_v$ be the restriction map 
from $\calO_{L_v^{\rm
    nr}}$ to $\calO_v^{\rm nr}$; this is a two-to-one map.
We have the decomposition
\[ \Lambda_v \otimes_{\Zp} W=\oplus_{\sigma\in I_v} \Lambda_\sigma,
\quad \Lambda_\sigma=\Lambda_\tau\oplus \Lambda_{\tau'}    \]
where $\{\tau, \tau'\}={\rm pr}^{-1}(\sigma)$ and $\Lambda_\tau$
(resp.~$\Lambda_{\tau'}$) is the $\tau$-component
(resp.~$\tau'$-component) of $\Lambda_v$. Notice that the pairing
$\psi_\sigma$ induces a perfect pairing
\[ \psi_\sigma: \Lambda_{\tau}\times \Lambda_{\tau'}\to W. \]
Let 
\[ \psi'_\sigma:\Lambda_{\tau}\times \Lambda_{\tau'}\to W_\sigma \]
be the unique $W_\sigma$-valued 
$W_\sigma$-bilinear pairing such that $\psi_\sigma(x,y)=\Tr
[\delta_\sigma \cdot \psi'_v(x,y)]$, where $\delta_\sigma$ is a generator of the inverse different $\calD^{-1}_{W_\sigma/W}$. 

The local model $\bfM_{\Lambda_\sigma}$ over $\Spec W$ parametrizes
the $W_\sigma\otimes_W \calO_S$-submodules
\[ \scrF_v=\scrF_\tau \oplus \scrF_{\tau'}\subset (\Lambda_\tau \oplus
\Lambda_{\tau'})\otimes \calO_S \]
 such that 
\begin{itemize}
\item [(i)] $\scrF_\tau$ and $\scrF_{\tau'}$ are locally free
  $\calO_S$-modules of rank $e$ and they are locally direct summands of
  $\Lambda_\tau \otimes \calO_S$ and $\Lambda_{\tau'} \otimes
  \calO_S$, respectively;

\item [(ii)] $\Pi(\scrF_\tau)\subset \scrF_{\tau'}$ and 
$\Pi(\scrF_{\tau'})\subset \scrF_{\tau}$;
\item [(iii)] $\psi_\sigma(\scrF_\tau,\scrF_{\tau'})=0$.
\end{itemize}

As $\scrF_\tau$ and $\scrF_{\tau'}$ are of rank $e$, condition 
(iii) says that
one is the orthogonal complement of the other and hence one submodule 
determines the other.  

We check that $\psi_\sigma(\scrF_\tau,\scrF_{\tau'})=0$ if
and only if $\psi'_\sigma(\scrF_\tau,\scrF_{\tau'})=0$. As
$\calD^{-1}=\calD^{-1}_{W_\sigma/W}$ is the largest
$W_\sigma$-submodule in 
$W_\sigma[1/p]$ such that $\tr(\calD^{-1})\subset W$,
one has $\tr(\pi^{-1} \calD^{-1})=p^{-1}W$. So $\tr
(\pi^{e-1} \calD^{-1})=W$. Consider the structure 
map $\phi:W\to \calO_S$. If $\ker \phi=0$, then 
$\psi'_\sigma(\scrF_\tau,\scrF_{\tau'})\neq 0$ implies 
$\psi_\sigma(\scrF_\tau,\scrF_{\tau'})\neq 0$. Suppose $\ker \phi=p^r
W$. If $\psi'_\sigma(\scrF_\tau,\scrF_{\tau'})\neq 0$, then 
\[ \delta_\sigma 
 \psi'_\sigma(\scrF_\tau,\scrF_{\tau'})\supset p^{r-1} \pi^{e-1}
\calD^{-1} \otimes_W \calO_S \]
Taking the trace one gets
\[ \psi_\sigma(\scrF_\tau,\scrF_{\tau'})\supset p^{r-1} 
W \otimes_W \calO_S \neq 0. \]
This verifies the assertion. 

By Lemma~\ref{42} we can choose a $W_\sigma$-basis
$x_1,x_2$ for $\Lambda_{\tau}$ and a 
$W_\sigma$-basis $x'_1,x'_2$ for $\Lambda_{\tau'}$ such that
 
\begin{equation}
  \label{eq:11.01}
  \psi_\sigma'(x_i, x'_j)=\delta_{i,j}, \quad \text{for}\quad 1\le
i,j\le 2
\end{equation}
and 
\begin{equation}
  \label{eq:11.02}
  \Pi(x_1)=x_1', \quad \Pi(x_2)=-\pi x'_2, \quad \Pi(x'_1)=-\pi x_1,
  \quad \Pi(x'_2)=x_2.
\end{equation}

Put $\ol \Lambda_{\tau}:= \Lambda_{\tau}/p\Lambda_{\tau}$ and 
$\ol \Lambda_{\tau'}:= \Lambda_{\tau'}/p\Lambda_{\tau'}$
Write $\bar x_i$ or $\bar x'_i$ for the image of $x_i$ or $x'_i$ in 
$\ol \Lambda_{\tau}$ or $\ol \Lambda_{\tau'}$, respectively.
Let $\scrF_\sigma=\scrF_\tau \oplus \scrF_{\tau'}$ be an object in
$\bfM_{\Lambda_\sigma}(k)$. One has 
\[ \ol \Lambda_{\tau}/\scrF_\tau \simeq k[\pi]/(\pi^{e_1})\oplus
k[\pi]/(\pi^{e_2}) \]
as $k[\pi]/(\pi^e)$-modules for some integers $e_1, e_2$ with $0\le
e_1\le e_2\le e$ and $e_1+e_2=e$; the pair $(e_1,e_2)$ is called the
{\it Lie type} of $\scrF_\tau$. It is easy to see that $\scrF_{\tau'}$
has the same Lie type as $\scrF_\tau$. The {\it reduced Lie type} of
  $\scrF_\sigma$ is defined to be the Lie type of $\scrF_\tau$. 
We call a reduced Lie type $(e_1,e_2)$ 
of an object $\scrF_\sigma$ {\it minimal}
if $e_2-e_1\in \{0,1\}$. 

Similar to the unramified case, we can write
\[ \scrF_\tau={\rm Span} \{\pi^{e_1} \bar y_1, \pi^{e_2} \bar y_2 \}, \]
where $\bar y_1$ and $\bar y_2$ are in one of the following cases
\begin{itemize}
\item [(a)] $\bar y_1=\bar x_1+t \bar x_2$ and $\bar y_2=\bar x_2$, or
\item [(b)] $\bar y_1=t \bar x_1+\bar x_2$ and $\bar y_2=\bar x_1$, 
\end{itemize}
where $t\in k[\pi]/(\pi^e)$.

In case (a), we compute 
\[ \scrF_{\tau'}={\rm Span} \{\pi^{e_1} (t \bar x'_1-\bar x'_2),
\pi^{e_2} \bar x'_1 \}. \]
As $\scrF_\tau$ and $\scrF_{\tau'}$ are orthogonal to each other, 
condition (ii) is equivalent to 
\begin{equation}
  \label{eq:11.1}
 \psi'_\sigma(\scrF_\tau,\Pi \scrF_\tau)=\psi'_\sigma(\Pi \scrF_{\tau'},
 \scrF_{\tau'})=0. 
\end{equation}
This yields the equation 
\begin{equation}
  \label{eq:11.2}
  \pi^{2e_1}(1-t^2 \pi)=0.
\end{equation}
If $e=2c+1$ is odd, then there is no solution to the equation
(\ref{eq:11.2}). If $e=2c$ is even, then a solution to (\ref{eq:11.2}) exists only when $(e_1,e_2)=(c,c)$, and any solution gives rise to the same object  
\begin{equation}
  \label{eq:11.3}
  \scrF_\tau=\pi^c \ol \Lambda_\tau \quad 
\text{and} \quad \scrF_{\tau'}=\pi^c \ol \Lambda_{\tau'}.
\end{equation}

In case (b), we compute 
\[ \scrF_{\tau'}={\rm Span} \{\pi^{e_1} (\bar x'_1-t\bar x'_2),
\pi^{e_2} \bar x'_2 \}. \]
The condition (\ref{eq:11.1}) yields the following equation
\begin{equation}
  \label{eq:11.4}
  \pi^{2e_1}(t^2-\pi)=0. 
\end{equation}
If $e=2c$ is even, then a solution to \eqref{eq:11.4} exists only when $(e_1,e_2)=(c,c)$. In this case we only get the object
$\scrF_\sigma$ as in (\ref{eq:11.3}).
If $e=2c+1$ is odd, then to have a solution we must have $(e_1,e_2)=(c,c+1)$ and we only get the object
\begin{equation}
  \label{eq:11.5}
  \scrF_\tau={\rm Span} \{\pi^{c} \bar x_2, \pi^{c+1} \bar x_1 \}
\quad \text{and} \quad \scrF_{\tau'}={\rm Span} \{\pi^{c} \bar x'_1,
\pi^{c+1} \bar x_2' \}. 
\end{equation}

\begin{prop}\label{11.5}
  Notations being as above, assume that $v$ is ramified in
  $B$, and let $\sigma\in I_v$. 
\begin{enumerate}
  \item If $e=2c$ is even, then $\bfM_{\Lambda_\sigma}(k)$ consists of
    the single $k[\pi]/(\pi^e)$-submodule
    $\scrF_\sigma=\scrF_\tau\oplus \scrF_{\tau'}$ with 
\[  \scrF_\tau=\pi^c \ol \Lambda_\tau \quad 
\text{and} \quad \scrF_{\tau'}=\pi^c \ol \Lambda_{\tau'}. \]

\item If $e=2c+1$ is odd, then  $\bfM_{\Lambda_\sigma}(k)$ consists of
    the single $k[\pi]/(\pi^e)$-submodule
    $\scrF_\sigma=\scrF_\tau\oplus \scrF_{\tau'}$ with 
\[   \scrF_\tau={\rm Span} \{\pi^{c} \bar x_2, \pi^{c+1} \bar x_1 \}
\quad \text{and} \quad \scrF_{\tau'}={\rm Span} \{\pi^{c} \bar x'_1,
\pi^{c+1} \bar x_2' \}, \]
where the bases $\{x_i\}$ and $\{x'_i\}$ are chosen as 
in (\ref{eq:11.01}) and (\ref{eq:11.02}). \qed
\end{enumerate}
\end{prop}

In particular, only the minimal reduced Lie type can occur in the space
$\bfM_{\Lambda_\sigma}(k)$.

\begin{prop}\label{11.6}  
  Assume that $v$ is ramified in
  $B$, and let $\sigma\in I_v$. The structure morphism
    $f:\bfM_{\Lambda_\sigma}\to \Spec W$ is finite and flat.  
\end{prop}
\begin{proof}
  As $f$ is projective and quasi-finite (Proposition~\ref{11.5}), the
  morphism $f$
  is finite. Let $B(k)^{\rm alg}$ be an algebraic closure of the
  fraction field $B(k)={\rm Frac}(W)$. Since
  $\bfM_{\Lambda_\sigma}(k)$ consists of one element, the
  specialization map 
\[ {\rm sp}:\bfM_{\Lambda_\sigma}(B(k)^{\rm alg}) \to
  \bfM_{\Lambda_\sigma}(k) \]
is surjective. Therefore, any (the unique) object in
  $\bfM_{\Lambda_\sigma}(k)$ can be lifted to \ch zero. This shows
  that the coordinate ring of $\bfM_{\Lambda_\sigma}$ is torsion free
  and hence $f$ is flat. \qed 
\end{proof}


\begin{thm}\label{11.7}
  Suppose $v$ is ramified in $B$. The structure morphism
    $\bfM_{\Lambda_v}\to \Spec \Zp$
is finite and flat. 
\end{thm}
\begin{proof}
  This follows from Proposition~\ref{11.6} immediately. \qed
\end{proof}

\subsection{Flatness of $\bfM_\Lambda$.}
\label{sec:11.4}

\begin{thm}\label{11.8}
  Let $\Lambda$ be a free unimodular skew-Hermitian 
  $O_B\otimes \Zp$-module of rank one and let $\bfM_\Lambda$ be the
  associated local model. The structure morphism $\bfM_\Lambda\to
  \Spec \Z_p$ is finite and flat.  
\end{thm}
\begin{proof}
  This follows from Theorems~\ref{11.4} and \ref{11.7}. \qed
\end{proof}


\begin{thm}\label{10.1}
  The moduli scheme $\calMpK\to \Spec \Z_{(p)}$ 
  is flat and every
  connected component is 
  projective and of relative dimension zero. \qed
\end{thm}

\section{More constructions of \dieu modules}
\label{sec:12}

In this section we handle two technical problems that arise from the
results of previous sections. Keep the notation and assumptions 
of Section~\ref{sec:11}.

\subsection{\dieu $\calO_\bfB$-modules with given Lie type}
\label{sec:12.1}
In \eqref{eq:10.5} we introduced the map 
$\alpha: {\rm Dieu}^{O_B\otimes \Zp}_m(k)\to \calG(k)\backslash \bfM_\Lambda(k)$. 
For each place $v|p$ of $F$, let ${\rm Dieu}^{\calO_{B_v}}_1(k)$ 
denote the set of
isomorphism classes of separably quasi-polarized \dieu $\calO_{B_v}$-modules of $W$-rank $4d_v$
satisfying the determinant condition, where $d_v=[F_v:\Qp]$.


Suppose $M=\oplus_{v|p}M_v$ is a \dieu $O_F\otimes \Zp$-module of ($W$-)rank $4d$ such that $\rank_W M_v=4d_v$.
Put $I=\coprod_{v|p} I_v$.
For each place $v|p$ and $i\in I_v$, let $e_{i,1}\le e_{i,2}\le  e_{i,3}\le e_{i,4}$ be the integers
such that  
\[ (M_v/\sfV M_v)^i\simeq k[\pi]/(\pi^{e_{i,1}})
\oplus k[\pi]/(\pi^{e_{i,2}}) \oplus k[\pi]/(\pi^{e_{i,3}})
\oplus k[\pi]/(\pi^{e_{i,4}}), \]
where $(M_v/\sfV M_v)^i$ denotes the $i$-component of $M_v/\sfV M_v$.
We define the {\it Lie type} of $M_v$ by 
\begin{equation}
  \label{eq:12.4}
  \ul e(M_v):=( \ul e_i)_{i\in I_v},\quad \ul e_i:=(e_{i,1},e_{i,2}, 
   e_{i,3},e_{i,4}). 
\end{equation}
The {\it Lie type} of $M$ is defined by
\begin{equation}
  \label{eq:12.5}
  \ul e (M):=(\ul e (M_v))_{v|p}.
\end{equation}
For two \dieu $O_F\otimes \Zp$-modules $M_1$ and $M_2$, 
one has $M_1/\sfV M_1\simeq  M_2/\sfV M_2$ as 
$O_F\otimes k$-modules if and only if $\ul e(M_1)=\ul e(M_2)$. 
If $(M_v, \<\ , \, \>)\in {\rm Dieu}^{\calO_{B_v}}_1(k)$, then 
$\ul e_i=(e_{i,1},e_{i,1},e_{i,2}, e_{i,2})$ for two  
integers $e_{i,1}$ and $e_{i,2}$ with $0\le e_{i,1}\le e_{i,2}\le e_v$ 
and $e_{i,1}+e_{i,2}=e_v$ for each $i\in I_v$, 
In this case we define the {\it reduced Lie type of $M_v$} 
and that of $M$, respectively, by 
\begin{equation}
  \label{eq:12.55}
  \ul e^r(M_v):=((e_{i,1}, e_{i,2})))_{i\in I_v} \quad \text{and \ } 
\ul e^r(M):=(\ul e^r (M_v))_{v|p}.
\end{equation}
 


\begin{prop}\label{12.1}
  The map $\alpha: {\rm Dieu}^{O_B\otimes \Zp}_1(k)\to \calG(k)\backslash   
  \bfM_\Lambda(k)$ is surjective. 
\end{prop}

\begin{proof}
  It suffices to show the surjectivity of the map 
  \begin{equation}
    \label{eq:12.6}
   \alpha_v: {\rm Dieu}^{\calO_{B_v}}_1(k)\to
  \calG_v(k)\backslash \bfM_{\Lambda_v}(k) 
  \end{equation}
  for each place $v|p$. 
  The target orbit space in (\ref{eq:12.6}) 
  is classified by the reduced Lie types of the
  objects (Theorem~\ref{11.4} and Proposition~\ref{11.5}). 
  When $v$ is unramified in $B$, this consists of tuples
  $(e_{i,1}, e_{i,2})_{i\in I_v}$ of pairs of integers
  with $0\le e_{i,1}\le e_{i,2}\le e_v$ and
  $e_{i,1}+e_{i,2}=e_v$. When $v$ is ramified in $B$, this
  consists of the single tuple $((c,e_v-c))_{i\in I_v}$, where
  $c:=\lfloor e_v/2 \rfloor$. 

  In the ramified case, by Theorem~\ref{72}
  there is a separably quasi-polarized \dieu $\calO_{B_v}$-module $M$
  with the determinant condition whose Lie type is the minimal one, that
  is, $(M/\sfV M)^j\simeq k[\pi]/(\pi^c)\oplus k[\pi]/(\pi^{e_v-c})$ for
  all $j\in \Z/2f_v\Z$. So one has the surjectivity of $\alpha_v$. 
  
  It remains to treat the 
  unramified case.
  We need to write down a separably
  anti-quasi-polarized \dieu $\calO_v$-module $M_1$ of rank $2d_v$ such
  that the Lie type $\ul e(M_1)$ of $M_1$ is equal to
  the given one $ ((e_{i,1},e_{i,2}))_{i\in I_v}$. 
  Fix an identification 
  $I_v\simeq \Z/f_v\Z$. Let $M_1=\oplus_{i\in \Z/f_v\Z}
  M_1^i$, where each $M_1^i$ is a free rank two  
  $W^i$-module generated by two
  elements $X_i$ and $Y_i$. For each $i\in \Z/f_v\Z$, let $(\,,\, ):M_1^i\times M_1^i\to W$ be the pairing that satisfies 
\begin{equation}\label{eq:81}
 (X_i, \pi^{e-1} Y_i)=1\quad \text{and}\  (X_i, \pi^{b} Y_i)=0,
 \quad 
 \forall\ 0\le b \le e-2  
\end{equation}
and $(ax,y)=(x,ay)$ for $a\in \calO$ and $x,y\in M^i_1$.
  Define
  the Verschiebung map $\sfV :M_1^{i+1}\to M_1^i$ by 
  \begin{equation}
    \label{eq:12.7}
    \sfV  X_{i+1}=\pi^{e_{i,1}} X_i, \quad \sfV  
    Y_{i+1}=p \pi^{e_v-e_{i,2}}
    Y_i.  
  \end{equation}
  It is easy to show that $(\sfV X,\sfV Y)=p(X,Y)^{\sigma^{-1}}$ for $X,Y\in
  M_1$ and that the Lie type $\ul e(M_1)$ of $M_1$ is equal to 
  $(e_{i,1},e_{i,2})_{i\in I_v}$. Therefore,  
  one has the surjectivity of $\alpha_v$. \qed
\end{proof}

\begin{remark}\label{12.2}
  The \dieu module $M_1$ constructed in the proof of
  Proposition~\ref{12.1} has the slope sequence
  \begin{equation}
    \label{eq:12.8}
  \nu(M_1)=\left \{
  \left ( \frac{\sum_{i} e_{i,1}}{d_v} \right )^{d_v},\left
  (\frac{\sum_{i} e_{i,2}}{d_v} \right )^{d_v} \right \}.  
  \end{equation}
This exhausts all possible slope sequences that can occur in
Corollary~\ref{76} in the case where $v$ is unramified in $B$.
\end{remark}

\subsection{Slope sequences of \dieu
$\calO_\bfB$-modules: a refinement}
\label{sec:12.2}
Our goal is to determine precisely slope sequences 
that arise from separably quasi-polarized \dieu $O_B\otimes \Zp$-modules 
satisfying the determinant condition (for $m=1$). 
This problem is local. So it suffices to treat 
\dieu $\calO_{B_v}$-modules for each place $v$ lying over
$p$. To simplify the notation as we did in 
Sections~\ref{sec:04}--\ref{sec:09}, 
we write $\bfB$, $\bfF$ etc.\ for $B_v$, $F_v$ etc.\ and
drop the subscript $v$ from our notation.   

Theorem~\ref{73} determines exactly 
all possible slope sequences for 
separably polarized \dieu $\calO_\bfB$-modules of rank
$4dm$, in particular for the present case of rank $4d$
(Corollary~\ref{76}). 
Recall that $d,e,f$ denote the degree, ramification index and
inertia degree of $\bfF$, respectively. 
The following result refines Corollary~\ref{76} 
for those \dieu modules in addition satisfying the determinant condition.

\begin{thm}\label{12.3} \ 
\begin{enumerate}
\item Suppose that $\bfB$ is the $2\times 2$ matrix algebra. 
Let $\nu$ be a slope
  sequence as follows:
\begin{equation}
  \label{eq:12.9}
   \nu =\left \{\left (\frac{1}{2}\right )^{4d}\right \}, \quad
  \text{or} \quad  \nu=\left \{\left (\frac{a}{d}\right )^{2d}, \left
 (\frac{d-a}{d}\right )^{2d}\right \}  
\end{equation}
for an integer $a$ with $0\le a< d/2$. Then there exists a separably
quasi-polarized \dieu $\calO_\bfB$-module $M$ of rank $4d$ satisfying the
determinant condition such that $\nu(M)=\nu$. 
\item Suppose that $\bfB$ is the quaternion division algebra. 
Suppose $M$ is a separably quasi-polarized \dieu $\calO_\bfB$-module of
rank $4d$ satisfying the determinant condition. Then 
\begin{equation}
  \label{eq:12.10}
   \nu(M) =\left \{\left (\frac{1}{2}\right )^{4d}\right \}, \quad
  \text{or} \quad  \nu(M)=\left \{\left (\frac{a}{2d}\right )^{2d}, \left
 (\frac{2d-a}{2d}\right )^{2d}\right \},  
\end{equation}
for an odd integer $a$ with $2\lfloor e/2\rfloor f \le a< d$.
Conversely, if $\nu$ is a slope sequence as in (\ref{eq:12.10}), then
there exists a separably
quasi-polarized \dieu $\calO_\bfB$-module $M$ of rank $4d$ satisfying the
determinant condition such that $\nu(M)=\nu$. 
\end{enumerate}
\end{thm}
\begin{proof}
  (1) This is proved in Proposition~\ref{12.1} and
      Remark~\ref{12.2}. 

  (2) Proposition~\ref{11.5} asserts that the reduced Lie type $\ul
      e^r(M)$ of $M$ is the minimal one $\{(c,e-c); i\in \Z/f\Z\}$,
      where 
      $c:=\lfloor e/2 \rfloor$. This yields $\sfF^{2f}(M)\subset \pi^{2fc} M$ and hence
      that smallest slope $\beta\ge 2cf/2d$. Then the first assertion
      follows from Corollary~\ref{76}. We now prove the second statement.
     
      Suppose that $\nu$ is a slope
      sequence as in (\ref{eq:12.10}). Suppose $e=2c$ is even. 
      Then $\nu$ is
      supersingular and this follows from Theorem~\ref{72}. 
      It remains to treat
      the case where $e=2c+1$ is odd. We may also assume that $\nu$ is
      non-supersingular as the supersingular case has been done 
      by Theorem ~\ref{72}. Write $a=2cf+2r+1$, where 
      $0< 2r+1<f$. 
Let 
\[ M:=\bigoplus_{j\in \Z/2f\Z} M^j, \]
where each $M^j$ is a free rank two $W^i$-module generated by elements
$X_j$ and $Y_j$ (with $i=j \mod f$). As before, we fix a presentation
$\calO_{\bfB}=\calO_{\bfL}[\Pi]$ as in (\ref{eq:41}) and
(\ref{eq:42}). We shall describe the Frobenius map $\sfF$ and the 
map $\Pi$ by
their representative matrices with respect to the bases
$\{X_i,Y_i\}$:
\[ \sfF_j:M^j\to M^{j+1}, \quad \Pi_j:M^j\to M^{j+f}, \quad \forall\,
j\in \Z/2f\Z  \]
in the sense that $\sfF_j=\begin{pmatrix} a_{11} & a_{12} \\ a_{21} & a_{22} \end{pmatrix}$ if the latter matrix presents $\sfF_j$ for the bases $\{X_j, Y_j\}$ and $\{X_{j+1}, Y_{j+1}\}$.

Put 
\begin{equation}
  \label{eq:12.11}
  \Pi_j=\begin{pmatrix} 0 & -\pi \\ 1 & 0 \end{pmatrix}, \quad \forall\,
j\in \Z/2f\Z. 
\end{equation}

For each $j\in \Z/2f\Z$, define a $W^i$-bilinear pairing 
\[ \<\, ,\>_\bfF:M^j\times M^{j+f} \to W^i \]
by 
\begin{equation}
  \label{eq:12.12}
  \begin{split}
  \<X_j,X_{j+f}\>_\bfF&=\<Y_j,Y_{j+f}\>_\bfF=0, \\  
\<X_j,Y_{j+f}\>_\bfF&=1 \quad \text{and}\quad \<Y_j,X_{j+f}\>_\bfF=-1.  
  \end{split}
\end{equation}
It is easy to show that $\<\Pi X, \Pi Y\>_{\bfF}=\pi \<X,Y\>_\bfF$ for
$X\in M^j$ and $Y\in M^{j+f}$. So $\<\,,\>_{\bfF}$ is an 
unimodular skew-Hermitian form on $M$ over $W\otimes \calO$. 
Put 
\begin{equation}
  \label{eq:12.13}
  \sfF_j=
  \begin{cases}
   \begin{pmatrix} 0 & -p\pi^{-c} \\ \pi^c & 0 \end{pmatrix}, & j=0, \\
   \begin{pmatrix} \pi^c & 0 \\ 0 & p\pi^{-c} \end{pmatrix}, & 1\le j
   \le r, \\ 
    \begin{pmatrix} p\pi^{-c} & 0 \\ 0 & \pi^c 
    \end{pmatrix}, & r< j<f. \\
  \end{cases}
\end{equation}
Using the commutative relation $\Pi_{j+1} F_j=F_{j+1} \Pi_j$,
we compute
\begin{equation}
  \label{eq:12.14}
  \sfF_j=
  \begin{cases}
   \begin{pmatrix} 0 & -\pi^{c+1} \\ p\pi^{-c}\pi^{-1} & 0   
    \end{pmatrix}, & j=f; \\
   \begin{pmatrix} \pi^c & 0 \\ 0 & p\pi^{-c} \end{pmatrix}, & f+1\le j
   \le f+r; \\ 
    \begin{pmatrix} p\pi^{-c} & 0 \\ 0 & \pi^c 
    \end{pmatrix}, & f+r< j<2f. \\
  \end{cases}
\end{equation} 
As the matrix coefficients of $\sfF_j$ lie in the image of $\Zp[\pi]$
and $\det \sfF_j=p$, one has $\<\sfF X,\sfF
Y\>_\bfF=p\<X,Y\>^\sigma_\bfF$ for 
$X\in M^{j}$ and $Y\in M^{j+f}$.

Putting $\<x,y\>:=\Tr_{\bfF/\Qp} (\delta \<x,y\>_\bfF)$, where
$\delta$ is a generator of $\calD^{-1}_{\bfF/\Qp}$, 
we obtain a separable
$\calO_{\bfB}$-linear quasi-polarization $\<\,,\>:M\times M\to W$. 
It is easy to see from our construction that $\dim_k (M/\sfV M)^j=e$
for all $j\in \Z/2f\Z$.  
Thus, by Lemma~\ref{52}, $M$ satisfies the determinant condition. 

We compute
\begin{equation}
  \label{eq:12.15}
  \sfF^f=p^r \begin{pmatrix} 0 & -(p\pi^{-c})^{f-2r} \\ (\pi^c)^{f-2r} &
  0 \end{pmatrix}: M^0\to M^f;  
\end{equation} 
\begin{equation}
  \label{eq:12.16}
  \sfF^f=p^r \begin{pmatrix} 0 & -(\pi^c)^{f-2r} \pi \\
  (p\pi^{-c})^{f-2r} \pi^{-1} &
  0 \end{pmatrix}: M^f\to M^0;  
\end{equation}
and 
\begin{equation}
  \label{eq:12.17}
  \sfF^{2f}=p^{2r} \begin{pmatrix} -(\pi^c)^{2(f-2r)} \pi & 0  \\
  0 &  - (p\pi^{-c})^{2(f-2r)} \pi^{-1}
 \end{pmatrix}: M^0\to M^0.  
\end{equation}
The valuation ($\ord_p$) of the first diagonal entry of the last matrix is
\[ [2er+2c(f-2r)+1]/e=(2cf+2r+1)/e=a/e. \]
This shows that the slope sequence of the \dieu module $M$ is equal to
$\nu$. This completes the construction of a desired \dieu module $M$
and hence completes the proof of the theorem. \qed  
\end{proof}

\section{Construction of Moret-Bailly families with $O_B$-action}
\label{sec:13}

Keeping our assumption $m=1$,  we assume in addition 
that $F=\Q$ in the remainder of this paper (Sections~\ref{sec:13} and
 \ref{sec:14}). 
In this section we assume in addition that {\it $p$ is ramified in $B$}.
Recall that $k$ denotes an \ac field of \ch $p$. 
We shall prove

\begin{thm}\label{13.1}
  There is a non-constant family of supersingular polarized $O_B$-abelian
  surfaces over $\bfP^1_k$. 
\end{thm}

\subsection{Case $B=B_{p,\infty}$}
\label{sec:13.2}
We begin with the construction of a Moret-Bailly family with $O_B$-action 
for the case $B=B_{p,\infty}$,
where $B_{p,\infty}$ is the definite quaternion $\Q$-algebra 
ramified exactly at $\{p, \infty\}$. 
Choose a supersingular elliptic curve $E$ over $k$. There is an isomorphism 
$B\simeq \End^0(E):=\End(E)\otimes \Q$ of $\Q$-algebras 
and we fix one. Then the 
endomorphism ring $\End(E)$ is a maximal order $O_B$ of $B$. The
subgroup scheme $E[\sfF]:=\ker \sfF=\alpha_p$ is $O_B$-stable 
as the Frobenius morphism is functorial. 
This induces a ring homomorphism 
\begin{equation}
  \label{eq:13.1}
  \phi: O_B/(p)=\F_{p^2}[\Pi]/(\Pi^2) \to \End_k(\alpha_p)=k.
\end{equation}
Since $k$ is commutative, this map factors through the maximal
commutative quotient 
$(\F_{p^2}[\Pi]/(\Pi^2))^{\rm ab}=\F_{p^2}[\Pi]/(\Pi^2,I)$, 
where $I$ is the two-sided ideal
of $\F_{p^2}[\Pi]/(\Pi^2)$ generated by elements of the form
$ab-ba$ for all $a, b\in \F_{p^2}[\Pi]/(\Pi^2)$. Since $\Pi a -
a\Pi=(a^p-a)\Pi$ and $a^p-a$ is invertible if $a\not\in \Fp$, 
the element $\Pi$ lies in $I$. 
This shows that the action of $O_B$ on
$E[\sfF]=\alpha_p$ factors through the quotient 
$O_B\twoheadrightarrow \F_{p^2}$. Put
$\Hom_{\Fp}(\F_{p^2},k)=\{\sigma_1, \sigma_2\}$. 
We may assume this action
is given by the embedding $\sigma_1:\F_{p^2}\to k$. 

Let $A_0:=E\times E$ and let $\iota_0:O_B \to \End(A_0)$ be the
diagonal action. Let $M$ be the \dieu module of $A_0$. Then the Lie
algebra $\Lie(A_0)=M/\sfV M$ has the decomposition of
$\sigma_i$-components (\ref{eq:45}):
\begin{equation}
  \label{eq:13.2}
  \Lie(A_0)=\Lie(A_0)^1\oplus \Lie(A_0)^2=k^2\oplus 0,
\end{equation}
and satisfies the condition 
\begin{equation}
  \label{eq:13.3}
  \Pi(\Lie(A_0))=0.
\end{equation}

Consider the functor $\calX$ over $k$ which sends each $k$-scheme $S$ to
the set of $O_B$-stable locally free finite subgroup schemes over $S$ of order $p$ in  $A_0[\sfF]_S=\alpha_{p,S}\times \alpha_{p,S}$.
Clearly this functor is representable by a
projective scheme (again denoted) $\calX\subset \bfP^1$ over $k$, 
as the condition of being $O_B$-stable is closed.
Each locally free finite subgroup $S$-scheme $H$ of order $p$ 
corresponds to a rank one locally
free $\calO_S$-submodule $\calL$ in $\calO_S^2=\Lie(A_0[\sfF]_S)$ 
which locally for Zariski topology is a direct
summand.
As the action of $O_B$ on $\Lie(E[\sfF]_S)$ is given by
the scalar multiplication through the map 
$\sigma_1: \F_{p^2} {\to} k\to \calO_S$, the condition that $\calL$ is
$O_B$-stable is automatically satisfied. This shows that
$\calX=\bfP^1$. 

Let $\bfH\subset A_0\times \bfP^1$ be the universal family, and let 
$\bfA:=(A_0\times \bfP^1)/\bfH$. 
As $\bfH$ is $O_B$-stable, we have 
a supersingular $O_B$-abelian scheme $(\bfA,\iota_{\bfA})$ over
$\bfP^1$. Ignoring the structure of the $O_B$-action, this family 
is constructed by Moret-Bailly  
and it is non-constant \cite[p.131]{moret-bailly:p1}. 
By \cite[Lemma~9.2]{kottwitz:jams92}, one can choose an $O_B$-linear
polarization 
$\lambda_0$ on $(A_0,\iota_0)$. 
Replacing $\lambda_0$ by $p\lambda_0$ if necessary 
one may assume that $\ker \lambda_0\supset A_0[\sfF]$. Since $\bfH$ is
isotropic with respect the Weil pairing defined by $\lambda_0$ (every
finite flat subgroup scheme of order $p$ has this property), the
polarization $\lambda_0\times \bfP^1$ on $A_0\times \bfP^1$ descends to a
polarization 
$\lambda_{\bfA}$ on $\bfA$, which is also $O_B$-linear. Therefore, we
have constructed a non-constant family of supersingular polarized
$O_B$-abelian surfaces $(\bfA,\lambda_{\bfA}, \iota_{\bfA})$ 
over $\bfP^1$ for $B=B_{p,\infty}$. 

\subsection{General case}
\label{sec:13.3}
Now let $B$ be an arbitrary definite quaternion algebra over $\Q$ ramified at $p$. 
By the construction above, we only need to
construct a superspecial $O_B$-abelian surface $(A_0,\iota_0)$ 
that satisfies the conditions (\ref{eq:13.2}) and (\ref{eq:13.3}).

We first find a superspecial $p$-divisible $O_B\otimes \Zp$-module
$(H_2,\iota_2)$ (of height 4) over $k$ such that the conditions
(\ref{eq:13.2}) and (\ref{eq:13.3}) for $\Lie(H_2)$ are satisfied. 
One can directly write 
down a superspecial \dieu $O_B\otimes \Zp$-module of rank $4$ 
with such conditions (see an example in Section~\ref{sec:13.5}) 
and let $(H_2,\iota_2)$ be the corresponding 
$p$-divisible $O_B\otimes \Zp$-module. Alternatively, let
$O_{B_{p,\infty}}$ be 
the maximal order in Section~\ref{sec:13.2} 
and $(A_0,\lambda_0)$ be the
superspecial $O_{B_{p,\infty}}$-abelian surface used there. After 
identifying 
$O_B\otimes \Zp$ with $O_{B_{p,\infty}}\otimes \Zp$, the attached
$p$-divisible $O_B\otimes \Zp$-module
$(H_2,\iota_2):=(A_0,\iota_0)[p^\infty]$ shares the desired property. 

Choose a supersingular $O_B$-abelian surface $(A_1,\iota_1)$. 
It exists by the non-emptiness of moduli spaces and Corollary~\ref{76} (2). Indeed, there exists a polarized 
$O_B$-abelian surface $(A_1,\lambda_1,\iota_1)$. Since $p$ is ramified in $B$, $A_1$ must be supersingular by Corollary~\ref{76} (2). Alternatively, one can first 
construct an embedding $B\to \Mat_2(B_{p,\infty})$ of $\Q$-algebras. One does this by choosing an appropriate quaternion algebra $B'$ such that $B\otimes B'\simeq M_2(B_{p,\infty})$. So we obtain a supersingular $B$-abelian surface $A'$ up to isogeny. Replacing $A'$ by an abelian surface in the isogeny class, we obtain a supersingular $O_B$-abelian surface. 
Let $(H_1,\iota_1):=(A_1,\iota_1)[p^\infty]$ be the associated
$p$-divisible $O_B\otimes \Zp$-module.  

\begin{lemma}\label{13.2}
  There is an $O_B\otimes \Zp$-linear isogeny $\varphi: (H_1,\iota_1)\to
  (H_2,\iota_2)$ over $k$.
\end{lemma}
\begin{proof}
  It suffices to find a $B_p$-linear quasi-isogeny $  
  (H_1,\iota_1)\to (H_2,\iota_2)$ in the category of $p$-divisible groups 
  over $k$ up to isogeny, where $B_p:=B\otimes \Qp$.
  Since $H_1$ and $H_2$ are supersingular, one chooses an isogeny
  $\varphi: H_1 \to H_2$. Define the map 
  $\iota'_2:B_p \to \End^0(H_1)$ by $\iota'_2(a):=\varphi^{-1} \iota_2(a)  
  \varphi$ for all $a\in B_p$. 
We have two
$\Qp$-algebra homomorphisms $\iota_1, \iota_2': B_p \to \End^0(H_1)$, and 
$\End^0(H_1)$ is a central simple $\Qp$-algebra, because $H_1$ is
supersingular.
By the Noether-Skolem theorem, there exists an element $g\in \End^0(H_1)^\times$
such that $\iota_2'=g\circ \iota_1 \circ g^{-1}$.   
That is, we have the following commutative diagram 
\[  \begin{CD}
    H_1 @>g>> H_1 @>\varphi>> H_2 \\
    @VV{\iota_1(a)} V@VV{\iota_2'(a)}V @VV\iota_2(a)V \\
    H_1 @>g>> H_1 @>\varphi>> H_2. \\ 
  \end{CD} 
\]
Replacing $\varphi$ by $\varphi\circ g$, we get a 
$B_p$-linear quasi-isogeny $\varphi:(H_1,\iota_1)\to
  (H_2,\iota_2)$. This proves the lemma. \qed 
\end{proof}

By Lemma~\ref{13.2}, we choose an $O_B\otimes \Zp$-linear isogeny
  $\varphi: 
  (H_1,\iota_1)\to (H_2,\iota_2)$. Let $K:=\ker \varphi$; this is an
  $O_B$-stable subgroup scheme of $A_1$. Let $A_0:=A_1/K$ and let
  $\iota_0: O_B\to \End(A_0)$ the induced action. Then one has
  an isomorphism $(A_0,\iota_0)[p^\infty]\simeq (H_2,\iota_2)$. This
  yields an $O_B$-abelian surface satisfying the conditions
  (\ref{eq:13.2}) and (\ref{eq:13.3}). 
  Applying the construction in 
  Section~\ref{sec:13.2} to this $(A_0,\iota_0)$, 
  we construct a non-constant family of polarized 
  $O_B$-abelian surfaces $(\bfA,\lambda_{\bfA}, \iota_{\bfA})$ 
  over $\bfP^1$. This finishes the proof of
  Theorem~\ref{13.1}. \qed

\subsection{An example.}
\label{sec:13.5}
We write down a superspecial 
\dieu $O_B\otimes \Zp$-module $M$ such that the conditions
(\ref{eq:13.2}) and (\ref{eq:13.3}) for $M/\sfV M$ are satisfied. 
Write $M=M^1\oplus M^2$ as a free $W$-module of rank $4$ 
with a $\Z_{p^2}$-action,
where $M^1$ and $M^2$ are free module generated by elements
$\{e^1_1,e^1_2\}$ and $\{e^2_1,e^2_2\}$, respectively.
Define the Verschiebung map $\sfV $ by 
\[ \sfV (e^2_1)=-pe^1_1, \quad \sfV (e^2_2)=-pe^1_2, \quad   \sfV (e^1_1)=e^2_1,
\quad \sfV (e^1_2)=e^2_2.  \] 
This determines the Frobenius map and defines a \dieu module with
a $\Z_{p^2}$-action which satisfies the condition (\ref{eq:13.2}) for
$M/\sfV M$. 
By the property $\Pi a=\sigma(a) \Pi$ for $a\in \Z_{p^2}$,
we must define $\Pi:M^1 \to  M^2$  and $\Pi: M^2 \to M^1$ such that
$\Pi^2=-p$ and $\Pi \sfV= \sf V \Pi$.
Note that the map $\Pi$ is determined by its restriction 
to $M^1$, $\Pi|_{M^1}:M^1 \to M^2$, because $\Pi^2=-p$.
The condition (\ref{eq:13.3}) implies 
that $\Pi(M^2)=pM^1$ and hence $\Pi(M^1)=M^2$. 


Clearly, $\Pi \sfV =\sfV \Pi$ if and only if 
$\Pi \sfV (e^i_j)=\sfV \Pi(e^i_j)$ for all $i,j$. 
We define $\Pi$ by setting 
$\Pi(e^i_j):=\sfV (e^i_j)$ for all $i,j$.
Then the conditions $\Pi \sfV =\sfV  \Pi$ and $\Pi^2=-p$ are satisfied. 
Therefore, this gives a 
superspecial \dieu $O_B\otimes \Zp$-module for which the conditions 
(\ref{eq:13.2}) and (\ref{eq:13.3}) are fulfilled.



\section{Dimensions of special fibers}
\label{sec:14}
Keep the setting of the previous section. 
We have proven that $\dim \calMpK\otimes \Fpbar=0$; see
Theorem~\ref{10.1}. 
Our goal is to determine the dimensions of 
the special fibers $\calM_{\Fpbar}:=\calM
 \otimes \Fpbar$, 
$\calM_{K, \Fpbar}:=\calM_{K}\otimes \Fpbar$ and
$\calM^{(p)}_{\Fpbar}:=\calM^{(p)} \otimes \Fpbar$. 

\begin{thm}\label{14.1} Assume that $m=1$ and $F=\Q$.
  \begin{enumerate}
  \item If $p$ is unramified in $B$, then $\dim \calM_{\Fpbar}=0$.
  \item If $p$ is ramified in $B$, then $\dim \calM_{\Fpbar}=1$.
  \item We have $\dim \calM^{(p)}_{\Fpbar}=0$. 
  \end{enumerate}
\end{thm}

A close examination shows that when $p$ is ramified in $B$, one has 
$\dim \calM_{K,\Fpbar}=1$; see Proposition~\ref{14.5}. This refines 
Theorem~\ref{14.1} (2).

\subsection{Unramified case}
\label{sec:14.1}
Suppose that $p$ is unramified in $B$. Let $(A,\lambda,\iota)$ be a
polarized $O_B$-abelian surface over $\Fpbar$. 
By Corollary~\ref{76}, $A$ is either ordinary or
supersingular. If $A$ is ordinary, then one has the canonical lifting
$(\bfA,\lambda_{\bfA}, \iota_{\bfA})$ over $W(\Fpbar)$ of
$(A,\lambda,\iota)$. Since the 
generic fiber $\calM_{\Qpbar}$ has dimension zero, each subscheme
$\calM_{D,\Qpbar}$ has finitely many points, if it is not empty. 
Recall that
$\calM_{D}\subset \calM$ denotes the subscheme parametrizing the
objects $(A,\lambda,\iota)$ in $\calM$ with polarization 
degree $\deg \lambda=D^2$.  
This implies that the
ordinary locus $\calM^{\rm ord}_{D,\Fpbar}$ of $\calM_{D,\Fpbar}$ has
finitely many points and hence it has dimension zero. 
Therefore, the ordinary locus $\calM^{\rm ord}_{\Fpbar}$ 
has dimension zero.  

Suppose now that $A$ is supersingular. Then $A$ must be
superspecial. To see this, let $H:=A[p^\infty]$ be the 
associated $p$-divisible group. Since $O_B\otimes
\Zp\simeq \Mat_2(\Zp)$, the $p$-divisible group $H$ is isomorphic to
$H_1\times H_2$, where $H_1$ and $H_2$ are supersingular $p$-divisible
groups of height $2$. Therefore, $A$ is superspecial.

For any positive integers $g$ and $D$, let $\calA_{g,D}$ denote the
coarse moduli space over $\Fpbar$ of $g$-dimensional 
polarized abelian varieties
$(A,\lambda)$ with polarization degree $\deg \lambda=D^2$.  
Let $\Lambda_{g,D}\subset \calA_{g,D}$ be the superspecial locus. It
is known that $\Lambda_{g,D}$
is a finite closed subscheme. 
Let $f:\calM_{D,\Fpbar}\to \calA_{2,D}$
be the forgetful morphism: $f(A,\lambda,\iota)=(A,\lambda)$. The
morphism $f$ induces a map
\[ f: \calM_{D,\Fpbar}^{\rm ss}\to \Lambda_{g,D}, \] 
where $\calM_{D,\Fpbar}^{\rm ss}$ is the supersingular locus of
$\calM_{D,\Fpbar}$. 
As $\dim \Lambda_{g,D}=0$ and the
forgetful map $f$ is finite (see \cite{yu:lift}), 
the supersingular locus $\calM_{D,\Fpbar}^{\rm ss}$     
also has dimension zero. We conclude that $\calM_{\Fpbar}$ 
has dimension zero. This proves Theorem~\ref{14.1} (1).

\subsection{Ramified case}
\label{sec:14.2}
Suppose that the prime $p$ is ramified in $B$. We know that any
polarized $O_B$-abelian surface over $k$ is supersingular
(Corollary~\ref{76}).  
By Theorem~\ref{13.1}, there is a non-constant family $\ul \bfA \to
\bfP^1_{\Fpbar}$ of supersingular polarized $O_B$-abelian surfaces. 
This gives rise to a non-constant moduli map 
\[ f': \bfP^1_{\Fpbar} \to \calM_{\Fpbar}. \]
Therefore, $\dim \calM_{\Fpbar}\ge \dim f'(\bfP^1)=1$. 

On the other hand, the forgetful morphism $f:\calM_{D,\Fpbar} \to
\calA_{2,D}$ factors through the supersingular locus 
$\calA_{2,D}^{\rm  ss}\subset \calA_{2,D}$. Since $f$ is finite, one
gets
\begin{equation}
  \label{eq:14.1}
  \dim \calM_{D,\Fpbar}\le \dim \calA_{2,D}^{\rm ss}.
\end{equation}
For any integer $i$ with $0\le i\le g$, 
let $\calA^{(i)}_{g,D}\subset \calA_{g,D}$ denote
the reduced locally closed subscheme that consists of objects
$(A,\lambda)$ of $p$-rank equal to $i$. Norman and Oort
\cite{norman-oort} showed that
the collection of $p$-strata forms a stratification and for each $i$
\[ \dim \calA^{(i)}_{g,D}=g(g-1)/2+i. \] 
When $g=2$, one has $\calA^{(0)}_{2,D}=\calA_{2,D}^{\rm ss}$ and gets 
$\dim \calA_{2,D}^{\rm ss}=1$. This proves the other direction 
\[ \dim \calM_{D,\Fpbar} \le 1.  \]
We conclude that $\dim \calM_{\Fpbar}=1$. This proves Theorem~\ref{14.1}
(2).



\subsection{Dimension of $\calM^{(p)}_{\Fpbar}$}
\label{sec:14.3}



We now prove Theorem~\ref{14.1} (3).
For the case where the prime $p$ is
unramified in $B$, we have shown in Theorem~\ref{14.1} (1) 
that $\calM_{\Fpbar}$ is zero-dimensional. Therefore, $\dim
\calM^{(p)}_{\Fpbar}=0$ in the unramified case. 

We now treat the case where 
$p$ is ramified in $B$. Recall that $k$ denotes an \ac field of \ch $p$, and 
  that $O_B\otimes \Zp=\Z_{p^2}[\Pi]$, $\Pi^2=-p$, $\Pi a=\sigma(a)\Pi$ 
  for $a\in \Z_{p^2}$.  

\begin{prop}\label{14.2}
  Assume that $p$ is ramified in $B$. Every prime-to-$p$ degree
  polarized $O_B$-abelian surface $(A,\lambda,\iota)$ over $k$
  is superspecial. 
\end{prop}
\begin{proof}
  Let $(M,\<\ , \,\>)$ be the covariant 
  \dieu $O_B\otimes \Zp$-module associated to $(A,\lambda,\iota)$.
  
  Suppose $M/\sfV M=k^2\oplus 0$ with respect to the action of
  $\F_{p^2}\otimes_{\Fp} k=k\times k$. Then $\sfV M^2=pM^1$ and
  $\sfV M^1=M^2$, and hence $\sfF M^1=M^2$ and $\sfF M^2=pM^1$. This proves that
  $\sfF M=\sfV M$ and that $M$ is superspecial 

  We show that if $\Pi(M/\sfV M)=0$, 
  then $M$ is
  superspecial. Indeed, it follows that $\Pi M\subset \sfV M$.
  From $\dim M/\Pi M=2$ (as $\Pi^2=-p$) and $\dim M/\sfV M=2$ it follows
  that $\Pi M=\sfV M$. Then $\sfV ^2M =\sfV \Pi M=\Pi \sfV M=\Pi^2
  M=pM$ and hence $M$ is superspecial.  
  So far we have not used the
  separability of polarizations.


Suppose that $M/\sfV M=k\oplus k$. Consider the induced 
perfect pairing $\<\,,\>:\ol M\times \ol M\to k$, where $\ol M=M/pM$.
Since $\Pi$ is nilpotent on $M/\sfV M$ one may assume for example that 
$\Pi \ol {M^2}=\sfV  \ol {M^2}$. Taking the orthogonal complements of 
$\Pi \ol {M^2}$ and $\sfV  \ol {M^2}$, we have 
$\Pi \ol {M^1}=\sfV  \ol {M^1}$. This proves that $\Pi(M/\sfV M)=0$. 
By the above argument, $M$ is superspecial. \qed
\end{proof}

Since the superspecial locus of the Siegel moduli space is
zero-dimensional, the superspecial locus of any PEL-type moduli space
is zero-dimensional, too. Proposition~\ref{14.2} implies that 
$\dim \calM^{(p)}_{\Fpbar}=0$. This proves Theorem~\ref{14.1} (3) and 
hence completes the proof of Theorem~\ref{14.1}. \qed 

\begin{lemma}\label{14.3}
  Assume that $p$ is ramified in $B$. There is a prime-to-$p$ degree
   polarized superspecial $O_B$-abelian surface over $k$ 
   that does not satisfy
  the determinant condition.  
\end{lemma}
\begin{proof}
  Using the construction in Section~\ref{sec:13.2}, we have a
  superspecial 
  $p$-divisible $O_B\otimes \Zp$-module $(H,\iota_H)$ of height 4 such
  that the 
  conditions (\ref{eq:13.2}) and (\ref{eq:13.3}) for $\Lie(H)$ are
  satisfied. Thus, $(H,\iota_H)$ does not  satisfy
  the determinant condition.  There is a superspecial abelian 
  $O_B$-surface $(A_0,\iota_0)$ such that $(A_0,\iota_0)[p^\infty]\simeq
  (H,\iota_H)$. We fix an identification $(A_0,\iota_0)[p^\infty]=
  (H,\iota_H)$. 

  We can choose a separable $O_B\otimes \Zp$-linear
  quasi-polarization $\lambda_H$. 
  Note that $(H,\iota_H)$ is isomorphic to the $p$-divisible
  $O_{B_{p,\!\infty}}\otimes \Zp$-module $E_0[p^\infty]^2$ ($E_0$ is
  a supersingular elliptic curve) through the identification
  $O_{B_{p,\!\infty}}\otimes\Zp=O_B\otimes \Zp$. We can pick the
  product principal polarization on $E_0^2$ which yields such a
  quasi-polarization $\lambda_H$.        

  Choose an $O_B$-linear polarization $\lambda$ on $(A_0,\iota_0)$ and
  let $*$ denote the Rosati involution induced by $\lambda$. Then
  $\lambda a_p=\lambda_H$ for some element 
  $a_p\in \End^0_{O_B\otimes \Zp} (A_0[p^\infty])$ with $a_p^*=a_p$. 
  Since $\End_{B}^0(A_0)\otimes
  \Qp=\End^0_{B\otimes\Qp}(A_0[p^\infty])$, 
  by weak approximation we can 
  choose a totally positive symmetric element $a\in \End^0_{B}(A_0)$
  such that $a$ is sufficiently close to $a_p$. Then we have $(A,\lambda a,
  \iota)[p^\infty] \simeq (H,\lambda_H,\iota_H)$. Replacing $a$ by
  $N a$ for a positive prime-to-$p$ integer $N$ if necessary, we get a
  prime-to-$p$ degree $O_B$-linear polarization $\lambda_0=\lambda a$ 
  on $(A_0,\iota_0)$. This proves the lemma. \qed     
\end{proof}

We know that when $p$ is ramified in $B$, the whole moduli space
$\calM_{\Fpbar}$ is supersingular. On the other hand when $p$ is
unramified in $B$, the 
moduli space may have both supersingular and ordinary points according
to Corollary~\ref{76}. 
The following lemma says that this is indeed the case. 

\begin{lemma}\label{14.35}
  When $p$ is unramified in $B$, the moduli space
  $\calM^{(p)}_{K,\Fpbar}$ consists of both ordinary and supersingular
  points. 
\end{lemma}
\begin{proof}
  As $p$ is unramified in $B$, the determinant condition is automatically
  satisfied for objects in $\calM(k)$. 
  Let $(H,\lambda_H,\iota_H)$ be a supersingular or ordinary separably
  quasi-polarized $p$-divisible $O_B\otimes \Zp$-module. 
  Since $B$ can
  be embedded into $\End^0(A)$ for any supersingular abelian surface,
  we can find a supersingular $O_B$-abelian surface with
  $(A_0,\iota_0)[p^\infty] \simeq (H,\iota_H)$. 
  We use the argument in the proof of
  Lemma~\ref{14.3} again to obtain a prime-to-$p$ degree $O_B$-linear
  polarization. For the ordinary case, we choose 
  any imaginary quadratic field $K$ such that $K$ splits $B$ and $p$
  splits in $K$. Then
  there is a $\Q$-algebra embedding of $B$ into $\Mat_2(K)$. Choose an
  ordinary elliptic curve $E$ such that $\End^0(E)\simeq K$ and take the
  ordinary abelian surface $A=E^2$. As $\End^0_B(A)\otimes
  \Qp=\End^0_{B\otimes \Qp}(A[p^\infty])$, we can repeat the previous
  argument and get a prime-to-$p$ degree polarized ordinary $O_B$-abelian
  surface. \qed 
\end{proof}

\begin{remark}\
(1) The proof of Theorem~\ref{14.1} does not use local models. 
The method of local models we understand is 
developed in \cite{rapoport-zink}, where
Rapoport and Zink formulate the moduli spaces of polarized self-dual
(multi-)chains of 
$O_B$-abelian varieties together with the determinant condition. The
moduli spaces considered in Theorem~\ref{14.1} 
parametrize polarized $O_B$-abelian varieties only
and are not exactly the same as
those in loc.~cit. In order to apply the method of local models in our
case, one
needs to first extend the
objects $(A,\lambda,\iota)$ over a base scheme to a
self-dual chain of $O_B$-abelian schemes, and then apply the theorem
``normal forms of lattice chains'' \cite[Theorem
3.16]{rapoport-zink}. To do that
the following closed
condition must hold: there is an $O_B$-linear 
isogeny $\varphi:A^t \to A$ such that
$\varphi\circ \lambda=p$. In the cases of Theorem~\ref{14.1}, this is
the case where $\deg \lambda$ is at most divisible by $p^2$ 
and in this case we use local models to
calculate the dimension;
see Corollary~\ref{14.10} and Section~\ref{sec:11}.
Note that the degree of the polarization of objects here
can be divisible by a high power of $p$ and 
we also consider the objects which do not satisfy the determinant 
condition.


We remark that there is a subtle difference between the moduli space of
polarized self-dual chains of $O_B$-abelian varieties and that of 
polarized $O_B$-abelian varieties in the chain.
This may be seen from
the fact the morphism of deformation functors 
${\rm Def}[A_0,\lambda_0]\to
{\rm Def} [A_0, p \lambda_0]$, sending $(A,\lambda)$ to
$(A,p\lambda)$, is not an isomorphism (for the sake of
simplicity, we disregard the $O_B$-structures). For example, if
$\lambda_0$ is principal, then the deformation functor 
${\rm Def}[A_0,\lambda_0]$ is formally smooth while ${\rm
  Def}[A_0,p\lambda_0]$ is not. Furthermore, the universal deformation
$(\wt A,\wt \lambda)$ of $(A_0,p\lambda_0)$ does not
give rise to a self-dual chain of $O_B$-abelian schemes because the
multiplication by $p$ is not divisible by the
polarization $\wt \lambda$.

The referee points out that one still can apply local models for
computing some subvariety in which
points may have polarization highly divisible by $p$ in the
following way. Consider the locus $\calN$ of
$\calM_{\Fpbar}$ consisting of points with polarization degree 
divisible by at most $p^2$ 
and let $\calV:=\cup_{n=1}^\infty \phi_p^n(\calN)\subset
\calM_{\Fpbar}$, where
$\phi_p:\calM_{\Fpbar}\to \calM_{\Fpbar}$ sends $(A,\lambda,\iota)$ to
$(A,p \lambda, \iota)$. Then we have $\dim \calV=\dim \calN$ and we
can use local models to compute $\dim \calN$.
We thank the referee for this comment. 

(2) We directly show that the method of local models does not work well for
the previous example $(A_0, p\lambda_0)$ over $k$, where $\lambda_0$
is a principal polarization.
Let $(\Lambda, \psi)$ be a symplectic $\Zp$-lattice of
rank $4$ with elementary divisor type $(p,p)$. We can choose a symplectic
basis $\Zp$-basis $e_1, e_2, f_1, f_2$ with
$\psi(e_i,f_j)=p\,\delta_{ij}$ such that the Hodge filtration of
$H^{DR}_1(A_0)$ corresponds to the $k$-subspace generated by
$\{f_1,f_2\}$ via a trivialization 
$(M(A_0), \<\,, \>)=(\Lambda\otimes W(k),\psi)$. The universal
deformation of this point is then spanned by $\wt f_1:=f_1+t_{11} e_1+t_{12}
e_2$ and $\wt f_2:=f_2+t_{21}e_1+t_{22}e_2$ subject to the condition
$\psi(\wt f_1,\wt f_2)=0$. Thus, the universal deformation ring 
is  isomorphic to
$W(k)[[t_{11},t_{12}, t_{21},t_{22}]]/(p(t_{11}-t_{22}))$ and in
particular the local model $\bfM_{\Lambda}$ is not flat over $\Zp$.
However, by a theorem of Mumford \cite[Theorem 2.3.3]{oort:oslo}
(cf.~\cite[Theorem 4.5]{yu:lift})  
the moduli scheme $\calA_{2,p^2}$ is flat over $\Zp$. Therefore, the
local model does not control of the singularity of the moduli scheme
$\calA_{2,p^2}$ at the point $(A_0,p\lambda_0)$.

(3) We used local models to prove $\dim \calMpK\otimes \Fpbar =0$ (Theorem~\ref{10.1}). 
Proposition~\ref{14.2} gives a different proof of this result.  
Lemma~\ref{14.3} shows that 
the inclusion $\calMpK(k)\subset \calM^{(p)}(k)$ is
strict at least when  $B\otimes \Qp$ is a division $\Qp$-algebra. 
This phenomenon is different from the reduction modulo $p$ of
Hilbert moduli schemes or Hilbert-Siegel moduli schemes. 
In the Hilbert-Siegel case, any separably polarized
abelian variety with RM by $O_F$ of a totally real field $F$
satisfies the determinant condition automatically; see Yu
\cite{yu:reduction}, G\"ortz~\cite{goertz:topological} and
Vollaard~\cite{vollaard:thesis}.       

(4) We know that the moduli space $\calM^{(p)}_{\Fpbar}$ is non-empty
    (Lemma~\ref{215}). When $p$ is
    ramified in $B$, the moduli space $\calM_{\Fpbar}$ consists of
    both one-dimensional components (e.g.\ Moret-Bailly
    families) and zero-dimensional components (e.g. points in
    $\calM^{(p)}_{\Fpbar}$).  
\end{remark}

\subsection{Dimension of $\calM_{K,\Fpbar}$}
\label{sec:14.4}

Using Theorem~\ref{14.1} (1), we only need to treat the ramified case.

\begin{lemma}\label{14.4}
  Assume that $p$ is ramified in $B$. Let $M_0$ be a \dieu $O_B\otimes
  \Zp$-module over $k$ such that
\[ M_0/\sfV M_0=k^2 \oplus 0, \]
that is the Lie type of $M_0$ is $(2,0)$ with respect to the action of $\Z_{p^2}$. 
  Let $M$ be any \dieu module such that $\sfV M_0\subset M \subset M_0$
  and $\dim_k (M_0/M)=1$. Then one has
\[ M/\sfV M=k\oplus k. \] 
\end{lemma}
\begin{proof}
  Choose bases $\{X^1_1, X^1_2\}$ and $\{X^2_1, X^2_2\}$ for 
  $M_0^1$ and $M_0^2$, respectively. Since $M_0/\sfV M_0=k^2 \oplus 0$,   
  $M_0$ is superspecial (see the proof of Proposition~\ref{14.2}). 
  Since $\sfV M\supset \sfV ^2M_0=pM_0$, we
  can check the statement by passage to  $\ol M_0:=M_0/pM_0$. 
  Write $x^i_j$ for the image
  of $X^i_j$ in $\ol M_0$. One has
\[ \sfV \ol M_0={\rm Span}_k\{x^2_1, x^2_2\}, \quad 
\ol M:=M/pM_0={\rm Span}_k\{x^2_1, x^2_2, ax^1_1+bx^1_2\}, \]
for some $(a,b)\neq (0, 0)\in k^2$. Then 
\[  \sfV \ol M={\rm Span}_k \{\sfV (ax^1_1+bx^1_2)\}\subset \ol M_0^2, \quad
\text{and }\quad \dim \sfV \ol M=1. \]
This gives $M/\sfV M=k\oplus k$. \qed
\end{proof}

\begin{prop}\label{14.5}
  Assume that $p$ is ramified in $B$. We have $\dim \calM_{K,\Fpbar}=1$.
\end{prop}
\begin{proof}
  In the previous section, we construct a polarized $O_B$-abelian
  surface $(\bfA,\lambda_\bfA, \iota_\bfA)$ over $\bfP^1$
  starting from a superspecial abelian surface 
  $(A_0,\lambda_0,\iota_0)$ with additional structures and get
  a non-constant moduli map $f':\bfP^1 \to \calM_{\Fpbar}$. The
  \dieu module $M_0$ of $A_0$ has the property $M_0/\sfV M_0=k^2\oplus
  0$. By Lemma~\ref{14.4}, every fiber of the family
  $(\bfA,\lambda_\bfA, \iota_\bfA)\to \bfP^1$ has Lie type
  $(1,1)$. Then the image $f'(\bfP^1)$ is contained in
  $\calM_{K,\Fpbar}$. This shows that $\dim \calM_{K,\Fpbar}=1$. \qed   
\end{proof}

\begin{remark}\label{14.6}
  As a consequence of Proposition~\ref{14.5}, we have the following
  result: When $p$ is
  ramified in $B$, there is a
  polarized $O_B$-abelian surface over $k$ with the determinant
  condition that can not be lifted to a polarized $O_B$-abelian surface
  in \ch zero. 
\end{remark}

\subsection{Dimension of $\calM_{K,D,\Fpbar}$}
\label{sec:14.5}
For an integer $D\ge 1$, let
$\calM_{K,D,\Fpbar}:=\calM_{K,D}\otimes_{\Z_{(p)}} \Fpbar$, where
  $\calM_{K,D}=\calM_K\cap \calM_D$. We would like to determine the
  dimension of $\calM_{K,D,\Fpbar}$ if it is non-empty. By
  Theorem~\ref{14.1}, $\dim \calM_{K,D,\Fpbar}=0$ if $p$ is unramified
  in $B$ or $p\nmid D$. Thus, we shall assume that $p$ is ramified in
  $B$ and $p\mid D$. For the case where $p|| D$, we use local models
  to compute the dimension. 

Let $(\Lambda_1,\psi)$ 
be a skew-Hermitian $O_B\otimes \Zp$-module of $\Zp$-rank
$4$ of discriminant $(p^2)$. We may write
$\Lambda_1=O_B\otimes \Zp=\Z_{p^2}[\Pi]$ and 
$\psi(x,y)=\tr_{\bfB/\Qp} (x\xi
y^*)$ for a skew-Hermitian element $\xi\in B_p^\times$. The
condition $\disc \psi=(p^2)$ implies that $\xi\in (O_B\otimes \Zp)^\times$.
Up to isomorphism the element $\xi$ is uniquely determined up to a
scalar in $(\Z_{p}^\times)^2$. 

Let $\bfM_{\Lambda_1}$ be the local model over $\Zp$ 
associated to the lattice $\Lambda_1$ defined as in
Section~\ref{sec:10.1}. 

\begin{lemma}\label{14.9}
  The special fiber $\bfM_{\Lambda_1, {\Fpbar}}$ is a union of two projective
  lines meeting at one point.
\end{lemma}
\begin{proof}
  Write $\Lambda_{1,W}:=\Lambda_1\otimes W=\oplus_{j\in
  \Z/2\Z}\Lambda_{1,W}^j$. 
  We can choose 
  a $W$-base $\{e^j_1,e^j_2\}$ for
  $\Lambda_{1,W}^j$ for each $j\in \Z/2\Z$ such that  
\begin{equation}
\label{eq:14.2}
  \begin{split}
  &\Pi e^j_1=e^{j+1}_2, \quad \Pi e^j_2=-p e^{j+1}_1,\\
   \psi(e^1_1,e^2_1)=1, \quad  &\psi(e^1_2,e^2_2)=-p, \quad
   \psi(e^i_r,e^j_s)=0\quad \text{if $i=j$ or $r\neq s$.}
  \end{split}
  \end{equation}
   To see this,
    write $O_B \otimes_{\Z} W=(W\times W)[\Pi]$ and let $e^1_1=(1,0)$,
    $e^1_2=(1,0)\Pi$,  $e^2_1=(0,1)$,
    $e^2_2=(0,1)\Pi$. Up to $(W^{\times})^2=W^\times$, 
    we may choose $\xi$ to be 
    $(1,-1)$. Then the
    conditions in \eqref{eq:14.2} are satisfied. 

  Any point in $\bfM_{\Lambda_1}({\Fpbar})$ is given by an ${\Fpbar}$-subspace
  $\scrF=\scrF^1\oplus\scrF^2\subset \Lambda_{1,{\Fpbar}}^1\oplus
  \Lambda_{1,{\Fpbar}}^2$ such that 
  \begin{equation}
    \label{eq:14.3}
    \dim_{\Fpbar} \scrF^j=1,\quad \Pi(\scrF^j)\subset \scrF^{j+1}, \quad
    \psi(\scrF^1,\scrF^2)=0 
  \end{equation}
  for $j\in \Z/2\Z$. Write $\scrF^j=\< t_{j1} e^j_1+t_{j2}
  e^j_2\>$ with $[t_{j1}:t_{j2}]\in \bfP^1({\Fpbar})$ for $j\in \Z/2\Z$. 
  The conditions in \eqref{eq:14.3} is then given by the condition
  $t_{11} t_{21}=0$. Thus, $\calM_{\Lambda_1,\Fpbar}=\{[0:1]\}\times
  \bfP^1\cup \bfP^1\times \{[0:1]\}$ with two $\bfP^1$'s 
  meeting at the point $([0:1],[0:1])$. \qed 
\end{proof}

\begin{cor}\label{14.10}
 If $p$ is ramified in $B$, $p||D$ and $\calM_{K,D,\Fpbar}$ is
 non-empty, then $\dim \calM_{K,D,\Fpbar}=1$.  
\end{cor}
\begin{proof}
  Choosing an auxiliary prime to $pD$ level structure, we have a
  finite surjective cover $\calM_{K,D,N}\to \calM_{K,D}$. Using the
  local model diagram, we have $\dim \calM_{K,D,N,\Fpbar}=\dim
  \bfM_{\Lambda_1,\Fpbar}=1$ by Lemma~\ref{14.9}. It follows that
  $\dim \calM_{K,D,\Fpbar}=1$. \qed
\end{proof}

\begin{question}
  We mention a few  problems which are not solved in this paper.

(1) What is the dimension of $\calM_{K,D,\Fpbar}$ if it is non-empty
    and $p^2|D$? We expect the answer to be one.

(2) We prove in Theorem~\ref{10.1} that $\calMpK\to \Spec \Z_{(p)}$ is
    finite and flat. Which points in $\calMpK$ are \'etale
    over $\Spec \Z_{(p)}$?
    How about the same questions for $\calM_K$ and $\calM^{(p)}$ when
    $F=\Q$ in Theorem~\ref{14.1}? 

(3) What are the dimensions of variant moduli spaces in
    Theorem~\ref{14.1} where $F$ is an arbitrary totally real field
    and $m=1$?

(4) Is the map $\varphi^{\rm loc}$ (and its variant for the canonical
    local model considered in \cite{pappas-rapoport} or
    \cite{pappas-zhu})  
    in Section~\ref{sec:10.2} surjective?   
\end{question}








\section*{Acknowledgments}
  Parts of the present work were done while the author's visits at the
  RIMS, Kyoto University, the IEM, Universit\"at Duisburg-Essen, 
  the IMS, Chinese University of Hong Kong, the PMI of POSTECH, and
  the Max-Planck-Institut f\"ur Mathematik in Bonn. 
  He wishes to thank Akio Tamagawa and Ulrich G\"ortz for helpful
  discussions and these institutions for kind hospitality 
  and excellent working conditions. The author was 
  partially supported by the grants MoST 100-2628-M-001-006-MY4,
  103-2918-I-001-009 and 107-2115-M-001-001-MY2. He is indebted to the
  referee for a long list of helpful comments which have eliminated
  several inaccuracies and improved the exposition significantly.  


\begin{thebibliography}{99}
\def\jams{{\it J. Amer. Math. Soc.}} 
\def\invent{{\it Invent. Math.}} 
\def\ann{{\it Ann. Math.}} 
\def\ihes{{\it Inst. Hautes \'Etudes Sci. Publ. Math.}} 


\def\acta{{\it Acta Math.}}
\def\ecole{{\it Ann. Sci. \'Ecole Norm. Sup.}}
\def\ecole4{{\it Ann. Sci. \'Ecole Norm. Sup. (4)}} 
\def\mathann{{\it Math. Ann.}} 
\def\duke{{\it Duke Math. J.}} 
\def\jag{{\it J. Algebraic Geom.}} 
\def\advmath{{\it Adv. Math.}}
\def\compos{{\it Compositio Math.}} 
\def\ajm{{\it Amer. J. Math.}} 
\def\grenoble{{\it Ann. Inst. Fourier (Grenoble)}}
\def\crelle{{\it J. Reine Angew. Math.}}
\def\mrl{{\it Math. Res. Lett.}}
\def\imrn{{\it Int. Math. Res. Not.}}
\def\acad{{\it Proc. Nat. Acad. Sci. USA}}
\def\tams{{\it Trans. Amer. Math. Sci.}}
\def\cras{{\it C. R. Acad. Sci. Paris S\'er. I Math.}} 
\def\mathz{{\it Math. Z.}} 
\def\cmh{{\it Comment. Math. Helv.}}
\def\docmath{{\it Doc. Math. }}
\def\asian{{\it Asian J. Math.}}
\def\jussieu{{\it J. Inst. Math. Jussieu}} 
\def\plms{{\it Proc. London Math. Soc.}}
\def\bams{{\it Bull. Amer. Math. Soc.}}

\def\jlms{{\it J. London Math. Soc.}}
\def\blms{{\it Bull. London Math. Soc.}}
\def\manmath{{\it Manuscripta Math.}} 
\def\jnt{{\it J. Number Theory}} 
\def\ijm{{\it Israel J. Math.}}
\def\ja{{\it J. Algebra}} 
\def\pams{{\it Proc. Amer. Math. Sci.}}
\def\smfmemoir{{\it Bull. Soc. Math. France, Memoire}}
\def\bsmf{{\it Bull. Soc. Math. France}}
\def\sb{{\it S\'em. Bourbaki Exp.}}
\def\jpaa{{\it J. Pure Appl. Algebra}}
\def\jems{{\it J. Eur. Math. Soc. (JEMS)}}
\def\jtokyo{{\it J. Fac. Sci. Univ. Tokyo}}
\def\cjm{{\it Canad. J. Math.}}
\def\jaums{{\it J. Australian Math. Soc.}}
\def\pspm{{\it Proc. Symp. Pure. Math.}}
\def\ast{{\it Ast\'eriques}}
\def\pamq{{\it Pure Appl. Math. Q.}}
\def\nagoya{{\it Nagoya Math. J.}}
\def\forum{{\it Forum Math. }}
\def\tjm{{\it Taiwanese J. Math.}}
\def\rt{{\it Represent. Theory}}
\def\bordeaux{{\it J. Th\'eor. Nombres Bordeaux}}
\def\ijnt{{\it Int. J. Number Theory}}
\def\jmsj{{\it J. Math. Soc. Japan}}
\def\aa{{\it Acta Artih.}}


\def\tp{{to appear}}

\newcommand{\princeton}[1]{Ann. Math. Studies #1, Princeton
  Univ. Press}

\newcommand{\LNM}[1]{Lecture Notes in Math., vol. #1, Springer-Verlag}

\bibitem{artin-winters} M. Artin and G. Winters, 
Degenerate fibres and stable reduction of curves.
{\it Topology}~{\bf 10} (1971), 373--383. 

\bibitem{deligne:corvallis} P. Deligne, Vari\'et\'es de Shimura:
  interpr\'etation modulaire, et techniques de construction de
  mod\`eles canoniques. {\it Automorphic forms, representations and
  $L$-functions (Proc. Sympos.Pure Math., Oregon State
  Univ. Corvallis, 1977), Part 2,} 247--289. Proc. Sympos. Pure Math.,
  XXXIII, {\it Amer. Math. Soc.,} 1979.

\bibitem{deligne-pappas} P. Deligne and G. Pappas, Singularit\'es des
  espaces de modules de Hilbert, en les caract\'eristiques divisant le
  discriminant. \compos~{\bf 90} (1994), 59--79.


\bibitem{faltings:end} G.~Faltings, Endlichkeitss\"atze f\"ur abelsche
  Variet\"aten \"uber Zahlk\"orpern. \invent~{\bf 73} (1983), 349--366.

\bibitem{goertz:symplectic}  U. G\"ortz, On the flatness of local
  models for the symplectic group. \advmath~{\bf 176} (2003),
  89--115. 



\bibitem{goertz:topological}  U. G\"ortz, Topological flatness of
  local models in the ramified case. \mathz~{\bf 250} (2005), 775--790. 


\bibitem{grothendieck:bt} A. Grothendieck, {\it Groupes de
    Barsotti-Tate et Cristaux de Dieudonn\'e}. Les Presses de
    l'Universit\'e de Montr\'eal, 1974.

\bibitem{sga7-1}
{\it Groupes de monodromie en g\'eometrie alg\'ebrique
  I.} S\'eminaire de G\'eometrie Alg\'ebrique du Bois-Marie 1967--1969 
 (SGA 7 I). Dirig\'e par A. Grothendieck.Avec la collaboration de
 M. Raynaud et D. S. Rim. \LNM{288}. 1972. 

\bibitem{hartwig:kr}
P.~Hartwig,  Kottwitz-Rapoport and p-rank strata in the reduction of
Shimura varieties of PEL type. \grenoble~{\bf 65}, no. 3, 1031--1103.


\bibitem{he-rapoport}
X.~He and M.~Rapoport, 
Stratifications in the reduction of Shimura varieties.
{\it Manuscripta Math.}~{\bf 152} (2017), no. 3-4, 317--343. 




\bibitem{helgason:gsm34} S.~Helgason, 
{\it Differential geometry, Lie groups, and symmetric
spaces.}  Graduate Studies in Mathematics, 34. AMS, 2001. 641 pp.  

\bibitem{katsura-oort:surface} T. Katsura and F. Oort, Families of
  supersingular abelian surfaces, \compos~{\bf 62} (1987), 107--167.

\bibitem{koblitz:thesis} N. Koblitz, $p$-adic variant of the
  zeta-function of families of varieties defined over finite fields. 
  \compos~{\bf 31} (1975), 119--218.

\bibitem{kottwitz:isocrystals} R. E. Kottwitz, Isocrystals with
  additional structure. \compos~{\bf 56} (1985), 201--220.

\bibitem{kottwitz:jams92} R. E.~Kottwitz, Points on some Shimura
  varieties over finite fields. \jams~{\bf 5} (1992), 373--444.

\bibitem{kottwitz:isocrystals2} R. E. Kottwitz, Isocrystals with
  additional structure. II. \compos~{\bf 109} (1997), 255--339.

\bibitem{manin:thesis} Yu. Manin, Theory of commutative formal groups
  over fields of finite characteristic. 
  {\it Russian Math. Surveys}~{\bf 18} (1963), 1--80.

\bibitem{moonen:bt1} B. Moonen,  Group schemes with additional
    structures and Weyl group cosets. {\it Moduli of Abelian Varieties},
  255--298.  (ed. by C. Faber, G. van der Geer and F. Oort), 
  {Progress in Mathematics}~{\bf 195}, Birkh\"auser 2001.

\bibitem{moonen:eo} B. Moonen, A dimension formula for Ekedahl-Oort
  strata. \grenoble~{\bf 54} (2004), 666--698.



\bibitem{moonen:st} B. Moonen, Serre-Tate Theory for Moduli Spaces of
  PEL Type. \ecole4~{\bf 37} (2004), 223--269.



\bibitem{moret-bailly:p1} L. Moret-Bailly, Familles de courbes et de
   vari\'et\'es ab\'eliennes sur ${\mathbf P}^1$. S\'em. sur les
   piur ${\mathbf P}^1$. S\'em. sur les
   pinceaux de courbes de genre au moins deux
   (ed. L. Szpiro). \ast~{\bf 86} (1981), 109--140.

\bibitem{moret-bailly:ast85} L.~Moret-Bailly, {\it Pinceaux de
  vari\'et\'es ab\'eliennes.}  {\ast}~{\bf 129}  (1985), 266 pp.  

\bibitem{messing:bt} W. Messing, {\it The crystals associated to
  Barsotti-Tate groups: with applications to abelian schemes.} 
  Lecture Notes in Math. 264, Springer-Verlag, 1972. 

\bibitem{mumford:av} D. Mumford, {\it Abelian Varieties.} Oxford
  University Press, 1974.

\bibitem{norman-oort} P.~Norman and F.~Oort, Moduli of abelian varieties,
  \ann~{\bf 112} (1980), 413--439. 

\bibitem{omeara:book} O.~T.~O'Meara, {\it Introduction to quadratic
  forms}. Reprint of the 1973 edition. 
  Classics in Mathematics. Springer-Verlag, Berlin, 2000. xiv+342 pp. 

\bibitem{oort:cm} F. Oort, The isogeny class of a CM-type abelian
  variety is defined over a finite extension of the prime
  field. \jpaa~{\bf 3} (1973), 399--408.


\bibitem{oort:oslo} F.~Oort,  Finite group schemes, local moduli for
  abelian varieties, and lifting problems. \compos~{\bf 23} (1971),
  265--296. 

\bibitem{pappas:rz} G.~ Pappas, On the arithmetic moduli
schemes of PEL Shimura varieties. \jag~{\bf 9} (2000), 577--605.

\bibitem{pappas-rapoport} G. Pappas and M. Rapoport, Local models in
  the ramified case I. The EL-case. \jag~{\bf 12} (2003), 107--145. 



\bibitem{pappas-rapoport:2} G. Pappas and M. Rapoport, Local models in
  the ramified case II. Splitting models. 
  \duke~{\bf 127}  (2005), 193--250. 


\bibitem{pappas-rapoport:3} G. Pappas and M. Rapoport, Local models in
  the ramified case. III. Unitary groups. \jussieu~{\bf 8}
  (2009), no. 3, 507--564. 


\bibitem{pappas-zhu}
G. Pappas and X. Zhu,  Local models of Shimura varieties and a
conjecture of Kottwitz. {\it Invent. Math.}~{\bf 194} (2013), no. 1, 
147--254. 


\bibitem{rapoport-richartz} M. Rapoport and M. Richartz, On the
  classification and specialization of $F$-isocrystals with additional
  structure. \compos~{\bf 103} (1996), 153--181. 

\bibitem{rapoport-viehmann} M. Rapoport and E. Viehmann, Towards
  a theory of local Shimura varieties. {\it M\"unster J. Math.}~{\bf
  7} (2014), no. 1, 273--326. 

\bibitem{rapoport-zink} M. Rapoport and Th. Zink, {\it Period Spaces
    for $p$-divisible groups}. \princeton{141}, 1996.


\bibitem{smithling:admD} B. D. Smithling, 
Admissibility and permissibility for minuscule cocharacters in
orthogonal groups. \manmath~{\bf 136} (2011), no. 3-4, 295--314.  



\bibitem{smithling:tfD} B. D. Smithling,
Topological flatness of orthogonal local models in the split, 
even case. I \mathann~{\bf 350} (2011) no. 2, 381--416.



\bibitem{tsukamoto:1961} T.~Tsukamoto, On the local theory of
  quaternionic anti-hermitian forms. \jmsj~{\bf 13} (1961)
  387--400. 

  
\bibitem{vollaard:thesis} I.~Vollaard, On the Hilbert-Blumenthal moduli 
problem.  {\it J. Inst. Math. Jussieu}~{\bf  4}  (2005), 653--683.    


\bibitem{wedhorn:ordinary} T. Wedhorn, Ordinariness in good reductions
  of Shimura varieties of PEL-type. \ecole4~{\bf 32} (1999), 575--618.

\bibitem{wedhorn:eo} T. Wedhorn, The dimension of Oort strata of
  Shimura varieties of PEL-type. {\it Moduli of Abelian Varieties},
  441--471. (ed. by C. Faber, G. van  der Geer and F. Oort), 
  {Progress in Mathematics}~{\bf 195}, Birkh\"auser 2001.

\bibitem{vasiu:D1} 
A.~Vasiu,  Integral models in unramified mixed characteristic
 (0,2) of Hermitian orthogonal Shimura varieties of PEL type, Part
 I. {\it J. Ramanujan Math. Soc.}~{\bf 27} (2012), no. 4, 425--477. 


\bibitem{vasiu:D2} A.~Vasiu,  
Integral models in unramified mixed characteristic (0,2) of hermitian
orthogonal Shimura varieties of PEL type, Part II. {\it Math. Nachr.} 
~{\bf 287} (2014), no. 14-15, 1756--1773. 

\bibitem{yu:lift} C.-F. Yu, Lifting abelian varieties with additional
  structures. \mathz~{\bf 242} (2002), 427--441. 

\bibitem{yu:reduction} C.-F. Yu, On reduction of Hilbert-Blumenthal
  varieties. \grenoble~{\bf 53} (2003), 2105--2154.

\bibitem{yu:cm} C.-F. Yu, The isomorphism classes of abelian varieties
  of CM-type. \jpaa~{\bf 187} (2004) 305--319.

\bibitem{yu:mass_hb} C.-F.~Yu, On the mass formula of supersingular
  abelian varieties with real multiplications. \jaums~{\bf 78} (2005),
  373--392.

\bibitem{yu:gmf} C.-F. Yu, An exact geometric mass
  formula. \imrn~{\bf 2008}, Article ID rnn113, 11 pages. 


\bibitem{yu:smf} C.-F. Yu, Simple mass formulas on Shimura varieties
  of PEL-type. \forum~{\bf 22} (2010), no. 3, 565--582.



\bibitem{yu:endo} C.-F. Yu, On finiteness of endomorphism rings of
  abelian varieties. \mrl~{\bf 17} (2010), no. 2, 357--370.



\bibitem{yu:embed} C.-F. Yu,  Embeddings of fields into simple
  algebras: generalizations and applications. \ja~{\bf 368} (2012),
  1--20. 




\bibitem{zarhin:end}
J. G.~Zarhin, Isogenies of abelian varieties over fields of finite
characteristics, {\it Math. USSR Sbornik}~{\bf 24} (1974), 451--461.

\bibitem{zink:cartier} Th. Zink, {\it Cartiertheorie kommutativer
  formaler Gruppen. } Teubner-Texte Math. Teubner, Leipzig, 1984.

\end{thebibliography}
\end{document}